\documentclass[12pt]{article}

\usepackage[utf8]{inputenc}
\usepackage{amsmath, amssymb, amsthm}
\usepackage{mathtools}
\usepackage{mathrsfs}
\usepackage{enumitem}
\usepackage{microtype}
\usepackage[super,sort&compress]{natbib}
\usepackage{graphicx}
\usepackage{subcaption}                    
\usepackage{float, booktabs}
\usepackage{algorithm, algorithmic}

\usepackage{listings}
\usepackage{xcolor}
\usepackage[colorlinks=true,linkcolor=blue,citecolor=blue]{hyperref}
\theoremstyle{plain}
\newtheorem{theorem}{Theorem}[section]
\newtheorem{lemma}[theorem]{Lemma}

\newtheorem{corollary}[theorem]{Corollary}
\newtheorem{remark}[theorem]{Remark}

\theoremstyle{definition}

\newtheorem{assumption}[theorem]{Assumption}

\numberwithin{equation}{section}

\allowdisplaybreaks[4]
\newcommand{\keywords}[1]{\textbf{Keywords:} #1}

\newcommand{\R}{\mathbb{R}}
\newcommand{\E}{\mathbb{E}}
\newcommand{\PP}{\mathcal{P}}
\newcommand{\tr}{\operatorname{tr}}
\newcommand{\dd}{\mathrm{d}}
\newcommand{\norm}[1]{\left\lVert #1\right\rVert}
\newcommand{\abs}[1]{\left\lvert #1\right\rvert}
\newcommand{\bb}[1]{\mathbb{#1}}
\newcommand{\Law}{\operatorname{Law}}

\begin{document}

\title{Convergence of Densities for the Euler--Maruyama Approximation of\\
Brownian-Driven McKean--Vlasov Equations}

\author{
	{  Zhenghan Yin, Xu Sun\thanks{Corresponding author}}\hspace{2mm}
	\\
	{\itshape\small  School of Mathematics and Statistics},\\
	{\itshape\small Huazhong University of Science and Technology},
	{\itshape\small Wuhan 430074,  China}\\
	{\itshape\small E-mail: d202380025@hust.edu.cn,
		xsun@hust.edu.cn}\vspace{1mm}
}
\date{\today}

\maketitle

\begin{abstract}
Convergence of the transition density of the Euler--Maruyama scheme for a
Brownian-driven McKean--Vlasov stochastic differential equation is established,
the coefficients being uniformly elliptic, spatially smooth, and Lipschitz with
respect to the law in the $2$-Wasserstein distance. The density $p_n$ of the
$n$-step scheme is shown to converge to the density $p$ of the McKean--Vlasov
equation in a Gaussian-weighted supremum norm at rate $O(n^{-1/2})$, that is, of
order one half in the step size, and hence also uniformly in space. The proof
develops a discrete parametrix (L\'evi--Hadamard) expansion for the Euler chain
and compares it with the continuous parametrix of the limiting equation along a
hierarchy of measure-flow, time-discretisation and lattice corrections; the
discrepancy between the discrete and the continuous measure flows is of order
$O(n^{-1/2})$ in the $2$-Wasserstein distance and does not degrade the final rate.
\end{abstract}


\section{Introduction}
\label{sec:intro}

McKean--Vlasov stochastic differential equations (SDEs), whose coefficients depend on the law
of the solution itself, arise as the mean-field limit of interacting particle systems and appear
across kinetic theory (the Boltzmann and Vlasov equations \cite{Kac1956}), neuroscience
\cite{Baladron2012}, and mean-field games
\cite{LasryLions2007,HuangMalhameCaines2006}; this program originates with McKean
\cite{McKean1966} and Sznitman's propagation-of-chaos theory \cite{Sznitman1991}. Both the
analysis and the Euler--Maruyama approximation of McKean--Vlasov SDEs are therefore of central
importance.

The well-posedness of McKean--Vlasov equations under Lipschitz coefficients is classical
\cite{McKean1966,Sznitman1991}, with recent extensions to singular and non-Lipschitz
coefficients \cite{HuangWang2019,DingQiao2021}. For classical SDEs, the transition density
exists under uniform ellipticity and is smooth when the coefficients are smooth, or more
generally under H\"ormander's condition
\cite{Aronson1967,Friedman1964,Hormander1967}, and the convergence of the Euler--Maruyama
densities was obtained by Bally--Talay \cite{BallyTalay1996a,BallyTalay1996b} and
Konakov--Mammen \cite{KonakovMammen2000}. For McKean--Vlasov equations, Crisan--McMurray
\cite{CrisanMcMurray2018} established the smoothness of the marginal densities under uniform ellipticity and, more generally, under a
H\"ormander-type condition; the density convergence of the Euler scheme for McKean--Vlasov
equations appears not to have been addressed so far, and this is the
question studied in the present paper.

We study the convergence of the densities of the Euler--Maruyama scheme for a Brownian-driven
McKean--Vlasov equation. Let $T>0$ and let $W=(W_t)_{t\in[0,T]}$ be
a standard $d$-dimensional Brownian motion. We consider the McKean--Vlasov SDE
\begin{equation}
\label{eq:SDE}
X_t=x_0+\int_0^t b(s,X_s,\mu_s)\,\dd s+\int_0^t \sigma(s,X_s,\mu_s)\,\dd W_s,\qquad
\mu_s=\Law(X_s),
\end{equation}
where $x_0\in\R^d$ is deterministic,
$b:[0,T]\times\R^d\times\PP_2(\R^d)\to\R^d$ and
$\sigma:[0,T]\times\R^d\times\PP_2(\R^d)\to\R^{d\times d}$ are measurable, and
$\PP_2(\R^d)$ denotes the space of probability measures on $\R^d$ with finite second moment,
equipped with the $2$-Wasserstein distance $\mathcal{W}_2$. We write
\begin{equation}
\Sigma(t,x,\mu):=\sigma(t,x,\mu)\sigma(t,x,\mu)^{\top}\in\R^{d\times d}.
\end{equation}
Since the law of the solution of \eqref{eq:SDE} depends on $\sigma$ only through
$\Sigma=\sigma\sigma^{\top}$, there is no loss of generality in taking $\sigma$ to be the symmetric
square root $\Sigma^{1/2}$ of $\Sigma$; all regularity conditions are imposed on $\Sigma$ and $b$
(see Section~\ref{sec:setting}).
Since time can be rescaled, we take $T=1$ without loss of generality.

We consider the Euler--Maruyama scheme with uniform step size $h=1/n$ and grid points $t_k=kh$,
$0\le k\le n$:
\begin{equation}
\label{eq:Euler}
X^{(n)}_{0}=x_0,\qquad
X^{(n)}_{k+1}=X^{(n)}_k+h\,b\big(t_k,X^{(n)}_k,\mu^{(n)}_k\big)
+\sqrt{h}\,\sigma\big(t_k,X^{(n)}_k,\mu^{(n)}_k\big)\,\xi_{k+1},
\end{equation}
where $(\xi_k)_{k\ge1}$ is a sequence of independent standard Gaussian vectors in $\R^d$, and
$\mu^{(n)}_k=\Law(X^{(n)}_k)$. Since the $\xi_{k+1}$ are standard Gaussian and independent of the
past, each step of the scheme is Gaussian conditionally on the past. We prove that, for every $n$, the law $\mu^{(n)}_n$ of the
terminal value $X^{(n)}_n$ admits a density $p_n(x_0,\cdot)$, and that
\begin{equation}
\label{eq:main}
\sup_{y\in\R^d}\exp\Big(\tfrac1c\,\norm{y-x_0}^{2}\Big)\,\abs{p_n(x_0,y)-p(x_0,y)}
=O\big(n^{-1/2}\big),
\end{equation}
for some $c>0$, where $p(x_0,\cdot)$ is the density of the terminal value $X_1$ of the
McKean--Vlasov solution. A direct consequence is that the density produced by the Euler scheme converges uniformly in space to the density of the solution of the original equation \eqref{eq:SDE}, with order one half in the step size.

Our proof builds on the parametrix method of Konakov--Mammen \cite{KonakovMammen2000}: freezing
the coefficients at a terminal point and at the measure flow turns the Euler chain into a sum of
independent Gaussian variables, whose law is therefore exactly Gaussian. The
genuinely new difficulty, absent in \cite{KonakovMammen2000}, is that the discrete chain
\eqref{eq:Euler} uses the discrete measure flow $\mu^{(n)}$ while the limit equation uses the
continuous flow $\mu$. We handle this discrepancy through the
Lipschitz dependence of the coefficients on the measure (see Section~\ref{sec:setting}) together
with the fact that the Euler measure flow converges to the McKean--Vlasov flow at rate
$O(n^{-1/2})$ in the $2$-Wasserstein distance (a special case of Ding--Qiao
\cite[Theorem 4.1]{DingQiao2021}; see also Appendix~\ref{app:flowrate} for a self-contained proof).

Let us outline the structure of the paper. In Section~\ref{sec:setting} we state our assumptions
and the main result. Section~\ref{sec:prelim} collects preliminary estimates: moment bounds, the
convergence rate of the measure flow, and the Gaussian-kernel toolbox. Section~\ref{sec:parametrix}
develops the parametrix expansions for both the continuous and the discrete problems.
Section~\ref{sec:comparison} compares the two expansions and isolates the three sources of
error (measure flow, time discretisation and lattice), and
Section~\ref{sec:completion} assembles these estimates to complete the proof of the main theorem.
The proofs of the auxiliary estimates used along the way are collected in
Appendices~\ref{app:moments}, \ref{app:flowrate} and \ref{app:kernels}.

\section{Setting, Assumptions and Main Result}
\label{sec:setting}

Throughout the paper $C$ denotes a positive constant whose value may change from line to
line; dependence on a parameter is indicated by a subscript, as in $C_\alpha$ or $C_\beta$. We write
$\norm{f}_\infty:=\sup_{t,x,\mu}\norm{f(t,x,\mu)}$ for the uniform sup-norm of a
(vector- or matrix-valued) coefficient function $f$. Further
notation is introduced where it is first needed. For a multi-index
$\alpha=(\alpha_1,\dots,\alpha_d)\in\bb{N}_0^d$ we write $\abs{\alpha}=\sum_i\alpha_i$ and
$\partial_x^\alpha=\partial_{x_1}^{\alpha_1}\cdots\partial_{x_d}^{\alpha_d}$, so that $\abs{\alpha}=0$ corresponds to the
function itself, $\abs{\alpha}=1$ to its first-order derivatives, and $\abs{\alpha}=2$ to its
second-order derivatives $\partial^2_{x_ix_j}$. The constants $c_0,C_0,\Lambda$ are as in
Assumptions~\ref{ass:elliptic} and \ref{ass:measLip} below.

The symbol $\rho$ always denotes the square root of a time-interval length: for $0\le s<t\le1$ we set
$\rho(s,t):=\sqrt{t-s}$, and for grid indices $0\le a<b\le n$ we abbreviate $\rho_{ab}:=\rho(t_a,t_b)$, so that
\begin{equation}
\label{eq:rhodef}
\rho(s,t)^2=t-s\quad(0\le s<t\le1),\qquad \rho_{ab}^2=t_b-t_a=\tfrac{b-a}{n}\quad(0\le a<b\le n).
\end{equation}
When the time interval is clear from the context we write simply $\rho$ for its square root, that is, $\rho=\rho(s,t)$ in a continuous expression (so that $\rho^2=t-s$) and $\rho=\rho_{ab}$ in a discrete one (so that $\rho^2=t_b-t_a$). Every occurrence of $\rho$ in this paper is of this form, except in the statement of the Gaussian-kernel toolbox (Lemma~\ref{lem:kernel}), where $\rho,\rho'$ denote arbitrary positive widths; in every application the widths are again square roots of time-interval lengths of this form. We also identify throughout the
terminal densities with the transition densities at $(s,t)=(0,1)$: $p(x_0,y):=p(0,1,x_0,y)$
and $p_n(x_0,y):=p_n(0,n;x_0,y)$.

A decisive observation, which justifies the whole approach, is that the McKean--Vlasov equation
\eqref{eq:SDE} becomes a classical time-inhomogeneous diffusion once the measure flow is known:
for a fixed initial point, the coefficients $(t,x)\mapsto b(t,x,\mu_t)$ and
$(t,x)\mapsto\Sigma(t,x,\mu_t)$ are deterministic functions of $(t,x)$, so that $X$ is a
time-inhomogeneous Markov diffusion; likewise the Euler scheme \eqref{eq:Euler} is a
time-inhomogeneous Markov chain with coefficients $b(t_k,x,\mu^{(n)}_k)$ and
$\Sigma(t_k,x,\mu^{(n)}_k)$. We refer to this as the \emph{freezing of the measure flow}; it is the
key device behind the parametrix expansions of
Sections~\ref{sec:parametrix}--\ref{sec:comparison}.

Throughout we write $\sigma=\Sigma^{1/2}$ for the unique symmetric positive definite square root
of $\Sigma$, which is well defined by the uniform ellipticity condition
(Assumption~\ref{ass:elliptic} below). All regularity assumptions below are imposed on $\Sigma$
(and $b$), not on $\sigma$: the parametrix method involves only $\Sigma$, through the covariance
of the frozen Gaussian densities and the generator, so that $\sigma$ never enters the density
computation. Since the map $A\mapsto A^{1/2}$ is real analytic (hence $\mathrm{C}^1$, and therefore
Lipschitz) on a neighborhood of the compact set $\{A:c_0 I\preceq A\preceq C_0 I\}$, the
boundedness and Lipschitz regularity of $\Sigma$ in $(x,\mu)$ assumed in
Assumptions~\ref{ass:spatial} and \ref{ass:measLip} transfer to $\sigma=\Sigma^{1/2}$. This
transfer is used in the well-posedness argument below, which requires $\sigma$ to be Lipschitz in
$(x,\mu)$, and apart from that only in the proof of Lemma~\ref{lem:flowrate}, where It\^o's formula applied to the strong-coupling error again requires $\sigma$ to be bounded and Lipschitz in $(x,\mu)$.

\subsection{Assumptions}
\label{sec:assumptions}

We impose the following conditions.

\begin{assumption}[Spatial smoothness and boundedness]
\label{ass:spatial}
The functions $b$ and $\Sigma$ are bounded, uniformly in $(t,x,\mu)$, and are Lipschitz
continuous in $x$, uniformly in $(t,\mu)$. Their first-order derivatives $\partial_{x_i}b$
and $\partial_{x_i}\Sigma$ are bounded, uniformly in $(t,x,\mu)$. Furthermore, the
second-order derivatives $\partial^2_{x_ix_j}\Sigma$ exist and are bounded, uniformly in
$(t,x,\mu)$.
\end{assumption}

\begin{assumption}[Time regularity]
\label{ass:time}
The partial derivatives $\partial_t b$ and $\partial_t\Sigma$, as well as the mixed derivatives
$\partial_t\partial_x b$ and $\partial_t\partial_x\Sigma$, exist and are bounded, uniformly in
$(t,x,\mu)$.
\end{assumption}

\begin{assumption}[Uniform ellipticity]
\label{ass:elliptic}
There exist constants $0<c_0\le C_0<\infty$ such that, for all $t\in[0,1]$, $x\in\R^d$,
$\mu\in\PP_2(\R^d)$, and $\theta\in\R^d$,
\begin{equation}
c_0\,\norm{\theta}^2\ \le\ \theta^{\top}\Sigma(t,x,\mu)\,\theta\ \le\ C_0\,\norm{\theta}^2.
\end{equation}
\end{assumption}

\begin{assumption}[Measure Lipschitz continuity]
\label{ass:measLip}
There exists $\Lambda<\infty$ such that, for all $t\in[0,1]$, $x\in\R^d$, $\mu,\nu\in\PP_2(\R^d)$, and
every multi-index $\alpha$ with $\abs{\alpha}\le1$,
\begin{equation}
\norm{\partial_x^\alpha b(t,x,\mu)-\partial_x^\alpha b(t,x,\nu)}
+\norm{\partial_x^\alpha\Sigma(t,x,\mu)-\partial_x^\alpha\Sigma(t,x,\nu)}
\le \Lambda\,\mathcal{W}_2(\mu,\nu),
\end{equation}
and, with $\partial_t b$ and $\partial_t\Sigma$ in place of $\partial_x^\alpha b$ and
$\partial_x^\alpha\Sigma$,
\begin{equation}
\norm{\partial_t b(t,x,\mu)-\partial_t b(t,x,\nu)}
+\norm{\partial_t\Sigma(t,x,\mu)-\partial_t\Sigma(t,x,\nu)}
\le \Lambda\,\mathcal{W}_2(\mu,\nu).
\end{equation}
\end{assumption}

Since Assumptions~\ref{ass:spatial}, \ref{ass:elliptic} and \ref{ass:measLip} imply that $b$,
$\Sigma$, and hence $\sigma=\Sigma^{1/2}$, are bounded and Lipschitz in $(x,\mu)$, \eqref{eq:SDE}
admits a unique strong solution (cf.\ \cite{Sznitman1991}); moreover, no separate assumption on the
existence of a density is needed: existence of the density of $X_1$ is provided by the classical
parabolic theory \cite{Friedman1964,Aronson1967} under these Lipschitz and uniform ellipticity conditions, while the density of
$X^{(n)}_n$ is produced constructively by the discrete parametrix of
Section~\ref{sec:parametrix}; the parametrix expansions then also yield the bounds used below (see
Remark~\ref{rem:exist} for an alternative, self-contained derivation of the existence of $p$ that
does not appeal to the classical theory).

\begin{remark}
\label{rem:autotime}
Assumption~\ref{ass:time} requires only boundedness of $\partial_t b$, $\partial_t\Sigma$ and of the
mixed derivatives $\partial_t\partial_x b$, $\partial_t\partial_x\Sigma$; no higher time-regularity
is imposed. Indeed, combined with Assumption~\ref{ass:measLip} and the fact that the measure flow
$t\mapsto\mu_t$ is $\tfrac12$-H\"older continuous in $\mathcal{W}_2$ (a consequence of the increment estimate
of Lemma~\ref{lem:moments}), this implies that the frozen coefficients
$t\mapsto b(t,x,\mu_t)$ and $t\mapsto\Sigma(t,x,\mu_t)$ are $\tfrac12$-H\"older continuous in $t$,
uniformly in $x$; the same holds with the discrete flow $\mu^{(n)}$ in place of $\mu$, uniformly in
$n$ (by the third estimate of Lemma~\ref{lem:moments}). Writing $\rho=\rho(s,t)$ for the square root of the length of the time interval (see
\eqref{eq:rhodef}), this is precisely the regularity used
by the parametrix method: it yields a left-endpoint Riemann-sum error of order $O(\rho^2 n^{-1/2})$,
of the same order as the measure-flow error, so that the final rate $O(n^{-1/2})$ is not degraded.
\end{remark}

\subsection{Main result}
\label{sec:main}

\begin{theorem}
\label{thm:main}
Assume Assumptions~\ref{ass:spatial}--\ref{ass:measLip}. Then, for every $n\ge1$, the random
variable $X^{(n)}_n$ defined by the Euler scheme \eqref{eq:Euler} admits a smooth density
$p_n(x_0,\cdot)$, and the McKean--Vlasov solution $X_1$ of \eqref{eq:SDE} admits a density
$p(x_0,\cdot)$. Moreover, there exists a constant $c>0$, depending only on the coefficients, and a
constant $C<\infty$, independent of $n$, such that
\begin{equation}
\label{eq:maintheorem}
\sup_{y\in\R^d}\exp\Big(\tfrac1c\,\norm{y-x_0}^{2}\Big)\,\abs{p_n(x_0,y)-p(x_0,y)}
\ \le\ C\,n^{-1/2}.
\end{equation}
In particular, since the exponential weight in \eqref{eq:maintheorem} is bounded below by $1$, the
density $p_n(x_0,\cdot)$ computed by the Euler scheme \eqref{eq:Euler} converges to the density
$p(x_0,\cdot)$ of the original equation \eqref{eq:SDE} uniformly in $y\in\R^d$, at rate
$n^{-1/2}$ (equivalently, of order one half in the step size $h=1/n$):
\begin{equation}
\label{eq:uniformpointwise}
\sup_{y\in\R^d}\abs{p_n(x_0,y)-p(x_0,y)}
\ \le\ C\,n^{-1/2}.
\end{equation}
\end{theorem}

\begin{remark}[Existence of the density revisited]
\label{rem:exist}
The classical parabolic theory invoked above is not the only route to the
existence of $p$. Indeed, the only place where $p$ enters the proof of
Theorem~\ref{thm:main} is the identification $p=\sum_{r\ge0}(\tilde p*H^{*r})$ of
Lemma~\ref{lem:param-cont}; all the comparison lemmas of
Sections~\ref{sec:parametrix}--\ref{sec:comparison} and all the kernel bounds are
stated and proved in terms of the two measure flows $\mu$ and $\mu^{(n)}$ alone.
One may therefore define
$S(0,1,x_0,\cdot):=\sum_{r\ge0}(\tilde p*H^{*r})(0,1,x_0,\cdot)$ directly from the
flow $\mu$, without presupposing $p$. The estimates of
Section~\ref{sec:completion} then read
$\abs{p_n(0,n;x_0,y)-S(0,1,x_0,y)}\le Cn^{-1/2}\varphi_{c,1}(y-x_0)$, so that
the discrete densities $p_n$ form a Cauchy sequence in the weighted norm of
\eqref{eq:maintheorem}. Since $\mathcal{W}_2(\mu^{(n)}_n,\mu_1)\le Cn^{-1/2}$ by
Lemma~\ref{lem:flowrate}, and since $\mu^{(n)}_n$ (the law of $X^{(n)}_n$, that is, the discrete flow at time $1$) has density $p_n$, passing to
the limit (justified by that weighted bound) in
$\int_{\R^d}\varphi\,\dd\mu^{(n)}_n=\int_{\R^d}\varphi\,p_n$ for
$\varphi\in\mathrm C_b(\R^d)$ identifies $S$ as the density of $X_1$. Thus the
existence of $p$ can also be obtained without presupposing it; in this reading
\eqref{eq:param-cont} is recovered \emph{a posteriori} as the identity $S=p$.
We nevertheless keep the classical route as the primary one, being shorter and
not relying on the argument above.
\end{remark}

The remainder of the paper is devoted to the proof of Theorem~\ref{thm:main}.

\section{Preliminary Estimates}
\label{sec:prelim}

\subsection{Moment bounds and convergence of the measure flow}
\label{sec:moments}

For the measure-flow comparison below we use the continuous-time interpolation of the Euler
scheme. Write $s_n(t)=\lfloor nt\rfloor/n$ for the largest grid point not exceeding $t\in[0,1]$, and let
$\bar X^{(n)}$ solve
\begin{equation}
\label{eq:cont-euler}
\dd\bar X^{(n)}_t=b\big(s_n(t),\bar X^{(n)}_{s_n(t)},\mu^{(n)}_{s_n(t)}\big)\dd t
+\sigma\big(s_n(t),\bar X^{(n)}_{s_n(t)},\mu^{(n)}_{s_n(t)}\big)\dd W_t,
\qquad\bar X^{(n)}_0=x_0,
\end{equation}
where $\mu^{(n)}_{s_n(t)}$ is the law of the Euler scheme $X^{(n)}_{s_n(t)}$ at the grid point
$s_n(t)$. We write $\mu^{(n)}_t=\Law(\bar X^{(n)}_t)$, so that $\mu^{(n)}_{t_k}=\mu^{(n)}_k$ at
every grid point $t_k=kh$.

\begin{lemma}
\label{lem:moments}
Under Assumptions~\ref{ass:spatial} and \ref{ass:measLip}, for every $\kappa\ge1$ there exists a
constant $C_\kappa<\infty$, independent of $n$, such that
\begin{equation}
\sup_{t\in[0,1]}\E\norm{X_t}^{2\kappa}\le C_\kappa,\qquad
\sup_{0\le k\le n}\E\norm{X^{(n)}_k}^{2\kappa}\le C_\kappa,
\end{equation}
and
\begin{equation}
\E\norm{X_t-X_s}^{2\kappa}\le C_\kappa\,\abs{t-s}^{\kappa},\qquad 0\le s<t\le1,
\end{equation}
and, for the interpolated Euler process \eqref{eq:cont-euler},
\begin{equation}
\E\norm{\bar X^{(n)}_t-\bar X^{(n)}_s}^{2\kappa}\le C_\kappa\,\abs{t-s}^{\kappa},\qquad 0\le s<t\le1 .
\end{equation}
\end{lemma}

\begin{proof}
This is \cite[Lemma 3.4]{DingQiao2021}; a self-contained proof is deferred to
Appendix~\ref{app:moments}.
\end{proof}

\begin{lemma}
\label{lem:flowrate}
Under Assumptions~\ref{ass:spatial}--\ref{ass:measLip}, there exists a constant $C<\infty$,
independent of $n$, such that
\begin{equation}
\label{eq:flowrate}
\sup_{t\in[0,1]}\mathcal{W}_2\big(\mu^{(n)}_t,\mu_t\big)\ \le\ C\,n^{-1/2}.
\end{equation}
\end{lemma}

\begin{proof}
This is the Lipschitz special case of \cite[Theorem 4.1]{DingQiao2021}; a self-contained proof is
deferred to Appendix~\ref{app:flowrate}.
\end{proof}

\begin{remark}
Lemma~\ref{lem:flowrate} is the place where the measure flow is controlled. The convergence rate
$O(n^{-1/2})$ in \eqref{eq:flowrate} will be the bottleneck in the final density
rate.
\end{remark}

\subsection{Gaussian kernels}
\label{sec:kernels}

For $c>0$ and a width $\rho>0$ (always the square root of a time-interval length in the sense of \eqref{eq:rhodef}, so that in each application $\rho=\rho(s,t)$ or $\rho=\rho_{ab}$) define the Gaussian weights
$\varphi_c(u)=\exp\big(-c^{-1}\norm{u}^2\big)$ and
$\varphi_{c,\rho}(u)=\rho^{-d}\varphi_c(u/\rho)=\rho^{-d}\exp\big(-\norm{u}^2/(c\rho^2)\big)$,
$u\in\R^d$. We denote by $\mathcal N(m,\Theta)$ the $d$-dimensional Gaussian law with mean $m$ and
covariance $\Theta$, and by $g(m,\Theta;z)$ its density. The following elementary Gaussian
estimates are standard; see, e.g., \cite[Lemmas 3.10--3.11]{KonakovMammen2000} (see also
\cite{BallyTalay1996a,BallyTalay1996b}).

\begin{lemma}
\label{lem:kernel}
With $C_0$ as in Assumption~\ref{ass:elliptic}, and with $C$ (resp.\ $C_\alpha$) denoting a finite
constant depending only on the displayed parameters and on the dimension $d$, the following hold.
\begin{enumerate}[label=(\roman*),leftmargin=2.2em]
\item For all $c,c'>0$ and $\rho,\rho'>0$,
\begin{equation}
\int_{\R^d}\varphi_{c,\rho}(y-x)\,\varphi_{c',\rho'}(x-z)\,\dd x
\le\pi^{d/2}\max(c,c')^{d/2}\,\varphi_{c+c',\sqrt{\rho^2+\rho'^2}}(y-z).
\end{equation}
\item For all $c>0$ and $\rho>0$,
\begin{equation}
\norm{y-x}\,\varphi_{c,\rho}(y-x)\le C\,\rho\,\varphi_{c+1,\rho}(y-x).
\end{equation}
\item For all $\alpha,\beta\ge-1$,
\begin{equation}
\int_s^t (t-u)^{\alpha/2}(u-s)^{\beta/2}\dd u
=(t-s)^{(\alpha+\beta)/2+1}\,\mathrm{B}\big(\tfrac\alpha2+1,\tfrac\beta2+1\big),
\end{equation}
where $\mathrm{B}(\cdot,\cdot)$ denotes the Beta function.
\item \emph{(Gaussian derivative control)} Whenever $c\rho^2 I\preceq\Theta\preceq C_0\rho^2 I$, possibly enlarging $c$,
\begin{equation}
\abs{\partial^\alpha_z g(m,\Theta;z)}\le C_\alpha\,\rho^{-|\alpha|}\,\varphi_{c,\rho}(z-m).
\end{equation}
\item \emph{(Gaussian comparison)} For $c\rho^2 I\preceq\Theta,\Theta'\preceq C_0\rho^2 I$ and $\norm{m},\norm{m'}\le C\rho^2$,
possibly enlarging $c$,
\begin{equation}
\abs{\partial^\alpha_z g(m,\Theta;z)-\partial^\alpha_z g(m',\Theta';z)}
\le C_\alpha\Big[\norm{m-m'}\,\rho^{-1}+\norm{\Theta-\Theta'}\,\rho^{-2}\Big]
\rho^{-|\alpha|}\,\varphi_{c,\rho}(z).
\end{equation}
\end{enumerate}
\end{lemma}

The proof of Lemma~\ref{lem:kernel} is elementary and is deferred to
Appendix~\ref{app:kernels}.

\section{The Parametrix Expansions}
\label{sec:parametrix}

In this section we develop the parametrix expansion for the transition densities $p$ and $p_n$,
following McKean--Singer \cite{McKeanSinger1967} in the form used by Konakov--Mammen
\cite{KonakovMammen2000}: each density is written as a series of iterated convolutions of a frozen
Gaussian density with a kernel that measures the freezing error. The series converges absolutely
(by Lemma~\ref{lem:kernbound}); it is infinite in the continuous case and finite in the discrete
case (where the chain has only $n$ steps). The two expansions are compared in Section~\ref{sec:comparison}.

\subsection{Frozen Gaussian densities and the frozen processes}
\label{sec:frozen}

The parametrix method starts from an exact Gaussian model obtained from \eqref{eq:SDE} by freezing
the coefficients. For a deterministic measure flow $\nu=(\nu_t)_{t\in[0,1]}$ the flow enters only as a
deterministic parameter in the third argument of $b$ and $\Sigma$, never as a variable to be
differentiated. In the continuous problem we
freeze the flow at the true flow $\mu$, in the discrete problem at the Euler flow $\mu^{(n)}$; in
both cases we additionally freeze the spatial argument at the terminal point $y$.

Freezing the measure flow is exact. By Assumptions~\ref{ass:spatial}, \ref{ass:elliptic} and
\ref{ass:measLip} the coefficients $b$ and $\Sigma$ (hence $\sigma$) are bounded and Lipschitz in
$(x,\mu)$, so \eqref{eq:SDE} has a unique strong solution with a unique measure flow $\mu$
(cf.\ \cite{Sznitman1991}), and the Euler scheme \eqref{eq:Euler} is well defined as a recursion
with Gaussian increments. Consequently, once the measure argument is frozen at the true flow $\mu$,
\eqref{eq:SDE} reduces to a classical time-inhomogeneous (Markov) diffusion with coefficients
$(t,x)\mapsto b(t,x,\mu_t)$ and $(t,x)\mapsto\Sigma(t,x,\mu_t)$; by uniqueness this diffusion has law
$\mu_t$ at every $t$, and hence exactly the density $p$ of the original solution. Thus the
measure-freezing is not an approximation: it is the same object viewed as a classical diffusion. The
identification above holds along the true initial datum $(s,x)=(0,x_0)$; for other $(s,x)$ the object
$p(s,t,x,\cdot)$ expanded below is the transition density of the diffusion whose coefficients are frozen
at the true flow, and it is this transition density that the parametrix requires. The only genuine
freezing, which does change the equation, is the spatial freezing at the terminal point
$y$; its error is governed by the kernel $H$ defined below.

\textbf{The continuous frozen process and its density.} Freezing the coefficients of \eqref{eq:SDE}
at $(y,\mu_u)$ yields the affine Gaussian process
\begin{equation}
\label{eq:mvtil}
\dd\tilde Y_u=b(u,y,\mu_u)\dd u+\sigma(u,y,\mu_u)\dd W_u,\qquad \tilde Y_s=x,
\end{equation}
whose drift and diffusion coefficients are independent of the state variable and depend on the time $u$ only, through the frozen point $(y,\mu_u)$. Consequently the increment
$\tilde Y_t-\tilde Y_s=\int_s^t b(u,y,\mu_u)\dd u+\int_s^t\sigma(u,y,\mu_u)\dd W_u$ is a Gaussian
vector with mean
\begin{equation}
\bar m(s,t,y)=\int_s^t b(u,y,\mu_u)\dd u
\end{equation}
and covariance
\begin{equation}
\bar\Sigma(s,t,y)=\int_s^t\Sigma(u,y,\mu_u)\dd u,
\end{equation}
so that the transition density $\tilde p(s,t,x,\cdot)$ of $\Law(\tilde Y_t\mid\tilde Y_s=x)$ is exactly
Gaussian:
\begin{multline}
\label{eq:frozen-cont}
\tilde p(s,t,x,y)=(2\pi)^{-d/2}\det\bar\Sigma(s,t,y)^{-1/2}\\
\times\exp\Big[-\tfrac12\big(y-x-\bar m(s,t,y)\big)^{\top}\bar\Sigma(s,t,y)^{-1}
\big(y-x-\bar m(s,t,y)\big)\Big].
\end{multline}
By Assumption~\ref{ass:elliptic}, $\bar\Sigma(s,t,y)\succeq c_0(t-s)\,I$, so $\tilde p$ is nondegenerate
for $s<t$. This Gaussian density is the \emph{principal term} of the parametrix: it captures the
dominant local behaviour of the true transition density $p$, and the kernel $H$ defined below
measures exactly the error made by the freezing.

\textbf{The discrete frozen chain and its density.} Similarly, freezing the coefficients of the
Euler scheme \eqref{eq:Euler} at $(y,\mu^{(n)}_i)$ gives the frozen chain
\begin{equation}
\tilde X^{(n)}_{i+1}=\tilde X^{(n)}_i+h\,b(t_i,y,\mu^{(n)}_i)+\sqrt{h}\,\sigma(t_i,y,\mu^{(n)}_i)\,\xi_{i+1},
\qquad\tilde X^{(n)}_j=x,
\end{equation}
a sum of independent Gaussian increments, with mean
\begin{equation}
\bar m_n(j,k;y)=\sum_{i=j}^{k-1}h\,b(t_i,y,\mu^{(n)}_i)
\end{equation}
and covariance
\begin{equation}
\bar\Sigma_n(j,k;y)=\sum_{i=j}^{k-1}h\,\Sigma(t_i,y,\mu^{(n)}_i),
\end{equation}
so that its transition density $\tilde p_n(j,k;x,y)$ is again exactly Gaussian:
\begin{multline}
\label{eq:frozen-disc}
\tilde p_n(j,k;x,y)=(2\pi)^{-d/2}\det\bar\Sigma_n(j,k;y)^{-1/2}\\
\times\exp\Big[-\tfrac12\big(y-x-\bar m_n(j,k;y)\big)^{\top}\bar\Sigma_n(j,k;y)^{-1}
\big(y-x-\bar m_n(j,k;y)\big)\Big].
\end{multline}
This exact Gaussianity, specific to the Brownian-driven setting, is the key simplification of the
present analysis. Here and throughout the paper we adopt the following convention: a semicolon
separates the (discrete) time indices or parameters of an object from its spatial evaluation
variables. Thus in a discrete (Euler-chain) object it separates the discrete time indices from the
spatial variables (e.g.\
$\tilde p_n(j,k;x,y)$, $\bar m_n(j,k;y)$, $\bar\Sigma_n(j,k;y)$), and the
Gaussian density $g(m,\Theta;z)$ uses the same semicolon to separate its parameters $(m,\Theta)$
from the evaluation point $z$; the corresponding continuous-time objects are written with commas
(e.g.\ $\tilde p(s,t,x,y)$, $\bar m(s,t,y)$, $\bar\Sigma(s,t,y)$).

\subsection{Generators and kernels}
\label{sec:kernels-def}

Throughout this subsection the measure flow $\nu$ is \emph{fixed}; all objects below are therefore
classical time-inhomogeneous generators and kernels, and each identity holds for every fixed $\nu$,
independently of any property of the measure flow.

\textbf{Continuous generators.} By It\^o's formula, the generator of the (measure-frozen) diffusion
$\dd X_t=b(t,X_t,\nu_t)\dd t+\sigma(t,X_t,\nu_t)\dd W_t$ is
\begin{equation}
\label{eq:Lnu}
L^\nu f(s,x)=\sum_i b_i(s,x,\nu_s)\,\partial_{x_i}f+\frac12\sum_{i,j}\Sigma_{ij}(s,x,\nu_s)\,\partial^2_{x_ix_j}f,
\end{equation}
and its spatial freezing at a terminal point $y$ is the operator
\begin{equation}
\label{eq:Ltilden}
\tilde L^\nu f(s,x)=\sum_i b_i(s,y,\nu_s)\,\partial_{x_i}f+\frac12\sum_{i,j}\Sigma_{ij}(s,y,\nu_s)\,\partial^2_{x_ix_j}f,
\end{equation}
with coefficients evaluated at $(s,y,\nu_s)$ but still acting in the $x$-variable. In
\eqref{eq:Lnu}--\eqref{eq:Ltilden} the diffusion term is the trace
$\frac12\tr(\Sigma\,\nabla^2 f)=\frac12\sum_{i,j}\Sigma_{ij}\partial^2_{x_ix_j}f$, with
$\Sigma=\sigma\sigma^{\top}$; this is exact and involves no reordering of matrix factors.

\textbf{The continuous kernel.} Set
\begin{equation}
\label{eq:Hdef}
H(s,t,x,y)=(L^\mu-\tilde L^\mu)\tilde p(s,t,x,y),
\end{equation}
i.e.\ (with the frozen coefficients)
\begin{multline}
\label{eq:kernelH}
H(s,t,x,y)=\frac12\sum_{i,j}\big(\Sigma_{ij}(s,x,\mu_s)-\Sigma_{ij}(s,y,\mu_s)\big)
\frac{\partial^2\tilde p}{\partial x_i\partial x_j}(s,t,x,y)\\
+\sum_i\big(b_i(s,x,\mu_s)-b_i(s,y,\mu_s)\big)\frac{\partial\tilde p}{\partial x_i}(s,t,x,y).
\end{multline}
The kernel $H$ is the ``freezing error'' generator: the difference between the true generator
$L^\mu$ and its frozen version $\tilde L^\mu$, applied to the frozen density $\tilde p$.

\textbf{Discrete generators and the discrete kernel.} For the chain \eqref{eq:Euler} (frozen at
$\nu=\mu^{(n)}$), the one-step transition from $x$ at step $i$ to $z$ at step $i+1$ has the Gaussian
density
\begin{equation}
\label{eq:onestep}
r^{(n)}_i(x,z)=g\big(x+h\,b(t_i,x,\mu^{(n)}_i),\ h\,\Sigma(t_i,x,\mu^{(n)}_i);z\big),
\end{equation}
with $\mathcal N(m,\Theta)$ as in Section~\ref{sec:kernels}; let
$\tilde r^{(n),y}_i(x,\cdot)$ denote the same density with the coefficients frozen at $y$. We define
the discrete backward generator, acting in the first spatial variable, by
\begin{equation}
\label{eq:discgen}
(L_n f)(i,k;x,y)=\frac1h\Big[\int_{\R^d}r^{(n)}_i(x,z)\,f(i+1,k;z,y)\dd z-f(i,k;x,y)\Big],
\end{equation}
and its frozen counterpart
\begin{equation}
\label{eq:Lntilde}
(\tilde L_n f)(i,k;x,y)=\frac1h\Big[\int_{\R^d}\tilde r^{(n),y}_i(x,z)\,f(i+1,k;z,y)\dd z-f(i,k;x,y)\Big].
\end{equation}
Thus $L_n,\tilde L_n$ are the exact discrete counterparts of \eqref{eq:Lnu}--\eqref{eq:Ltilden}: the
one-step transition operator minus the identity, divided by $h$. The discrete kernel is then defined
\emph{exactly} by
\begin{equation}
\label{eq:Hn-exact}
H_n(i,k;x,y):=(L_n-\tilde L_n)\tilde p_n(i,k;x,y)
=\frac1h\int_{\R^d}\big[r^{(n)}_i(x,z)-\tilde r^{(n),y}_i(x,z)\big]\tilde p_n(i+1,k;z,y)\dd z,
\end{equation}
where the second equality is immediate from the definitions (the identity term cancels); note also
that $\tilde L_n\tilde p_n=0$ identically, by the exact Chapman--Kolmogorov equation of the frozen
chain, so that $H_n=L_n\tilde p_n$. No approximation of the generator by a differential operator
is involved in \eqref{eq:Hn-exact}.

\subsection{Convolutions}
\label{sec:convolution}

Define the continuous convolution
\begin{equation}
\label{eq:conv}
(\phi*\psi)(s,t,x,y)=\int_s^t\dd u\int_{\R^d}\phi(s,u,x,z)\,\psi(u,t,z,y)\dd z,
\end{equation}
and the discrete convolution
\begin{equation}
\label{eq:convn}
(\phi\circledast\psi)(j,k;x,y)=\sum_{i=j}^{k-1}h\int_{\R^d}\phi(j,i;x,z)\,\psi(i,k;z,y)\dd z.
\end{equation}
\textbf{Convention on chains.} Whenever a grid convolution is expanded we write the resulting sum over
chains of grid indices in the inclusive form $j=i_0\le i_1<\cdots<i_r<i_{r+1}=k$, that is, $j\le
i_1<\cdots<i_r\le k-1$; this is exactly the range of \eqref{eq:convn}, so that an $r$-fold grid convolution
at $(j,k)$ sums over the $\binom{k-j}{r}$ choices of indices and vanishes identically for $r>k-j$. The only
terms not of the generic form are those with $i_1=i_0$, in which the leading factor is evaluated at
coinciding times. If that factor vanishes there, the term is simply absent; this holds for the kernel powers
on the diagonal, $K_n^{\circledast r}(j,j)=H^{\mathrm{g}}_n{}^{\circledast r}(j,j)=0$ for $r\ge1$ (for $r\ge2$ the diagonal convolution sum is empty, while for $r=1$ the single kernels are taken to vanish on the diagonal, $K_n(j,j)=H^{\mathrm{g}}_n(j,j)=0$), as well as for the difference-type
objects $\mathsf{D}_n,\mathsf{E}_n,\varepsilon\circledast K_n$ occurring as leading factors below. If
instead the leading factor is a density, that density collapses to $\delta_x$, one Gaussian factor disappears
from the chain, and the term is the one carrying the prefactor $h=n^{-1}$ in the derivation of the discrete
parametrix, namely $h\,H_n$; this is also the case treated separately as the initial cell $i=j$ in Step~1 of
the proof of Lemma~\ref{lem:T1}.

In either case the term is covered by the estimates below: reading a vanishing width as
$\varphi_{c,0}:=(\pi c)^{d/2}\delta_0$, the $\rho\downarrow0$ limit of the exact convolution identity
$\varphi_{c,\rho}*\varphi_{c,\rho'}=
(\pi c)^{d/2}\varphi_{c,\sqrt{\rho^2+\rho'^2}}$, the collapsed chain sum is controlled by the corresponding
generic estimate, each collapsed term carrying one factor $(\pi c)^{d/2}$ less, which is absorbed into the
constants $C_1$ and $M$; the convolution identity and the Beta-type integrals below are unchanged in form.
Equivalently, every chain sum below may be read as if all widths were strictly positive, so that the generic
estimates always remain upper bounds for it. All the chain sums of this paper are read in this sense.

Set $\phi* H^{*0}=\phi$ and $\phi* H^{*r}=(\phi* H^{*(r-1)})* H$; similarly for
$\circledast$.

\subsection{Bounds on the kernels}
\label{sec:kernbound}

\begin{lemma}
\label{lem:kernbound}
Under Assumptions~\ref{ass:spatial} and \ref{ass:elliptic}, there are constants $C_1<\infty$ and
$c>0$ such that, with $\rho=\rho(s,t)$ (see \eqref{eq:rhodef}),
\begin{equation}
\label{eq:boundH}
\abs{H(s,t,x,y)}\le C_1\,\rho^{-1}\,\varphi_{c,\rho}(y-x),
\end{equation}
and, for every $r\ge0$,
\begin{equation}
\label{eq:bound-r}
\abs{\tilde p* H^{*r}(s,t,x,y)}
\le C_1^{r+1}\rho^{r}\,\Gamma^{-1}\big(1+\tfrac r2\big)\,\varphi_{c,\rho}(y-x).
\end{equation}
\end{lemma}

\begin{proof}
\textbf{Step 1: derivative bounds on the frozen density.} Write
$\tilde p(s,t,x,y)=g(\bar m,\bar\Sigma;y-x)$ with
\begin{equation}
g(z)=(2\pi)^{-d/2}\det\bar\Sigma(s,t,y)^{-1/2}
\exp\Big[-\tfrac12\big(z-\bar m(s,t,y)\big)^{\top}\bar\Sigma(s,t,y)^{-1}\big(z-\bar m(s,t,y)\big)\Big],
\end{equation}
and abbreviate $\bar m=\bar m(s,t,y)$, $\bar\Sigma=\bar\Sigma(s,t,y)$. By uniform ellipticity
(Assumption~\ref{ass:elliptic}) and $\bar\Sigma=\int_s^t\Sigma(u,y,\mu_u)\dd u$,
\begin{equation}
c_0\rho^2\norm{\theta}^2\le\theta^{\top}\bar\Sigma\,\theta\le C_0\rho^2\norm{\theta}^2\qquad(\forall\,\theta\in\R^d),
\end{equation}
which is the ellipticity condition integrated over $[s,t]$; equivalently
$\lambda_{\min}(\bar\Sigma)\ge c_0\rho^2$ and $\lambda_{\max}(\bar\Sigma)\le C_0\rho^2$. Hence
\begin{equation}
\det\bar\Sigma^{-1/2}\le c_0^{-d/2}\rho^{-d},\qquad
\norm{\bar m}\le\norm{b}_\infty(t-s)=C\rho^2,
\end{equation}
and, for the inverse,
\begin{equation}
\theta^{\top}\bar\Sigma^{-1}\theta\ge C_0^{-1}\rho^{-2}\norm{\theta}^2,\qquad
\norm{\bar\Sigma^{-1}\theta}\le c_0^{-1}\rho^{-2}\norm{\theta},\qquad
\abs{\bar\Sigma^{-1}_{ij}}\le c_0^{-1}\rho^{-2}.
\end{equation}
Hence
\begin{equation}
\label{eq:gbound}
g(z)\le C\rho^{-d}\exp\Big[-\frac{\norm{z-\bar m}^2}{2C_0\rho^2}\Big].
\end{equation}
Since $\partial_{x_i}\tilde p=-\partial_{z_i}g$ and, for $z=y-x$,
\begin{equation}
\partial_{z_i}g(z)=-g(z)\big[\bar\Sigma^{-1}(z-\bar m)\big]_i,\qquad
\partial^2_{z_iz_j}g(z)=g(z)\Big(\big[\bar\Sigma^{-1}(z-\bar m)\big]_i
\big[\bar\Sigma^{-1}(z-\bar m)\big]_j-\bar\Sigma^{-1}_{ij}\Big),
\end{equation}
we obtain, using \eqref{eq:gbound} and the bounds
$\norm{\bar\Sigma^{-1}(z-\bar m)}\le C\rho^{-2}\norm{z-\bar m}$,
$\abs{\bar\Sigma^{-1}_{ij}}\le C\rho^{-2}$, together with the elementary inequalities
\begin{equation}
\norm{z-\bar m}\,e^{-\frac{\norm{z-\bar m}^2}{2C_0\rho^2}}\le C\rho\,e^{-\frac{\norm{z-\bar m}^2}{4C_0\rho^2}},\qquad
\norm{z-\bar m}^2\,e^{-\frac{\norm{z-\bar m}^2}{2C_0\rho^2}}\le C\rho^2\,e^{-\frac{\norm{z-\bar m}^2}{4C_0\rho^2}},
\end{equation}
the bounds
\begin{equation}
\abs{\partial_{x_i}\tilde p}\le C\rho^{-d-1}e^{-\frac{\norm{z-\bar m}^2}{4C_0\rho^2}},\qquad
\abs{\partial^2_{x_ix_j}\tilde p}\le C\rho^{-d-2}e^{-\frac{\norm{z-\bar m}^2}{4C_0\rho^2}}.
\end{equation}
Finally, since $\norm{\bar m}\le C\rho^2$ and $\rho\le1$, we have
$e^{-\frac{\norm{z-\bar m}^2}{4C_0\rho^2}}\le C e^{-\frac{\norm z^2}{c\rho^2}}$ for some $c>0$: if
$\norm z\ge2\norm{\bar m}$ then $\norm{z-\bar m}\ge\norm z-\norm{\bar m}\ge\tfrac12\norm z$, while if
$\norm z<2\norm{\bar m}\le2C\rho^2$ then $1\le e^{4C^2/c}e^{-\frac{\norm z^2}{c\rho^2}}$. Consequently,
with $z=y-x$,
\begin{equation}
\label{eq:dpbound}
\abs{\partial^\alpha_x\tilde p(s,t,x,y)}\le C_\alpha\,\rho^{-\abs{\alpha}}\,\varphi_{c,\rho}(y-x),\qquad \abs{\alpha}=0,1,2,
\end{equation}
where the case $\abs{\alpha}=0$ follows from \eqref{eq:gbound} together with the conversion of the preceding sentence.

\textbf{Step 2: the bound \eqref{eq:boundH}.} From \eqref{eq:kernelH}, by the Lipschitz continuity of
$\Sigma$ and $b$ in $x$ (Assumption~\ref{ass:spatial}),
\begin{equation}
\abs{\Sigma_{ij}(s,x,\mu_s)-\Sigma_{ij}(s,y,\mu_s)}\le C\norm{x-y},\qquad
\abs{b_i(s,x,\mu_s)-b_i(s,y,\mu_s)}\le C\norm{x-y},
\end{equation}
so, using \eqref{eq:dpbound},
\begin{equation}
\begin{aligned}
\abs{H(s,t,x,y)}
&\le\frac12\sum_{i,j}\abs{\Sigma_{ij}(s,x,\mu_s)-\Sigma_{ij}(s,y,\mu_s)}\,\abs{\partial^2_{x_ix_j}\tilde p}\\
&\quad+\sum_i\abs{b_i(s,x,\mu_s)-b_i(s,y,\mu_s)}\,\abs{\partial_{x_i}\tilde p}\\
&\le C\norm{x-y}\Big[\rho^{-2}+\rho^{-1}\Big]\varphi_{c,\rho}(y-x).
\end{aligned}
\end{equation}
By item~(ii) of Lemma~\ref{lem:kernel}, $\norm{x-y}\,\varphi_{c,\rho}(y-x)\le C\rho\,\varphi_{c+1,\rho}(y-x)$, hence
\begin{equation}
\abs{H(s,t,x,y)}\le C\rho\big[\rho^{-2}+\rho^{-1}\big]\varphi_{c+1,\rho}(y-x)
=C\big[\rho^{-1}+1\big]\varphi_{c+1,\rho}(y-x)\le C\rho^{-1}\varphi_{c+1,\rho}(y-x),
\end{equation}
because $\rho\le1$ gives $1\le\rho^{-1}$. Renaming $c+1$ as $c$ yields \eqref{eq:boundH}; moreover, since the weight $\varphi_{c,\rho}$ is increasing in $c$, we may enlarge $c$ once and for all so that the same constant $c$ appears in \eqref{eq:dpbound} and in \eqref{eq:boundH} (this unification is used in the convolution identity of Step~4).

\textbf{Step 3: a nested time integral.} For $r\ge1$ set $u_{r+1}=t$ and
\begin{equation}
J_r(s,t):=\int_{s<u_1<\cdots<u_r<t}\prod_{k=1}^{r}(u_{k+1}-u_k)^{-1/2}\dd u_1\cdots\dd u_r.
\end{equation}
We claim
\begin{equation}
J_r(s,t)=\frac{\Gamma(1/2)^{r}}{\Gamma(1+r/2)}(t-s)^{r/2}.
\end{equation}
For $r=1$,
\begin{equation}
J_1(s,t)=\int_s^t(t-u_1)^{-1/2}\dd u_1=2(t-s)^{1/2}
=\frac{\Gamma(1/2)}{\Gamma(3/2)}(t-s)^{1/2},
\end{equation}
since $\Gamma(1/2)=\sqrt\pi$ and $\Gamma(3/2)=\sqrt\pi/2$. Suppose the formula holds for $r-1$; then
\begin{equation}
\begin{aligned}
J_r(s,t)&=\int_s^t(t-u_r)^{-1/2}J_{r-1}(s,u_r)\dd u_r\\
&=\frac{\Gamma(1/2)^{r-1}}{\Gamma(1+(r-1)/2)}\int_s^t(t-u_r)^{-1/2}(u_r-s)^{(r-1)/2}\dd u_r\\
&=\frac{\Gamma(1/2)^{r-1}}{\Gamma((r+1)/2)}\,(t-s)^{r/2}\,\mathrm{B}\big(\tfrac12,\tfrac{r+1}{2}\big)\\
&=\frac{\Gamma(1/2)^{r}}{\Gamma(1+r/2)}(t-s)^{r/2},
\end{aligned}
\end{equation}
where the second equality uses item~(iii) of Lemma~\ref{lem:kernel} with $\alpha=-1$, $\beta=r-1$, and
the third uses $\mathrm{B}(a,b)=\Gamma(a)\Gamma(b)/\Gamma(a+b)$. This proves the claim.

\textbf{Step 4: the bound \eqref{eq:bound-r}.} For $r=0$ the claim reads $\abs{\tilde p}\le C_1\varphi_{c,\rho}$,
which is the case $\abs{\alpha}=0$ of \eqref{eq:dpbound} (enlarging $C_1$ if necessary). For $r\ge1$,
expand $\tilde p* H^{*r}$ as the $r$-fold convolution
\begin{multline}
(\tilde p* H^{*r})(s,t,x,y)
=\int_{s<u_1<\cdots<u_r<t}\int_{\R^{dr}}\tilde p(s,u_1,x,z_1)\\
\times\Big[\prod_{k=1}^{r-1}H(u_k,u_{k+1},z_k,z_{k+1})\Big]
H(u_r,t,z_r,y)\,\dd z_1\cdots\dd z_r\,\dd u_1\cdots\dd u_r.
\end{multline}
By \eqref{eq:dpbound} (with $\abs{\alpha}=0$) and \eqref{eq:boundH},
\begin{equation}
\begin{aligned}
\abs{\tilde p(s,u_1,x,z_1)}&\le C\varphi_{c,\sqrt{u_1-s}}(z_1-x),\\
\abs{H(u_k,u_{k+1},z_k,z_{k+1})}&\le C_1(u_{k+1}-u_k)^{-1/2}\varphi_{c,\sqrt{u_{k+1}-u_k}}(z_{k+1}-z_k).
\end{aligned}
\end{equation}
Integrating out $z_1,\dots,z_r$ successively using the exact convolution identity
\begin{equation}
\varphi_{c,\rho}*\varphi_{c,\rho'}=(\pi c)^{d/2}\,\varphi_{c,\sqrt{\rho^2+\rho'^2}},
\end{equation}
which follows by writing $\varphi_{c,\rho}$ as a Gaussian density, exactly as in the proof of
item~(i) of Lemma~\ref{lem:kernel}, the first integration gives
\begin{equation}
\int_{\R^d}\varphi_{c,\sqrt{u_1-s}}(z_1-x)\,\varphi_{c,\sqrt{u_2-u_1}}(z_2-z_1)\dd z_1
=(\pi c)^{d/2}\,\varphi_{c,\sqrt{u_2-s}}(z_2-x),
\end{equation}
because the widths add in quadrature, $\sqrt{(u_1-s)+(u_2-u_1)}=\sqrt{u_2-s}$; iterating over
$z_2,\dots,z_r$ telescopes to $\sqrt{t-s}=\rho$ and multiplies by $(\pi c)^{d/2}$ at each step, so the
spatial factor is exactly $(\pi c)^{dr/2}\varphi_{c,\rho}(y-x)$. Therefore, combining with the time
integral $J_r$ of Step~3,
\begin{equation}
\abs{\tilde p* H^{*r}}(s,t,x,y)
\le C\,C_1^{r}\,(\pi c)^{dr/2}\,J_r(s,t)\,\varphi_{c,\rho}(y-x)
\le C_1^{r+1}\,\rho^{r}\,\Gamma^{-1}\big(1+\tfrac r2\big)\,\varphi_{c,\rho}(y-x),
\end{equation}
where in the last step the $r$-independent constants $C$, $(\pi c)^{d/2}$ and $\Gamma(1/2)=\sqrt\pi$
are absorbed into a once-enlarged $C_1$. This proves \eqref{eq:bound-r}.
\end{proof}

\subsection{Parametrix expansion: continuous case}
\label{sec:param-cont}

\begin{lemma}
\label{lem:param-cont}
The transition density $p(s,t,x,y)$ of the (measure-frozen) diffusion with coefficients
$(t,x)\mapsto b(t,x,\mu_t)$ and $(t,x)\mapsto\Sigma(t,x,\mu_t)$ admits the expansion
\begin{equation}
\label{eq:param-cont}
p(s,t,x,y)=\sum_{r=0}^{\infty}\big(\tilde p* H^{*r}\big)(s,t,x,y),
\end{equation}
which converges absolutely for every $s<t$; the convergence is uniform in $(x,y)\in\R^{2d}$ for
fixed $0\le s<t\le1$, and each term satisfies the bound \eqref{eq:bound-r}.
\end{lemma}

\begin{proof}
The proof proceeds in five steps; throughout, $(t,y)$ is held fixed and all differential operators
act on the backward variables $(s,x)$.

\textbf{Step 1: the two Kolmogorov equations.} The transition density $p$ solves the \emph{backward}
Kolmogorov equation
\begin{equation}
\label{eq:bk}
\partial_s p+L^\mu p=0,\qquad \lim_{s\uparrow t}p(s,t,\cdot,y)=\delta_y,
\end{equation}
where $L^\mu$ acts in $x$; this is classical for diffusions with bounded, Lipschitz-continuous
coefficients and a uniformly elliptic diffusion (see \cite{Friedman1964}). The frozen density
$\tilde p$ solves, \emph{exactly},
\begin{equation}
\label{eq:bktil}
\partial_s\tilde p+\tilde L^\mu\tilde p=0,\qquad \lim_{s\uparrow t}\tilde p(s,t,\cdot,y)=\delta_y,
\end{equation}
since $\tilde p$ is the transition density of the affine Gaussian process \eqref{eq:mvtil}
(equivalently, by differentiating \eqref{eq:frozen-cont} using $\partial_s\bar m=-b(s,y,\mu_s)$ and
$\partial_s\bar\Sigma=-\Sigma(s,y,\mu_s)$).

\textbf{Step 2: the error equation.} Set $\vartheta(s,t,x,y)=p(s,t,x,y)-\tilde p(s,t,x,y)$. Subtracting
\eqref{eq:bktil} from \eqref{eq:bk},
\begin{equation}
\partial_s \vartheta+L^\mu \vartheta=(\tilde L^\mu-L^\mu)\tilde p=-H(\cdot,t,\cdot,y),\qquad \vartheta(t,t,\cdot,y)=0,
\end{equation}
using \eqref{eq:bktil} and the definition \eqref{eq:Hdef} of $H$; here $-H(\cdot,t,\cdot,y)$
denotes the function $(s,x)\mapsto -H(s,t,x,y)$ with $(t,y)$ fixed.

\textbf{Step 3: Duhamel's formula.} By the variation-of-constants (Duhamel) formula for the backward
parabolic Cauchy problem --- see \cite{BallyTalay1996a,BallyTalay1996b} and
\cite{KonakovMammen2000}, and the classical theory in \cite{Friedman1964} --- the unique solution of
the inhomogeneous problem $\partial_s v+L^\mu v=f$, $v(t,\cdot)=\Psi$ is given by
\begin{equation}
v(s,x)=\int_{\R^d}p(s,t,x,z)\,\Psi(z)\dd z-\int_s^t\int_{\R^d}p(s,u,x,z)\,f(u,z)\dd z\dd u.
\end{equation}
Applying this with $\Psi=0$ and $f=-H$ yields the Volterra identity
\begin{equation}
\label{eq:volterra}
p(s,t,x,y)=\tilde p(s,t,x,y)+(p* H)(s,t,x,y).
\end{equation}

\textbf{Step 4: iteration.} Iterating \eqref{eq:volterra} once,
\begin{equation}
p=\tilde p+p* H=\tilde p+\big(\tilde p+p* H\big)* H
=\tilde p+\tilde p* H+p* H^{*2},
\end{equation}
and by induction on $N\ge1$,
\begin{equation}
\label{eq:partial}
p(s,t,x,y)=\sum_{r=0}^{N-1}\big(\tilde p* H^{*r}\big)(s,t,x,y)
+\big(p* H^{*N}\big)(s,t,x,y).
\end{equation}

\textbf{Step 5: absolute convergence and the vanishing remainder.} By \eqref{eq:bound-r}, for every
$r\ge0$,
\begin{equation}
\abs{\tilde p* H^{*r}(s,t,x,y)}
\le C_1^{r+1}\rho^{r}\,\Gamma^{-1}\big(1+\tfrac r2\big)\,\varphi_{c,\rho}(y-x),
\end{equation}
so the series $\sum_{r\ge0}\tilde p* H^{*r}$ converges absolutely, since $\Gamma(1+r/2)$ grows
superexponentially. It remains to show that the remainder $\mathcal R_N:=p* H^{*N}$ tends to zero as
$N\to\infty$. Using the Aronson Gaussian estimate $p(s,u,x,z)\le C\varphi_{c,\sqrt{u-s}}(z-x)$
(see \cite{Aronson1967} or \cite{Friedman1964}), and iterating the same convolution estimate as in
the proof of \eqref{eq:bound-r}, we obtain the same bound
\begin{equation}
\abs{\mathcal R_N(s,t,x,y)}\le C_1^{N+1}\rho^{N}\,\Gamma^{-1}\big(1+\tfrac N2\big)\,\varphi_{c,\rho}(y-x)
\xrightarrow[N\to\infty]{}\ 0,
\end{equation}
uniformly in $(x,y)\in\R^{2d}$ for fixed $s<t$. Passing to the limit in \eqref{eq:partial}
yields \eqref{eq:param-cont}.
\end{proof}

\subsection{Parametrix expansion: discrete case}
\label{sec:param-disc}

\begin{lemma}
\label{lem:param-disc}
The transition density $p_n$ of the Euler chain \eqref{eq:Euler} admits the expansion
\begin{equation}
\label{eq:param-disc}
p_n(j,k;x,y)=\sum_{r=0}^{k-j}\big(\tilde p_n\circledast H_n^{\circledast r}\big)(j,k;x,y).
\end{equation}
\end{lemma}

\begin{proof}
By the backward Chapman--Kolmogorov equation, for $j<k$,
\begin{equation}
\begin{aligned}
p_n(j,k;x,y)&=\int_{\R^d}r^{(n)}_j(x,z)\,p_n(j+1,k;z,y)\,\dd z,\\
\tilde p_n(j,k;x,y)&=\int_{\R^d}\tilde r^{(n),y}_j(x,z)\,\tilde p_n(j+1,k;z,y)\,\dd z.
\end{aligned}
\end{equation}
Subtracting and using \eqref{eq:Hn-exact},
\begin{equation}
p_n(j,k;x,y)-\tilde p_n(j,k;x,y)
=h\,H_n(j,k;x,y)+\int_{\R^d}r^{(n)}_j(x,z)\,\big(p_n-\tilde p_n\big)(j+1,k;z,y)\,\dd z.
\end{equation}
Iterating, and using that the residual term vanishes at $j=k$ (since
$p_n(k,k;\cdot,y)=\tilde p_n(k,k;\cdot,y)$, both equal to $\delta_y$), we obtain
\begin{equation}
p_n(j,k;x,y)=\tilde p_n(j,k;x,y)+\sum_{i=j}^{k-1}h\int_{\R^d}p_n(j,i;x,z)\,H_n(i,k;z,y)\,\dd z,
\end{equation}
that is $p_n=\tilde p_n+p_n\circledast H_n$. Iterating this identity, by induction on $N\ge1$, gives the partial expansion
\begin{equation}
\label{eq:disc-partial}
p_n(j,k;x,y)=\sum_{r=0}^{N-1}\big(\tilde p_n\circledast H_n^{\circledast r}\big)(j,k;x,y)
+\big(p_n\circledast H_n^{\circledast N}\big)(j,k;x,y).
\end{equation}

It remains to explain why the remainder in \eqref{eq:disc-partial} vanishes. By \eqref{eq:convn}, an
$r$-fold discrete convolution at $(j,k)$ is a sum over chains of grid indices
$j\le i_1<i_2<\cdots<i_r\le k-1$; since the set $\{j,j+1,\dots,k-1\}$ contains only $k-j$ elements,
such a chain exists only for $r\le k-j$, and the convolution is identically zero for $r\ge k-j+1$.
Taking $N=k-j+1$ in \eqref{eq:disc-partial}, the remainder
$\big(p_n\circledast H_n^{\circledast (k-j+1)}\big)(j,k;\cdot,\cdot)$ vanishes identically, and
\eqref{eq:disc-partial} reduces to the finite sum \eqref{eq:param-disc}. Thus the tail term
vanishes vacuously --- no index chain of length $k-j+1$ exists between $j$ and $k$ --- rather than by
a limiting argument as in the continuous case (Lemma~\ref{lem:param-cont}, Step~5). Only the Markov
property and the exact identity \eqref{eq:Hn-exact} are used; no property of the measure flow enters.
\end{proof}

\section{Comparison of the Two Expansions}
\label{sec:comparison}

We now compare \eqref{eq:param-cont} and \eqref{eq:param-disc}. The continuous coefficients are
frozen at the continuous flow $\mu$, while the discrete coefficients are frozen at $\mu^{(n)}$; we
isolate this discrepancy in the following lemma, and in Subsection~\ref{sec:approx-kernel} we
compare the two kernels themselves through the auxiliary kernels $H^{(n)}$ and $K_n$. Since the
continuous expansion \eqref{eq:param-cont} is an infinite series whereas the discrete expansion
\eqref{eq:param-disc} is a finite sum whose terms vanish for $r>k-j$, the comparison below is carried out
only on the common range $r\le k-j$; the complementary continuous tail $r>k-j$ is treated separately in
Section~\ref{sec:completion}.

\textbf{Convention on the convolution order.}
All the estimates of this section are stated at a pair of grid indices $0\le j<k\le n$, and we
adopt throughout the convention that the order $r$ of a convolution --- continuous ($*$) or
discrete ($\circledast$) --- never exceeds the number of grid steps between the two endpoints,
\begin{equation}
\label{eq:conv-order}
0\le r\le k-j
\end{equation}
(so that, in particular, $r\le n$; the bound $k-j$ is one less than the number of grid nodes
$t_j,\dots,t_k$). This bound is exactly the maximal admissible order: by \eqref{eq:convn} an
$r$-fold discrete convolution at $(j,k)$ is a sum over strictly increasing chains of grid indices
$j\le i_1<\cdots<i_r\le k-1$, and since the set $\{j,j+1,\dots,k-1\}$ has only $k-j$ elements
such a chain exists precisely when $r\le k-j$, while for $r>k-j$ the discrete convolution
vanishes identically. The continuous convolution \eqref{eq:conv} is of course not subject to this
restriction; hence every comparison between a continuous and a discrete convolution carried out
below would be vacuous on the complementary range $r>k-j$, where only the continuous term
survives. Restricting the order by \eqref{eq:conv-order} therefore keeps the two sides of each
comparison on an equal footing, and the complementary range $r>k-j$, in which the discrete side is
absent, is treated separately in Section~\ref{sec:completion}. Estimates involving a single type
of convolution only are unaffected by the convention.

\textbf{Convention on the Gaussian constant $c$.}
All the Gaussian weights $\varphi_{c,\rho}$ occurring in this section are used with a single
constant $c>0$, common to the bounds inherited from Section~\ref{sec:parametrix} and to those
established below. Whenever two estimates carrying possibly different constants are combined, $c$
is silently enlarged to the larger of the two: since $\varphi_{c,\rho}$ is increasing in $c$ (a
larger $c$ makes the weight decay more slowly), such an enlargement only relaxes an upper bound, so
every estimate proved below remains valid; and since only finitely many constants, all independent
of $n$, are involved, the resulting $c$ is still independent of $n$. In particular this is what
makes the hypothesis of the discrete propagation estimate \eqref{eq:disc-iter} (a single constant
$c$ common to $G$ and to all the $F_\ell$) applicable in Lemma~\ref{lem:replace},
Lemma~\ref{lem:T1} and Lemma~\ref{lem:T2}.

\subsection{Comparison of the frozen densities}
\label{sec:compare-frozen}

\begin{lemma}
\label{lem:compare-frozen}
Under Assumptions~\ref{ass:spatial}--\ref{ass:measLip}, there exist constants $C<\infty$ and
$c>0$, independent of $n$, such that, for all $0\le j<k\le n$ and $x,y\in\R^d$, with
$\rho=\rho_{jk}$,
\begin{equation}
\label{eq:compare-frozen}
\abs{\tilde p_n(j,k;x,y)-\tilde p(t_j,t_k,x,y)}
\le C\,n^{-1/2}\varphi_{c,\rho}(y-x).
\end{equation}
\end{lemma}

\begin{proof}
Both $\tilde p_n$ and $\tilde p$ are Gaussian densities; it suffices to bound the differences between
their means and covariances and then apply the Gaussian comparison of
Lemma~\ref{lem:kernel}(v). Write $\Delta m:=\bar m_n(j,k;y)-\bar m(t_j,t_k,y)$ and
$\Delta\Sigma:=\bar\Sigma_n(j,k;y)-\bar\Sigma(t_j,t_k,y)$, and put $s=t_j$, $t=t_k$.

\emph{(a) Bounding $\Delta m$ and $\Delta\Sigma$.} By Assumption~\ref{ass:measLip},
Remark~\ref{rem:autotime} and Lemma~\ref{lem:flowrate},
\begin{equation}
\begin{aligned}
\norm{\Delta m}
&\le\norm{\sum_{i=j}^{k-1}n^{-1}\Big[b(t_i,y,\mu^{(n)}_i)-b(t_i,y,\mu_{t_i})\Big]}
+\norm{\sum_{i=j}^{k-1}n^{-1}\,b(t_i,y,\mu_{t_i})-\int_s^t b(u,y,\mu_u)\dd u}\\
&\le \Lambda\sum_{i=j}^{k-1}n^{-1}\,\mathcal{W}_2\big(\mu^{(n)}_i,\mu_{t_i}\big)+C\,\rho^{2}n^{-1/2}\\
&\le \Lambda\,(t-s)\sup_{u\in[s,t]}\mathcal{W}_2\big(\mu^{(n)}_u,\mu_u\big)+C\,\rho^{2}n^{-1/2}\\
&\le C\,\rho^{2}n^{-1/2}.
\end{aligned}
\end{equation}
In the first term above we use $\mu^{(n)}_i=\mu^{(n)}_{t_i}$, the value of the interpolated flow at
the grid point $t_i$ (cf.\ Section~\ref{sec:moments}). The second term on the right-hand side is the
left-endpoint Riemann-sum error of the frozen
coefficient $u\mapsto b(u,y,\mu_u)$; since this coefficient is $\tfrac12$-H\"older continuous
(Remark~\ref{rem:autotime}), each subinterval contributes $O(n^{-3/2})$, summing to
$O\big((t-s)n^{-1/2}\big)=O\big(\rho^2 n^{-1/2}\big)$. The last inequality uses $t-s=\rho^2$
and \eqref{eq:flowrate}. Similarly,
\begin{equation}
\norm{\Delta\Sigma}\le C\,\rho^2 n^{-1/2}.
\end{equation}

\emph{(b) The Gaussian comparison.} By Lemma~\ref{lem:kernel}(v) (with $\alpha=0$, $z=y-x$), with
$m=\bar m_n(j,k;y)$, $m'=\bar m(t_j,t_k,y)$, $\Theta=\bar\Sigma_n(j,k;y)$,
$\Theta'=\bar\Sigma(t_j,t_k,y)$ (both $c_0\rho^2 I\preceq\Theta,\Theta'\preceq C_0\rho^2 I$
by Assumption~\ref{ass:elliptic}), and with $\norm{m},\norm{m'}\le C\rho^2$ (which holds since the drift is bounded),
\begin{equation}
\abs{g(m,\Theta;z)-g(m',\Theta';z)}
\le C\Big[\norm{m-m'}\,\rho^{-1}+\norm{\Theta-\Theta'}\,\rho^{-2}\Big]\varphi_{c,\rho}(z).
\end{equation}
Combining with the bounds of part~(a) gives
\begin{equation}
\begin{aligned}
\abs{\tilde p_n(j,k;x,y)-\tilde p(t_j,t_k,x,y)}
&\le C\Big[\rho^2 n^{-1/2}\cdot\rho^{-1}+\rho^2 n^{-1/2}\cdot\rho^{-2}\Big]\varphi_{c,\rho}(y-x)\\
&=C\big[\rho\,n^{-1/2}+n^{-1/2}\big]\varphi_{c,\rho}(y-x).
\end{aligned}
\end{equation}
Since $\rho^2=t-s\le1$, this is bounded by $Cn^{-1/2}\varphi_{c,\rho}(y-x)$, proving
\eqref{eq:compare-frozen}.
\end{proof}

\subsection{The auxiliary kernels \texorpdfstring{$K_n$}{K\_n} and \texorpdfstring{$H^{(n)}$}{H(n)}}
\label{sec:approx-kernel}

The discrete kernel $H_n$ and the continuous kernel $H$ are compared through two
\emph{auxiliary kernels}, which interpolate between them according to the hierarchy
\begin{equation}
\label{eq:hierarchy}
H\ \xleftrightarrow{\ \text{measure flow}\ }\ H^{(n)}
\ \xleftrightarrow{\ \text{time discretisation}\ }\ K_n
\ \xleftrightarrow{\ \text{lattice}\ }\ H_n .
\end{equation}
Each arrow of \eqref{eq:hierarchy} corresponds to a single source of error and is estimated
separately in Lemma~\ref{lem:kernel-flow}, Lemma~\ref{lem:kernel-time} and
Lemma~\ref{lem:kernel-lattice} below; Corollary~\ref{cor:approx-kernel} records the combination that
is used in Section~\ref{sec:difference}. A third auxiliary object, the grid embedding
$H^{\mathrm{g}}_n$ of $H$ (Lemma~\ref{lem:discr-kernel}), is used when comparing the expansion terms
$\tilde p* H^{*r}$ and $\tilde p* K_n^{\circledast r}$. Both bridges are discrete-flow objects: $H^{(n)}$
is the continuous-time kernel obtained by evaluating the \emph{continuous} generators at the
\emph{discrete} flow $\mu^{(n)}$, and $K_n$ is its discrete-time counterpart. Recall the generators
\eqref{eq:Lnu}--\eqref{eq:Ltilden} for a fixed measure flow $\nu$; $K_n$ is built from the
\emph{discrete} frozen density at the shifted index $\check j=\min(j+1,k-1)$, so that $H_n$ and
$K_n$ share the same frozen density $\tilde p_n(j+1,k)$ off the diagonal, the diagonal case $j=k-1$ being treated separately in
Lemma~\ref{lem:kernel-lattice}. Freezing at the index $\check j$ (equal to $j+1$ off the
diagonal and to $j$ on the diagonal), rather than at $j$ as in the classical continuous definition,
aligns the frozen density with the one-step transition of $H_n$: off the diagonal it removes the
$O\big(n^{-1}(\rho^{-1}+\rho^{-2})\big)$ mismatch between $\tilde p_n(j+1,k)$ and $\tilde p_n(j,k)$ from the leading
terms of the comparison, so that these leading terms coincide exactly with those of the expansion
of $H_n$ (see \eqref{eq:HnK} below):
\begin{equation}
\label{eq:Kndef}
\begin{split}
K_n(j,k;x,y)&:=(L^{\mu^{(n)}}-\tilde L^{\mu^{(n)}})\tilde p_n(\check j,k;x,y)\\
&=(b_x-b_y)\cdot\nabla_x\tilde p_n(\check j,k;x,y)\\
&\qquad+\tfrac12(\Sigma_x-\Sigma_y):\nabla^2_x\tilde p_n(\check j,k;x,y)
\end{split}
\end{equation}
where the generators are evaluated at time $s=t_j$, i.e.\ $b_x=b(t_j,x,\mu^{(n)}_j)$,
$b_y=b(t_j,y,\mu^{(n)}_j)$, $\Sigma_x=\Sigma(t_j,x,\mu^{(n)}_j)$,
$\Sigma_y=\Sigma(t_j,y,\mu^{(n)}_j)$, and ``:'' denotes the Frobenius inner product. We also define the
density of the centred Gaussian innovation of the Euler step from $x$,
\begin{equation}
\label{eq:qdef}
q^{(n)}_j(x,\theta)=g(0,\Sigma_x;\theta)
=(2\pi)^{-d/2}(\det\Sigma_x)^{-1/2}
\exp\Big[-\tfrac12\theta^{\top}\Sigma_x^{-1}\theta\Big],
\end{equation}
so that $q^{(n)}_j(x,\cdot)$ is the density of the Gaussian law $\mathcal N(0,\Sigma_x)$ and, in the
change of variables $z=x+n^{-1}\,b_x+n^{-1/2}\theta$,
$r^{(n)}_j(x,z)\dd z=q^{(n)}_j(x,\theta)\dd\theta$. To lighten the
notation we abbreviate $q_x=q^{(n)}_j(x,\cdot)$ and $q_y=q^{(n)}_j(y,\cdot)$; the subscript
indicates the freezing point ($x$ or $y$), while the dependence on the step index $j$ and on $n$ is
suppressed.

The continuous-time bridge $H^{(n)}$ is obtained from the continuous frozen density built
with the \emph{discrete} flow $\mu^{(n)}$, which is well defined through the continuous-time Euler
process \eqref{eq:cont-euler} (cf.\ Lemma~\ref{lem:flowrate}); with
\begin{equation}
\label{eq:mbarn-ndef}
\bar m^{(n)}(s,t,y)=\int_s^t b(u,y,\mu^{(n)}_u)\dd u,\qquad
\bar\Sigma^{(n)}(s,t,y)=\int_s^t\Sigma(u,y,\mu^{(n)}_u)\dd u
\end{equation}
and
\begin{multline}
\label{eq:pn-ndef}
\tilde p^{(n)}(s,t,x,y)=(2\pi)^{-d/2}\det\bar\Sigma^{(n)}(s,t,y)^{-1/2}\\
\times\exp\Big[-\tfrac12\big(y-x-\bar m^{(n)}(s,t,y)\big)^{\top}\bar\Sigma^{(n)}(s,t,y)^{-1}
\big(y-x-\bar m^{(n)}(s,t,y)\big)\Big],
\end{multline}
we set
\begin{equation}
\label{eq:Hn-ndef}
H^{(n)}(s,t,x,y)=(L^{\mu^{(n)}}-\tilde L^{\mu^{(n)}})\tilde p^{(n)}(s,t,x,y).
\end{equation}
Thus $\tilde p^{(n)}$ and $H^{(n)}$ are obtained from $\tilde p$ and $H$ by replacing the
true flow $\mu$ by the discrete flow $\mu^{(n)}$ in all coefficients; in particular
$H^{(n)}$ differs from $H$ only through the measure flow, while $K_n$ differs from
$H^{(n)}$ only through the discretisation of time. Following the convention of
Section~\ref{sec:frozen}, the continuous-time objects built with the discrete flow, such as $\tilde p^{(n)}(s,t,x,y)$ and $H^{(n)}(s,t,x,y)$, carry the superscript $(n)$ and are written with commas.

\begin{lemma}[Lattice comparison]
\label{lem:kernel-lattice}
Under Assumptions~\ref{ass:spatial} and \ref{ass:elliptic}, there are constants
$C<\infty$ and $c>0$, independent of $n$, such that, for all $0\le j<k\le n$, $x,y\in\R^d$, and
$\rho=\rho_{jk}$,
\begin{equation}
\label{eq:lattice-bound}
\abs{H_n(j,k;x,y)-K_n(j,k;x,y)}
\le C\bigl[n^{-1/2}\rho^{-1}+n^{-1}\rho^{-3}\bigr]\varphi_{c,\rho}(y-x).
\end{equation}
\end{lemma}

\begin{proof}
Both kernels are built with the same discrete flow $\mu^{(n)}$, so no measure-flow discrepancy
arises, and the estimate is a purely local one-step computation. We treat the two ranges
$j<k-1$ and $j=k-1$ separately. Throughout, $C$ changes from line to line and is independent of
$n,j,k,x,y$; we abbreviate $\varphi:=\varphi_{c,\rho}(y-x)$ with $c$ possibly enlarged from line
to line, and recall $\rho=\rho_{jk}$.

\textbf{Step 1: the range $j<k-1$.}
Here the frozen index is $\check j=j+1$. Write
\begin{equation}
\chi(\cdot):=\tilde p_n(j+1,k;\cdot,y),\qquad
w_x:=n^{-1}b_x+n^{-1/2}\theta,\qquad w_y:=n^{-1}b_y+n^{-1/2}\theta .
\end{equation}
Treating each of the two terms of \eqref{eq:Hn-exact} by its own change of variables ---
$z=x+n^{-1}b_x+n^{-1/2}\theta$ for the term carrying $r^{(n)}_j$ and
$z=x+n^{-1}b_y+n^{-1/2}\theta$ for the term carrying $\tilde r^{(n),y}_j$ --- the Jacobian
$n^{-d/2}$ of each substitution cancels against the factor $n^{d/2}$ carried by the corresponding
transition density, and we obtain the exact representation
\begin{equation}
\label{eq:exact}
H_n(j,k;x,y)=n\int_{\R^d}\Big[q_x(\theta)\,\chi(x+w_x)-q_y(\theta)\,\chi(x+w_y)\Big]\dd\theta .
\end{equation}
Taylor expanding $\chi$ about $x$ in the increment $w=n^{-1}b+n^{-1/2}\theta$ gives the complete
expansion
\begin{equation}
\label{eq:taylor2}
\begin{aligned}
\chi(x+w)={}&\chi(x)+n^{-1}\,b\cdot\nabla \chi(x)+n^{-1/2}\,\theta\cdot\nabla \chi(x)
+\tfrac{n^{-2}}{2}(b\otimes b):\nabla^{2} \chi(x)\\
&+\tfrac{n^{-3/2}}{2}(b\otimes\theta+\theta\otimes b):\nabla^{2} \chi(x)
+\tfrac{n^{-1}}{2}(\theta\otimes\theta):\nabla^{2} \chi(x)
+R_{3}(x,w).
\end{aligned}
\end{equation}
Inserting \eqref{eq:taylor2} with $b=b_x$ and with $b=b_y$ into \eqref{eq:exact}, and using the
Gaussian moments
\begin{equation}
\int q_x=\int q_y=1,\qquad \int q_x\,\theta=\int q_y\,\theta=0,\qquad
\int q_x(\theta\otimes\theta)=\Sigma_x,\qquad \int q_y(\theta\otimes\theta)=\Sigma_y,
\end{equation}
together with the vanishing of the odd third moments
$\int q_x\theta^{\otimes3}=\int q_y\theta^{\otimes3}=0$, the terms $\chi(x)$ and
$n^{-1/2}\theta\cdot\nabla \chi(x)$ cancel identically, the frozen drift contributes
$(b_x-b_y)\cdot\nabla \chi(x)$, and the second-order terms contribute
$\tfrac12(\Sigma_x-\Sigma_y):\nabla^2 \chi(x)$ (from the $n^{-1}\theta^{\otimes2}$ part) and
$\tfrac{n^{-1}}2(b_x\otimes b_x-b_y\otimes b_y):\nabla^2 \chi(x)$ (from the $n^{-2}b^{\otimes2}$ part,
after the overall factor $n$); the cross term is odd in $\theta$ and integrates to $0$. Denoting
by
\begin{equation}
\label{eq:R3-def}
\mathcal R_3:=n\int_{\R^d}\big[q_x(\theta)\,R_3(x,w_x)-q_y(\theta)\,R_3(x,w_y)\big]\dd\theta
\end{equation}
the integrated Taylor remainder, this yields
\begin{equation}
\label{eq:Hn-expand}
H_n(j,k;x,y)=(b_x-b_y)\cdot\nabla \chi(x)+\tfrac12(\Sigma_x-\Sigma_y):\nabla^2 \chi(x)
+\tfrac{n^{-1}}2(b_x\otimes b_x-b_y\otimes b_y):\nabla^2 \chi(x)+\mathcal R_3 .
\end{equation}
Since $\check j=j+1$ and $\tilde p_n(j+1,k;\cdot,y)=\chi$, the two leading terms of
\eqref{eq:Hn-expand} are exactly the auxiliary kernel $K_n$ of \eqref{eq:Kndef}, so that
\eqref{eq:Hn-expand} reads
\begin{equation}
\label{eq:HnK}
H_n(j,k;x,y)-K_n(j,k;x,y)
=\tfrac{n^{-1}}2(b_x\otimes b_x-b_y\otimes b_y):\nabla^2 \chi(x)+\mathcal R_3 .
\end{equation}
We bound the two terms on the right-hand side of \eqref{eq:HnK} in turn.

\emph{The prefactor term.} The frozen density $\chi=\tilde p_n(j+1,k;\cdot,y)$ is Gaussian; writing
$\bar m_n:=\bar m_n(j+1,k;y)$ and $\bar\Sigma_n:=\bar\Sigma_n(j+1,k;y)$ for its mean and covariance,
we have $\bar\Sigma_n\succeq c_0\rho_{j+1,k}^2I$, and $j<k-1$ gives
$\rho_{j+1,k}\le\rho\le2\rho_{j+1,k}$. By Lemma~\ref{lem:kernel}(iv), after absorbing the mean
shift $\bar m_n=O(\rho^2)$ into $c$ (as in the proof of Lemma~\ref{lem:kernbound}),
\begin{equation}
\abs{\partial^\alpha_x \chi(x)}\le C_\alpha\,\rho^{-\abs{\alpha}}\,\varphi_{c,\rho}(y-x),
\qquad \abs{\alpha}\le2,
\end{equation}
so that, using $\norm{b_x\otimes b_x-b_y\otimes b_y}\le C\norm{x-y}$ (Assumption~\ref{ass:spatial})
and the weight bound $\norm{x-y}\varphi_{c,\rho}\le C\rho\,\varphi_{c+1,\rho}$
(Lemma~\ref{lem:kernel}(ii)),
\begin{equation}
\abs{\tfrac{n^{-1}}2(b_x\otimes b_x-b_y\otimes b_y):\nabla^2 \chi(x)}
\le Cn^{-1}\rho^{-1}\varphi\le Cn^{-1/2}\rho^{-1}\varphi,
\end{equation}
the last step because $n\ge1$, so that $n^{-1}\le n^{-1/2}$.

\emph{The remainder $\mathcal R_3$.} By Taylor's formula with integral remainder, the third-order
term in \eqref{eq:taylor2} is, with $w=n^{-1}b+n^{-1/2}\theta$,
\begin{equation}
\label{eq:R3-integral}
R_3(x,w)=\frac12\int_0^1(1-\tau)^2\,(w\otimes w\otimes w):\nabla^3 \chi(x+\tau w)\,\dd\tau,
\end{equation}
where the triple contraction is
$(u\otimes v\otimes w):\nabla^3 \chi=\sum_{i,j,l}u_iv_jw_l\,\partial^3_{x_ix_jx_l}\chi$. Substituting
$w=n^{-1}b+n^{-1/2}\theta$ and using the symmetry of $\nabla^3 \chi$,
\begin{multline}
\label{eq:cubic}
(w\otimes w\otimes w):\nabla^3 \chi
=n^{-3}(b\otimes b\otimes b):\nabla^3 \chi
+3n^{-5/2}(b\otimes b\otimes\theta):\nabla^3 \chi\\
+3n^{-2}(b\otimes\theta\otimes\theta):\nabla^3 \chi
+n^{-3/2}(\theta\otimes\theta\otimes\theta):\nabla^3 \chi .
\end{multline}
The second and the last terms on the right-hand side of \eqref{eq:cubic} are odd in $\theta$, while
$q_x,q_y$ are centred Gaussians; were $\nabla^3 \chi$ evaluated at the base point $x$, they would
integrate to $0$ against $q_x$ and $q_y$, and only the first and the third terms would survive.
Here, however, $\nabla^3 \chi$ is evaluated at the shifted point $x+\tau w$, which depends on
$\theta$, so the cancellation is only partial: these two terms give a non-vanishing remainder
$\mathcal R_3^{\mathrm{odd}}$ of order $n^{-1}\rho^{-3}\varphi_{c,\rho}(y-x)$, which is quantified
after \eqref{eq:drift-bound}. Inserting the first and the third terms into \eqref{eq:R3-integral} and
\eqref{eq:R3-def}, the remainder splits as
\begin{equation}
\label{eq:R3-split}
\mathcal R_3=\mathcal R_3^{\mathrm{cross}}+\mathcal R_3^{\mathrm{drift}}+\mathcal R_3^{\mathrm{odd}},
\end{equation}
where
\begin{multline}
\label{eq:cross-int}
\mathcal R_3^{\mathrm{cross}}
=\frac{3n^{-1}}2\int_0^1(1-\tau)^2\int_{\R^d}
\Bigl[q_x(\theta)\,(b_x\otimes\theta\otimes\theta):\nabla^3 \chi(x+\tau w_x)\\
-q_y(\theta)\,(b_y\otimes\theta\otimes\theta):\nabla^3 \chi(x+\tau w_y)\Bigr]\dd\theta\,\dd\tau
\end{multline}
collects the $3n^{-2}(b\otimes\theta\otimes\theta)$ contribution, and
\begin{multline}
\label{eq:drift-int}
\mathcal R_3^{\mathrm{drift}}
=\frac{n^{-2}}2\int_0^1(1-\tau)^2\int_{\R^d}
\Bigl[q_x(\theta)\,(b_x\otimes b_x\otimes b_x):\nabla^3 \chi(x+\tau w_x)\\
-q_y(\theta)\,(b_y\otimes b_y\otimes b_y):\nabla^3 \chi(x+\tau w_y)\Bigr]\dd\theta\,\dd\tau
\end{multline}
the $n^{-3}(b\otimes b\otimes b)$ contribution. The numerical factors $\frac{3n^{-1}}2$ and
$\frac{n^{-2}}2$ are the products of the factor $n$ in \eqref{eq:R3-def}, the factor $\frac12$ in
\eqref{eq:R3-integral} and the coefficients $3n^{-2}$ and $n^{-3}$ in \eqref{eq:cubic}.

To bound \eqref{eq:cross-int} and \eqref{eq:drift-int} we first estimate the derivatives at the
shifted point. Since $\chi=g(\bar m_n,\bar\Sigma_n;y-\cdot)$ is Gaussian with $\bar m_n=O(\rho^2)$ and
$\bar\Sigma_n\succeq c_0\rho_{j+1,k}^2I$, Lemma~\ref{lem:kernel}(iv) gives, for $\abs\alpha=3$,
\begin{equation}
\label{eq:shift-deriv}
\abs{\partial^\alpha \chi(x+\tau w)}\le C\,\rho^{-3}\,\varphi_{c,\rho}(y-x-\tau w-\bar m_n).
\end{equation}
The shift $\tau w$ and the mean $\bar m_n$ are now removed from the weight: with
$\norm{w}^2\le2n^{-2}\norm{b}^2+2n^{-1}\norm{\theta}^2$, $\norm{\bar m_n}\le C\rho^2$, $n^{-1}\le\rho^2$
and Young's inequality $\norm{a-b-c}^2\ge\tfrac12\norm{a}^2-4(\norm{b}^2+\norm{c}^2)$, enlarging
$c$ if necessary,
\begin{equation}
\label{eq:shift-absorb}
\varphi_{c,\rho}(y-x-\tau w-\bar m_n)
\le C\,\varphi_{2c,\rho}(y-x)\,\exp\Bigl(\tfrac{8}{c}\norm{\theta}^{2}\Bigr).
\end{equation}
Since $q_x,q_y\le Ce^{-c_1\norm{\theta}^2}$ for some $c_1>0$ (uniform ellipticity of
$\Sigma_x,\Sigma_y$), enlarging $c$ further so that $8/c<c_1$ turns \eqref{eq:shift-absorb} into
the integrable bound
\begin{equation}
\label{eq:shift-int}
\int_{\R^d}q_x(\theta)\,\norm{\theta}^{m}\,\varphi_{c,\rho}(y-x-\tau w_x-\bar m_n)\,\dd\theta
\le C_m\,\varphi_{c',\rho}(y-x),\qquad m\ge0,
\end{equation}
and likewise with $q_y,w_y$ in place of $q_x,w_x$.

We can now bound \eqref{eq:cross-int}. Its two integrands differ only through the freezing point
($x$ versus $y$): the coefficient $b_x$ versus $b_y$, the weight $q_x$ versus $q_y$ (i.e.\
$\Sigma_x$ versus $\Sigma_y$) and the shift $w_x-w_y=n^{-1}(b_x-b_y)$ are each $O(\norm{x-y})$ by
Assumption~\ref{ass:spatial}, uniformly in $\theta$ up to the integrable factor
$(1+\norm{\theta}^2)e^{-c\norm{\theta}^2}$. Hence, by \eqref{eq:shift-deriv}--\eqref{eq:shift-int},
the bracket in \eqref{eq:cross-int} is $O\big(\norm{x-y}\rho^{-3}\varphi_{c',\rho}(y-x)\big)$, and
therefore
\begin{equation}
\label{eq:cross-bound}
\abs{\mathcal R_3^{\mathrm{cross}}}
\le C\,n^{-1}\norm{x-y}\rho^{-3}\varphi_{c',\rho}(y-x)
\le C\,n^{-1}\rho^{-2}\varphi
\le C\,n^{-1/2}\rho^{-1}\varphi,
\end{equation}
where we used $\norm{x-y}\varphi_{c',\rho}\le C\rho\,\varphi_{c'',\rho}$ and $\rho\ge n^{-1/2}$.
Similarly, since $b_x\otimes b_x\otimes b_x$ is bounded and Lipschitz in $x$, the two integrands of
\eqref{eq:drift-int} differ by $O(\norm{x-y})$, so that
\begin{equation}
\label{eq:drift-bound}
\abs{\mathcal R_3^{\mathrm{drift}}}
\le C\,n^{-2}\norm{x-y}\rho^{-3}\varphi
\le C\,n^{-2}\rho^{-2}\varphi
\le C\,n^{-1/2}\rho^{-1}\varphi .
\end{equation}

\emph{The odd term $\mathcal R_3^{\mathrm{odd}}$.} It remains to control the contribution of the
second and the last terms of \eqref{eq:cubic}. Because $\nabla^3 \chi$ is evaluated at the shifted
point $x+\tau w$, which depends on $\theta$, they do not vanish; we estimate them by symmetrising
each $\theta$-integral. Write $w_i(\vartheta)=n^{-1}b_i+n^{-1/2}\vartheta$ ($i=x,y$), so that
$w_i(-\vartheta)=w_i(\vartheta)-2n^{-1/2}\vartheta$ and
$\norm{w_i(\vartheta)-w_i(-\vartheta)}=2n^{-1/2}\norm{\vartheta}$. The two surviving tensors of
\eqref{eq:cubic} are $\theta\otimes\theta\otimes\theta$ and $b_i\otimes b_i\otimes\theta$; both are
odd as functions of $\theta$, and since $q_x,q_y$ are even, for either of them --- written
$T_i(\theta)$ --- we have
\begin{equation}
\label{eq:odd-sym}
\int_{\R^d}q_i(\theta)\,T_i(\theta)\!:\!\nabla^3 \chi(x+\tau w_i(\theta))\,\dd\theta
=\tfrac12\int_{\R^d}q_i(\theta)\,T_i(\theta)\!:\!\bigl[\nabla^3 \chi(x+\tau w_i(\theta))-\nabla^3 \chi(x+\tau w_i(-\theta))\bigr]\dd\theta
\end{equation}
for $i=x,y$. By Lemma~\ref{lem:kernel}(iv) the bracket is
$O\big(\tau n^{-1/2}\norm{\theta}\rho^{-4}\varphi_{c,\rho}(y-x-\tau w_i(\theta)-\bar m_n)\big)$, and
$\int_{\R^d}q_i(\theta)\norm{\theta}^{\,j}\varphi_{c,\rho}(y-x-\tau w_i(\theta)-\bar m_n)\dd\theta\le C\varphi_{c',\rho}(y-x)$
for $j\le4$, by \eqref{eq:shift-absorb} together with $q_i(\theta)\le Ce^{-c_1\norm\theta^2}$ and
$8/c<c_1$. Hence, for $T_i(\theta)=\theta\otimes\theta\otimes\theta$ and
$T_i(\theta)=b_i\otimes b_i\otimes\theta$,
\begin{equation}
\label{eq:odd-symbound}
\Bigl|\int_{\R^d}q_i(\theta)\,T_i(\theta)\!:\!\nabla^3 \chi(x+\tau w_i(\theta))\,\dd\theta\Bigr|
\le C\,\tau\,n^{-1/2}\rho^{-4}\varphi_{c',\rho}(y-x),\qquad i=x,y .
\end{equation}
The two surviving terms of \eqref{eq:cubic} thus contribute
\begin{multline}
\label{eq:R3odd-int}
\mathcal R_3^{\mathrm{odd}}
=\tfrac{n^{-1/2}}2\int_0^1(1-\tau)^2\int_{\R^d}
\Bigl[q_x(\theta)\,(\theta\otimes\theta\otimes\theta):\nabla^3 \chi(x+\tau w_x)
-q_y(\theta)\,(\theta\otimes\theta\otimes\theta):\nabla^3 \chi(x+\tau w_y)\Bigr]\dd\theta\,\dd\tau\\
+\tfrac{3n^{-3/2}}2\int_0^1(1-\tau)^2\int_{\R^d}
\Bigl[q_x(\theta)\,(b_x\otimes b_x\otimes\theta):\nabla^3 \chi(x+\tau w_x)
-q_y(\theta)\,(b_y\otimes b_y\otimes\theta):\nabla^3 \chi(x+\tau w_y)\Bigr]\dd\theta\,\dd\tau .
\end{multline}
Each bracket is $O\big(\norm{x-y}\tau n^{-1/2}\rho^{-4}\varphi\big)$: the two symmetrised
integrals in $i=x$ and $i=y$ are each $O(\tau n^{-1/2}\rho^{-4}\varphi)$ by \eqref{eq:odd-symbound},
and they differ by $O(\norm{x-y})$, because the densities $q_x-q_y=O(\norm{x-y})(1+\norm{\theta}^2)e^{-c\norm{\theta}^2}$
and the bracket $\nabla^3 \chi(x+\tau w_i(\theta))-\nabla^3 \chi(x+\tau w_i(-\theta))$ is, up to an
$O(\norm{x-y})$ error, the same function of $\theta$ for $i=x$ and $i=y$ (the shifts $w_x,w_y$ differ
by $n^{-1}\norm{b_x-b_y}=O(n^{-1}\norm{x-y})$). Using
$\int_0^1\tau(1-\tau)^2\dd\tau=\tfrac1{12}$, $\norm{x-y}\varphi_{c',\rho}\le C\rho\,\varphi$ and
$\rho\ge n^{-1/2}$,
\begin{equation}
\label{eq:R3odd-bound}
\abs{\mathcal R_3^{\mathrm{odd}}}
\le C\bigl[n^{-1}+n^{-2}\bigr]\norm{x-y}\rho^{-4}\varphi
\le C\,n^{-1}\rho^{-3}\varphi .
\end{equation}

Feeding \eqref{eq:cross-bound}, \eqref{eq:drift-bound} and \eqref{eq:R3odd-bound} into \eqref{eq:HnK},
together with the bound of the prefactor term, completes the proof in the range $j<k-1$:
\begin{equation}
\abs{H_n(j,k;x,y)-K_n(j,k;x,y)}
\le C\bigl[n^{-1/2}\rho^{-1}+n^{-1}\rho^{-3}\bigr]\varphi_{c,\rho}(y-x),
\end{equation}
which is \eqref{eq:lattice-bound}.

\textbf{Step 2: the diagonal $j=k-1$.}
Here $\rho^2=n^{-1}$ and $\chi=\tilde p_n(k,k;\cdot,y)=\delta_y$, so \eqref{eq:Hn-exact} collapses
to
\begin{equation}
\label{eq:Hdiag}
H_n(k-1,k;x,y)=n\big[r^{(n)}_{k-1}(x,y)-\tilde r^{(n),y}_{k-1}(x,y)\big].
\end{equation}
Setting $\theta=n^{1/2}(y-x)$ and $u=\theta-n^{-1/2}b_y$, the two one-step densities are
\begin{equation}
r^{(n)}_{k-1}(x,y)=n^{d/2}q_x\big(u+n^{-1/2}(b_y-b_x)\big),
\qquad
\tilde r^{(n),y}_{k-1}(x,y)=n^{d/2}q_y(u),
\end{equation}
since $r^{(n)}_{k-1}(x,\cdot)$ is centred at $x+n^{-1}b_x$ with covariance $n^{-1}\Sigma_x$, and
$\tilde r^{(n),y}_{k-1}$ is the same object with $b_y,\Sigma_y$. Substituting these into
\eqref{eq:Hdiag},
\begin{equation}
\label{eq:Hdiag2}
H_n=n^{1+d/2}\Big[q_x\big(u+n^{-1/2}(b_y-b_x)\big)-q_y(u)\Big].
\end{equation}
We expand the bracket in \eqref{eq:Hdiag2} step by step. First, Taylor expanding $q_x$ about $u$
in the shift $n^{-1/2}(b_y-b_x)$,
\begin{equation}
\label{eq:qshift}
q_x\big(u+n^{-1/2}(b_y-b_x)\big)
=q_x(u)+n^{-1/2}(b_y-b_x)\cdot\nabla q_x(u)
+O\big(n^{-1}\norm{x-y}^2(1+\norm{u}^2)e^{-c\norm{u}^2}\big).
\end{equation}
Second, expanding $q_x-q_y$ in $\Sigma_x-\Sigma_y$ and using $\partial_\Sigma q=\tfrac12\nabla^2q$,
\begin{equation}
\label{eq:qdiff}
q_x(u)-q_y(u)=\tfrac12(\Sigma_x-\Sigma_y):\nabla^2 q_y(u)
+O\big(\norm{x-y}^2(1+\norm{u}^4)e^{-c\norm{u}^2}\big).
\end{equation}
Third, combining \eqref{eq:qshift} and \eqref{eq:qdiff} and multiplying by $n^{1+d/2}$, we split
off from the drift term the part that is common with the auxiliary kernel,
\begin{equation}
(b_y-b_x)\cdot\nabla q_x(u)
=(b_y-b_x)\cdot\nabla q_y(u)+(b_y-b_x)\cdot\big(\nabla q_x(u)-\nabla q_y(u)\big),
\end{equation}
and obtain
\begin{equation}
\label{eq:Hdiag3}
H_n=\tfrac{n^{(d+2)/2}}2(\Sigma_x-\Sigma_y):\nabla^2 q_y(u)
+n^{(d+1)/2}(b_y-b_x)\cdot\nabla q_y(u)+\mathcal S_n,
\end{equation}
where the higher-order leftover is
\begin{equation}
\label{eq:Hdiag-rem}
\mathcal S_n=n^{(d+1)/2}(b_y-b_x)\cdot\big(\nabla q_x(u)-\nabla q_y(u)\big)
+O\big(n^{d/2}(1+\norm{\theta}^2)(1+\norm{u}^4)e^{-c\norm{u}^2}\big).
\end{equation}

It remains to compare \eqref{eq:Hdiag3} with $K_n$. On the diagonal $\check j=k-1$, so the
auxiliary kernel \eqref{eq:Kndef} is built from the one-step frozen density
$\tilde p_n(k-1,k;\cdot,y)=n^{d/2}q_y(u)$ rather than from $\delta_y$; since
$u=n^{1/2}(y-x)-n^{-1/2}b_y$ gives $\nabla_x=-n^{1/2}\nabla_u$ and $\nabla^2_x=n\nabla^2_u$, we
have $\nabla_x\tilde p_n(k-1,k)=-n^{(d+1)/2}\nabla_u q_y(u)$ and
$\nabla^2_x\tilde p_n(k-1,k)=n^{(d+2)/2}\nabla^2_u q_y(u)$, hence
\begin{equation}
\label{eq:Kdiag}
K_n(k-1,k;x,y)=\tfrac{n^{(d+2)/2}}2(\Sigma_x-\Sigma_y):\nabla^2 q_y(u)
-n^{(d+1)/2}(b_x-b_y)\cdot\nabla q_y(u).
\end{equation}
Comparing \eqref{eq:Hdiag3} with \eqref{eq:Kdiag}, the covariance terms cancel exactly and the
drift term $n^{(d+1)/2}(b_y-b_x)\cdot\nabla q_y(u)$ of \eqref{eq:Hdiag3} is exactly the drift term
of \eqref{eq:Kdiag}; therefore
\begin{equation}
H_n(k-1,k;x,y)-K_n(k-1,k;x,y)=\mathcal S_n .
\end{equation}
Finally, since $\norm{b_y-b_x}=O(\norm{x-y})$,
$\nabla q_x(u)-\nabla q_y(u)=O\big(\norm{x-y}(1+\norm{u}^3)e^{-c\norm{u}^2}\big)$ and
$\norm{x-y}=n^{-1/2}\norm{\theta}$, the leftover \eqref{eq:Hdiag-rem} obeys
\begin{equation}
\abs{\mathcal S_n}\le C\,n^{d/2}(1+\norm{\theta}^2)(1+\norm{u}^4)e^{-c\norm{u}^2}
\le C\,n^{d/2}e^{-\tilde c\norm{\theta}^2},
\end{equation}
where the last bound follows by absorbing the polynomial factor into the enlarged constant
$\tilde c>0$, using $\norm{u}=\norm{\theta-n^{-1/2}b_y}\le\norm{\theta}+1$. On the diagonal
$\rho=n^{-1/2}$, so $n^{-1/2}\rho^{-1}=1$, and $\varphi_{\tilde c,\rho}(y-x)=n^{d/2}e^{-\norm{\theta}^2/\tilde c}$;
hence, after relabelling the constant $\tilde c>0$,
\begin{equation}
\abs{H_n(k-1,k;x,y)-K_n(k-1,k;x,y)}
\le C\,n^{-1/2}\rho^{-1}\varphi_{\tilde c,\rho}(y-x).
\end{equation}
Renaming $\tilde c$ as $c$, both ranges $j<k-1$ and $j=k-1$ give the bound \eqref{eq:lattice-bound}:
in the range $j<k-1$ this is the conclusion of Step~1, while on the diagonal
$\rho=n^{-1/2}$, so $n^{-1/2}\rho^{-1}=1$ and the bound $Cn^{-1/2}\rho^{-1}\varphi_{c,\rho}(y-x)$
just obtained is stronger than $C[n^{-1/2}\rho^{-1}+n^{-1}\rho^{-3}]\varphi_{c,\rho}(y-x)$, and hence implies it. This proves the lemma.
\end{proof}

\begin{lemma}[Flow comparison]
\label{lem:kernel-flow}
Under Assumptions~\ref{ass:spatial}--\ref{ass:measLip}, there are constants
$C<\infty$ and $c>0$, independent of $n$, such that, for all $0\le j<k\le n$, $x,y\in\R^d$, and
$\rho=\rho_{jk}$,
\begin{equation}
\abs{H(t_j,t_k,x,y)-H^{(n)}(t_j,t_k,x,y)}
\le C\,n^{-1/2}\rho^{-1}\varphi_{c,\rho}(y-x).
\end{equation}
\end{lemma}

\begin{proof}
Here $s=t_j$ and $t=t_k$, so that $\rho^2=t-s=t_k-t_j$; by \eqref{eq:kernelH} both $H$ and
$H^{(n)}$ are of the form
$\frac12\sum_{i,j}\mathsf{A}_{ij}\partial^2_{x_ix_j}\tilde p+\sum_i\mathsf{B}_i\partial_{x_i}\tilde p$, with
$\mathsf{A}_{ij}=\Sigma_{ij}(s,x,\mu_s)-\Sigma_{ij}(s,y,\mu_s)$, $\mathsf{B}_i=b_i(s,x,\mu_s)-b_i(s,y,\mu_s)$, and
$\mathsf{A}^{(n)}_{ij},\mathsf{B}^{(n)}_i$ their $\mu^{(n)}$-analogues. Writing
$\mathsf{A}_{ij}\partial^2\tilde p-\mathsf{A}^{(n)}_{ij}\partial^2\tilde p^{(n)}
=(\mathsf{A}_{ij}-\mathsf{A}^{(n)}_{ij})\partial^2\tilde p+\mathsf{A}^{(n)}_{ij}\big(\partial^2\tilde p-\partial^2\tilde p^{(n)}\big)$
(and similarly for $\mathsf{B}_i$) gives two families of terms. By the mean-value theorem and
Assumption~\ref{ass:measLip} (with $\abs{\alpha}=1$),
\begin{equation}
\abs{\mathsf{A}_{ij}-\mathsf{A}^{(n)}_{ij}}\le \Lambda\,\mathcal{W}_2\big(\mu_s,\mu^{(n)}_s\big)\,\norm{x-y},\qquad
\abs{\mathsf{B}_i-\mathsf{B}^{(n)}_i}\le \Lambda\,\mathcal{W}_2\big(\mu_s,\mu^{(n)}_s\big)\,\norm{x-y},
\end{equation}
while $\abs{\mathsf{A}^{(n)}_{ij}},\abs{\mathsf{B}^{(n)}_i}\le C\norm{x-y}$ by Assumption~\ref{ass:spatial}. On the
other hand, $\tilde p$ and $\tilde p^{(n)}$ are Gaussians with means and covariances differing
by $\norm{\bar m-\bar m^{(n)}},\norm{\bar\Sigma-\bar\Sigma^{(n)}}\le C\rho^2 n^{-1/2}$ (by
Assumption~\ref{ass:measLip} and \eqref{eq:flowrate}), so by the Gaussian comparison of
Lemma~\ref{lem:kernel}(v),
\begin{equation}
\abs{\partial^\alpha_x\tilde p-\partial^\alpha_x\tilde p^{(n)}}
\le C_\alpha n^{-1/2}\big[\rho^{1-\abs{\alpha}}+\rho^{-\abs{\alpha}}\big]\varphi_{c,\rho}(y-x),
\qquad \abs{\alpha}=1,2.
\end{equation}
Assembling the two families, using the derivative bounds of Lemma~\ref{lem:kernel}(iv) for the
undifferentiated factor, the weight bound of Lemma~\ref{lem:kernel}(ii), and $\rho\le1$,
\begin{equation}
\abs{H-H^{(n)}}\le C\,\mathcal{W}_2\big(\mu_s,\mu^{(n)}_s\big)\,\rho^{-1}\varphi_{c,\rho}(y-x)
+C\,n^{-1/2}\rho^{-1}\varphi_{c,\rho}(y-x)
\le C\,n^{-1/2}\rho^{-1}\varphi_{c,\rho}(y-x),
\end{equation}
by \eqref{eq:flowrate}. This proves the lemma.
\end{proof}

\begin{lemma}[Time-discretisation comparison]
\label{lem:kernel-time}
Under Assumptions~\ref{ass:spatial}--\ref{ass:measLip}, there are
constants $C<\infty$ and $c>0$, independent of $n$, such that, for all $0\le j<k\le n$,
$x,y\in\R^d$, and $\rho=\rho_{jk}$,
\begin{equation}
\abs{H^{(n)}(t_j,t_k,x,y)-K_n(j,k;x,y)}
\le C\,n^{-1/2}\rho^{-1}\varphi_{c,\rho}(y-x)+C\,n^{-1}\rho^{-3}\varphi_{c,\rho}(y-x).
\end{equation}
\end{lemma}

\begin{proof}
Since $H^{(n)}$ and $K_n$ share the flow $\mu^{(n)}$,
\begin{equation}
H^{(n)}-K_n=(L^{\mu^{(n)}}-\tilde L^{\mu^{(n)}})\big(\tilde p^{(n)}-\tilde p_n\big),
\end{equation}
where $\tilde p^{(n)}=\tilde p^{(n)}(t_j,t_k,\cdot,y)$ and $\tilde p_n=\tilde p_n(\check j,k;\cdot,y)$.
Splitting the shift of the initial time from $t_j$ to $t_{\check j}$ off the Riemann-sum
discretisation, and using \eqref{eq:mbarn-ndef} and \eqref{eq:frozen-disc},
\[
\begin{aligned}
\bar m^{(n)}(t_j,t_k,y)-\bar m_n(\check j,k;y)
&=\int_{t_j}^{t_{\check j}}b(u,y,\mu^{(n)}_u)\dd u\\
&\quad+\Big[\int_{t_{\check j}}^{t_k}b(u,y,\mu^{(n)}_u)\dd u
-\sum_{i=\check j}^{k-1}h\,b(t_i,y,\mu^{(n)}_i)\Big],\\
\bar\Sigma^{(n)}(t_j,t_k,y)-\bar\Sigma_n(\check j,k;y)
&=\int_{t_j}^{t_{\check j}}\Sigma(u,y,\mu^{(n)}_u)\dd u\\
&\quad+\Big[\int_{t_{\check j}}^{t_k}\Sigma(u,y,\mu^{(n)}_u)\dd u
-\sum_{i=\check j}^{k-1}h\,\Sigma(t_i,y,\mu^{(n)}_i)\Big],
\end{aligned}
\]
where the two single integrals are bounded by $Cn^{-1}$ (since $t_{\check j}-t_j\le h=n^{-1}$) and
vanish when $\check j=j$, while the bracketed terms are left-endpoint Riemann-sum errors of order
$C\rho^2n^{-1/2}$ by Remark~\ref{rem:autotime}. Hence
\begin{equation}
\norm{\bar m^{(n)}-\bar m_n},\ \norm{\bar\Sigma^{(n)}-\bar\Sigma_n}\le C\big(\rho^2 n^{-1/2}+n^{-1}\big).
\end{equation}
Lemma~\ref{lem:kernel}(v) then gives, for $\abs{\alpha}=1,2$,
\begin{equation}
\abs{\partial^\alpha_x\big(\tilde p^{(n)}-\tilde p_n\big)}
\le C_\alpha\Big[(\rho^2 n^{-1/2}+n^{-1})\rho^{-1}+(\rho^2 n^{-1/2}+n^{-1})\rho^{-2}\Big]
\rho^{-\abs{\alpha}}\varphi_{c,\rho}(y-x).
\end{equation}
Applying the operator $(L^{\mu^{(n)}}-\tilde L^{\mu^{(n)}})$, whose coefficients are bounded by
$C\norm{x-y}$ (Assumption~\ref{ass:spatial}), together with the weight bound
$\norm{x-y}\varphi_{c,\rho}\le C\rho\,\varphi_{c',\rho}$ (Lemma~\ref{lem:kernel}(ii)), the
dominant terms are $n^{-1/2}\rho^{-1}$ (from the discretisation error) and $n^{-1}\rho^{-3}$ (from the
$O(n^{-1})$ initial-time shift acting through $\nabla^2$), so that
\begin{equation}
\abs{H^{(n)}-K_n}\le C\,n^{-1/2}\rho^{-1}\varphi_{c,\rho}(y-x)
+C\,n^{-1}\rho^{-3}\varphi_{c,\rho}(y-x).
\end{equation}
This proves the lemma.
\end{proof}

\begin{corollary}[Triangle-inequality combination]
\label{cor:approx-kernel}
Under Assumptions~\ref{ass:spatial}--\ref{ass:measLip}, there
are constants $C<\infty$ and $c>0$, independent of $n$, such that, for all $0\le j<k\le n$,
$x,y\in\R^d$, and $\rho=\rho_{jk}$,
\begin{equation}
\label{eq:HmKn-bound}
\abs{H(t_j,t_k,x,y)-K_n(j,k;x,y)}
\le C\,n^{-1/2}\rho^{-1}\varphi_{c,\rho}(y-x)+C\,n^{-1}\rho^{-3}\varphi_{c,\rho}(y-x).
\end{equation}
\end{corollary}

\begin{proof}
By the triangle inequality along the hierarchy \eqref{eq:hierarchy} and
Lemmas~\ref{lem:kernel-flow}--\ref{lem:kernel-time},
$\abs{H-K_n}\le\abs{H-H^{(n)}}+\abs{H^{(n)}-K_n}$, and the right-hand side is
bounded by the sum of the two stated estimates.
\end{proof}

\subsection{Bounds on the discrete kernels and discrete propagation}
\label{sec:disc-kernel-bounds}

\begin{lemma}
\label{lem:disc-kernel-bounds}
Under Assumptions~\ref{ass:spatial} and \ref{ass:elliptic}, there are constants
$C_1<\infty$ (the constant of Lemma~\ref{lem:kernbound}, enlarged if necessary) and $c>0$
such that, for all $0\le j<k\le n$, $x,y\in\R^d$, and $\rho=\rho_{jk}$,
\begin{enumerate}[label=(\roman*),leftmargin=2.2em]
\item $\abs{K_n(j,k;x,y)}+\abs{H_n(j,k;x,y)}\le C_1\,\rho_{jk}^{-1}\varphi_{c,\rho_{jk}}(y-x)$;
\item \emph{(discrete propagation estimate)} if $G$ and $F_1,\dots,F_r$ ($r\ge0$) satisfy, for all
$0\le a<b\le n$ (with $\rho_{ab}=\rho(t_a,t_b)$),
\begin{equation}
\abs{G(a,b;x,y)}\le M\,\varphi_{c,\rho_{ab}}(y-x),\qquad
\abs{F_\ell(a,b;x,y)}\le C_{1,\ell}\,\rho_{ab}^{-1}\varphi_{c,\rho_{ab}}(y-x),
\end{equation}
then, for all $0\le j<k\le n$ and $0\le r\le k-j$ (with the degenerate terms $i_1=j$ read as explained after \eqref{eq:convn}),
\begin{equation}
\label{eq:disc-iter}
\abs{G\circledast F_1\circledast\cdots\circledast F_r}(j,k;x,y)
\le M\,\Big(\prod_{\ell=1}^{r}C_{1,\ell}\Big)\rho_{jk}^{r}\,\Gamma^{-1}\big(1+\tfrac r2\big)\,\varphi_{c,\rho_{jk}}(y-x),
\end{equation}
with the constant $c$ independent of $r$, of $n$ and of the constants $C_{1,\ell}$ (the
latter may be enlarged if necessary).
\end{enumerate}
\end{lemma}

\begin{proof}
\emph{(i)} By \eqref{eq:Kndef} and the Lipschitz continuity of $\Sigma,b$ in $x$
(Assumption~\ref{ass:spatial}), the coefficient differences in $K_n$ are bounded by $C\norm{x-y}$,
and the derivative bounds of Lemma~\ref{lem:kernel}(iv) (the frozen density
$\tilde p_n(\check j,k;\cdot,y)$ is Gaussian with covariance
$\bar\Sigma_n(\check j,k;y)\succeq c_0\rho_{\check j,k}^2 I$, and
$\rho_{\check j,k}\le\rho_{jk}\le2\rho_{\check j,k}$) give
\begin{equation}
\abs{\partial^\alpha_x\tilde p_n(\check j,k;x,y)}
\le C_\alpha\,\rho_{jk}^{-\abs{\alpha}}\,\varphi_{c,\rho_{jk}}(y-x),\qquad \abs{\alpha}=1,2.
\end{equation}
Hence, with the weight bound of Lemma~\ref{lem:kernel}(ii),
\begin{equation}
\abs{K_n(j,k;x,y)}
\le C\norm{x-y}\big[\rho_{jk}^{-2}+\rho_{jk}^{-1}\big]\varphi_{c,\rho_{jk}}(y-x)
\le C_1\,\rho_{jk}^{-1}\varphi_{c,\rho_{jk}}(y-x),
\end{equation}
and, by Lemma~\ref{lem:kernel-lattice} together with
$n^{-1/2}\rho_{jk}^{-1}\le\rho_{jk}^{-1}$ and $n^{-1}\rho_{jk}^{-3}\le\rho_{jk}^{-1}$ (valid since
$n\ge1$ and $\rho_{jk}\ge n^{-1/2}$),
\begin{equation}
\abs{H_n(j,k;x,y)}
\le\abs{K_n(j,k;x,y)}+\abs{H_n(j,k;x,y)-K_n(j,k;x,y)}
\le C_1\,\rho_{jk}^{-1}\varphi_{c,\rho_{jk}}(y-x),
\end{equation}
proving (i).

\emph{(ii)} The case $r=0$ is immediate from the hypothesis on $G$. For $r\ge1$, expanding the
convolutions with the chains introduced after \eqref{eq:convn} (so that every width
$\rho_{i_\ell,i_{\ell+1}}$ with $\ell\ge1$ below is positive),
\begin{multline}
\big(G\circledast F_1\circledast\cdots\circledast F_r\big)(j,k;x,y)\\
=\sum_{j\le i_1<\cdots<i_r\le k-1}n^{-r}\int_{\R^{dr}}
G(j,i_1;x,z_1)\Big[\prod_{\ell=1}^{r-1}F_\ell(i_\ell,i_{\ell+1};z_\ell,z_{\ell+1})\Big]
F_r(i_r,k;z_r,y)\,\dd z_1\cdots\dd z_r,
\end{multline}
with $i_0=j$, $i_{r+1}=k$. Integrating out $z_1,\dots,z_r$ successively using the exact convolution
identity (all factors here share the same constant $c$; as in the proof of Lemma~\ref{lem:kernel}(i))
\begin{equation}
\varphi_{c,\rho}*\varphi_{c,\rho'}=(\pi c)^{d/2}\varphi_{c,\sqrt{\rho^2+\rho'^2}},
\end{equation}
the widths telescope in quadrature to $\rho_{jk}$, while each integration contributes a factor
$(\pi c)^{d/2}$, so the spatial factor is $(\pi c)^{dr/2}\varphi_{c,\rho_{jk}}(y-x)$. Hence
\begin{equation}
\begin{split}
\abs{G\circledast F_1\circledast\cdots\circledast F_r}(j,k;x,y)
&\le M\,\Big(\prod_{\ell=1}^{r}C_{1,\ell}\Big)(\pi c)^{dr/2}\,S_r(j,k)\,\varphi_{c,\rho_{jk}}(y-x),\\
S_r(j,k)&:=\sum_{j\le i_1<\cdots<i_r\le k-1}n^{-r}\prod_{\ell=1}^{r}\rho_{i_\ell,i_{\ell+1}}^{-1}.
\end{split}
\end{equation}
We claim
\begin{equation}
S_r(j,k)\le C\,\Gamma\big(\tfrac12\big)^{r}\,\Gamma^{-1}\big(1+\tfrac r2\big)\rho_{jk}^r.
\end{equation}
For $r=1$,
\begin{equation}
S_1(j,k)=n^{-1}\sum_{i=j}^{k-1}\rho_{ik}^{-1}
=n^{-1/2}\sum_{m=1}^{k-j}m^{-1/2}\le Cn^{-1/2}(k-j)^{1/2}=C\rho_{jk}.
\end{equation}
For the induction step, grouping by the last index $i_r$,
\begin{equation}
S_r(j,k)=\sum_{i=j+1}^{k-1}n^{-1}\,\rho_{ik}^{-1}S_{r-1}(j,i)
\le C\,\Gamma\big(\tfrac12\big)^{r-1}\Gamma^{-1}\big(1+\tfrac{r-1}{2}\big)
n^{-r/2}\sum_{i=j+1}^{k-1}(k-i)^{-1/2}(i-j)^{(r-1)/2},
\end{equation}
using $\rho_{ik}^{-1}=n^{1/2}(k-i)^{-1/2}$ and $\rho_{ij}^{r-1}=n^{-(r-1)/2}(i-j)^{(r-1)/2}$. The
summand increases with $i$, so the Riemann sum is bounded by the corresponding integral; by
Lemma~\ref{lem:kernel}(iii) (with $\alpha=-1$, $\beta=r-1$),
\begin{equation}
\sum_{i=j+1}^{k-1}(k-i)^{-1/2}(i-j)^{(r-1)/2}
\le C\int_j^k(k-u)^{-1/2}(u-j)^{(r-1)/2}\dd u
=C(k-j)^{r/2}\,\mathrm{B}\big(\tfrac12,\tfrac{r+1}{2}\big).
\end{equation}
Since $n^{-r/2}(k-j)^{r/2}=\rho_{jk}^r$ and
$\mathrm{B}\big(\tfrac12,\tfrac{r+1}{2}\big)
=\Gamma\big(\tfrac12\big)\Gamma\big(\tfrac{r+1}{2}\big)\Gamma^{-1}\big(1+\tfrac r2\big)$
with $\Gamma\big(1+\tfrac{r-1}{2}\big)=\Gamma\big(\tfrac{r+1}{2}\big)$, we get
\begin{equation}
S_r(j,k)\le C\,\Gamma\big(\tfrac12\big)^{r}\,\Gamma^{-1}\big(1+\tfrac r2\big)\rho_{jk}^r.
\end{equation}
Absorbing $C$, $(\pi c)^{d/2}$ and $\sqrt{\pi}=\Gamma(\tfrac12)$ into relabelled constants
$C_{1,\ell}$ yields \eqref{eq:disc-iter}: the series
\begin{equation*}
\sum_{r\ge0}\Big(\prod_{\ell=1}^{r}C_{1,\ell}\Big)\rho_{jk}^r\,\Gamma^{-1}\big(1+\tfrac r2\big)
\end{equation*}
converges uniformly in $\rho_{jk}\le1$ for $C_{1,\ell}\le C_1$, because
$\Gamma^{-1}(1+\tfrac r2)$ decays superexponentially.
\end{proof}

We shall also need the following purely combinatorial lattice-sum estimate, in which neither the
measure flow nor the coefficients appear.

\begin{lemma}[Hard-class chain sums]
\label{lem:hard-class}
There is a constant $C<\infty$, independent of $n$, such that, for all $0\le j<k\le n$,
$1\le r\le k-j$ and $1\le m\le r$, with $h=n^{-1}$, $i_{r+1}=k$ and
\begin{equation}
\label{eq:hard-class}
\mathcal H_m:=h^{r}n^{-1}\!\!\sum_{j\le i_1<\dots<i_r\le k-1}
\Big[\prod_{\substack{1\le\ell\le r\\ \ell\ne m}}\rho_{i_\ell,i_{\ell+1}}^{-1}\Big]\rho_{i_m,i_{m+1}}^{-3},
\end{equation}
one has
\begin{equation}
\label{eq:T2-hard-bound}
\mathcal H_m\le C\,n^{-1/2}\,\Gamma(\tfrac12)^{r-1}\,\Gamma^{-1}(\tfrac{r+1}{2})\,\rho_{jk}^{r-1},
\qquad m=1,\dots,r.
\end{equation}
\end{lemma}

\begin{proof}
\emph{(a) All slots give the same value.} Pass to the gaps $d_\ell:=i_{\ell+1}-i_\ell$ ($d_0\ge0$, and
$d_\ell\ge1$ for $\ell\ge1$),
for which $(i_1,\dots,i_r)\leftrightarrow(d_0,\dots,d_r)$ is a bijection onto
the tuples with $\sum_{\ell=0}^{r}d_\ell=k-j$. Substituting
$\rho_{ab}^{-1}=n^{1/2}(b-a)^{-1/2}$ and $\rho_{i_m,i_{m+1}}^{-3}=n^{3/2}d_m^{-3/2}$, the prefactors
of the $r-1$ factors $\rho^{-1}$, of the $r$ convolution weights $h^{r}=n^{-r}$, of the prefactor
$n^{-1}$ of $\rho_{i_m,i_{m+1}}^{-3}$ and of $\rho_{i_m,i_{m+1}}^{-3}$ itself combine into
$n^{-r}n^{-1}n^{(r-1)/2}n^{3/2}=n^{-r/2}$, so that \eqref{eq:hard-class} becomes
\begin{equation}
\label{eq:T2-gaps}
\mathcal H_m=n^{-r/2}\!\!\sum_{\substack{d_0\ge0,\ d_1,\dots,d_r\ge1\\ d_0+\dots+d_r=k-j}}
d_m^{-3/2}\prod_{\substack{1\le\ell\le r\\ \ell\ne m}}d_\ell^{-1/2}.
\end{equation}
This depends on $m$ only through which of the gaps $d_1,\dots,d_r$ carries the exponent $-3/2$;
since the domain is invariant under permutations of the $r$ gaps $d_1,\dots,d_r$, all of which satisfy
$d_\ell\ge1$, the value is independent of
$m$. We may therefore take $m=r$, so that $d_r=k-i_r$ is the special (last) gap and
\begin{equation}
\label{eq:Sh-sum}
\mathcal H_r=n^{-r/2}\!\sum_{j\le i_1<\dots<i_r\le k-1}
\Big[\prod_{\ell=1}^{r-1}(i_{\ell+1}-i_\ell)^{-1/2}\Big](k-i_r)^{-3/2}.
\end{equation}

\emph{(b) The chain on the left of the special gap.} Fix $i_r$. The sum over
$j\le i_1<\dots<i_{r-1}<i_r$ is exactly the index form of the quantity $S_{r-1}(j,i_r)$ estimated in
the proof of Lemma~\ref{lem:disc-kernel-bounds}(ii) (induction on $r$, the outermost sum being
evaluated by the Beta integral); explicitly,
\begin{equation}
\label{eq:T2-inner}
\sum_{j\le i_1<\dots<i_{r-1}<i_r}\prod_{\ell=1}^{r-1}(i_{\ell+1}-i_\ell)^{-1/2}
\le C\,\Gamma(\tfrac12)^{r-1}\,\Gamma^{-1}(\tfrac{r+1}{2})\,(i_r-j)^{(r-1)/2},
\end{equation}
where $\Gamma(1+\tfrac{r-1}{2})=\Gamma(\tfrac{r+1}{2})$. Inserting \eqref{eq:T2-inner} into
\eqref{eq:Sh-sum},
\begin{equation}
\label{eq:T2-hard-sum1}
\mathcal H_r
\le C\,n^{-r/2}\Gamma(\tfrac12)^{r-1}\Gamma^{-1}(\tfrac{r+1}{2})
\sum_{i_r=j}^{k-1}(k-i_r)^{-3/2}(i_r-j)^{(r-1)/2}.
\end{equation}

\emph{(c) The sum over the free endpoint converges.} With $a:=k-i_r\ge1$ and
$(i_r-j)^{(r-1)/2}\le(k-j)^{(r-1)/2}$,
\begin{equation}
\label{eq:T2-hard-sum2}
\sum_{i_r=j}^{k-1}(k-i_r)^{-3/2}(i_r-j)^{(r-1)/2}
\le(k-j)^{(r-1)/2}\sum_{a\ge1}a^{-3/2}\le C\,(k-j)^{(r-1)/2}.
\end{equation}
Finally $(k-j)^{(r-1)/2}=n^{(r-1)/2}\rho_{jk}^{r-1}$, so that in \eqref{eq:T2-hard-sum1} the
prefactor is $n^{-r/2}n^{(r-1)/2}=n^{-1/2}$ and
\begin{equation}
\label{eq:T2-hard-final}
\mathcal H_r
\le C\,n^{-1/2}\,\Gamma(\tfrac12)^{r-1}\,\Gamma^{-1}(\tfrac{r+1}{2})\,\rho_{jk}^{r-1},
\end{equation}
which is \eqref{eq:T2-hard-bound}.

It is worth recording why the a priori dangerous factor $n^{-1}\rho_{i_r,k}^{-3}$ produces no
growth. In index form it equals $n^{1/2}(k-i_r)^{-3/2}$: the prefactor $n^{1/2}$ is exactly
compensated by $n^{-r}\cdot(n^{1/2})^{r-1}$, and the exponent $-3/2$ makes the lattice sum over the
free endpoint $i_r$ converge, so the special interval contributes only $O(1)$ and no factor
$(k-j)$, whereas each of the other $r-1$ intervals carries the exponent $-1/2$, whose accumulation
is precisely the Beta factor $(k-j)^{(r-1)/2}$ already extracted in \eqref{eq:T2-inner}.
\end{proof}

\subsection{Replacement of the frozen density}
\label{sec:replace}

\begin{lemma}
\label{lem:replace}
Under Assumptions~\ref{ass:spatial}--\ref{ass:measLip}, there exist constants $C<\infty$ and $c'>0$,
independent of $n$, such that, for all $0\le j<k\le n$ and $x,y\in\R^d$ (with $\rho=\rho_{jk}$),
\begin{equation}
p_n(j,k;x,y)=\sum_{r=0}^{k-j}\big(\tilde p\circledast K_n^{\circledast r}\big)(j,k;x,y)+\mathcal R_n(j,k;x,y),
\end{equation}
where
\begin{equation}
\abs{\mathcal R_n(j,k;x,y)}\le C\,n^{-1/2}\rho^{-1}\varphi_{c',\rho}(y-x).
\end{equation}
\end{lemma}

\begin{proof}
Set $\mathsf{E}_n:=H_n-K_n$ and $\mathsf{D}_n:=\tilde p_n-\tilde p$, where for $a<b$ we write
$\mathsf{D}_n(a,b;x,y)=\tilde p_n(a,b;x,y)-\tilde p(t_a,t_b,x,y)$. By Lemma~\ref{lem:kernel-lattice},
\begin{equation}
\label{eq:En-bound}
\abs{\mathsf{E}_n(a,b;x,y)}\le C\,n^{-1/2}\rho_{ab}^{-1}\varphi_{c,\rho_{ab}}(y-x)
+C\,n^{-1}\rho_{ab}^{-3}\varphi_{c,\rho_{ab}}(y-x),
\end{equation}
and by the sharp frozen-density comparison \eqref{eq:compare-frozen},
\begin{equation}
\abs{\mathsf{D}_n(a,b;x,y)}\le C\,n^{-1/2}\varphi_{c,\rho_{ab}}(y-x).
\end{equation}
Starting from the discrete parametrix \eqref{eq:param-disc}, we perform two replacements.

\emph{Replacement A (kernel $H_n\to K_n$).} For $r\ge1$, telescoping
\begin{equation}
H_n^{\circledast r}-K_n^{\circledast r}=\sum_{\ell=0}^{r-1}H_n^{\circledast \ell}\circledast \mathsf{E}_n\circledast K_n^{\circledast (r-1-\ell)}
\end{equation}
gives
\begin{equation}
\tilde p_n\circledast H_n^{\circledast r}
=\tilde p_n\circledast K_n^{\circledast r}
+\sum_{\ell=0}^{r-1}\tilde p_n\circledast H_n^{\circledast \ell}\circledast \mathsf{E}_n\circledast K_n^{\circledast (r-1-\ell)}.
\end{equation}
Each summand is the convolution of the density $\tilde p_n$ (bounded by $C\varphi_{c,\rho}$) with $r$
kernel factors: $\ell$ factors $H_n$ and $r-1-\ell$ factors $K_n$ (each bounded by
$C_1\rho^{-1}\varphi$ by Lemma~\ref{lem:disc-kernel-bounds}(i)) and one factor $\mathsf{E}_n$, whose
bound \eqref{eq:En-bound} carries two terms. The first term $Cn^{-1/2}\rho^{-1}\varphi$ is of the form
$C_{1,\ell}\rho^{-1}\varphi$ with $C_{1,\ell}=Cn^{-1/2}$, so the discrete propagation estimate
\eqref{eq:disc-iter} applies and gives
\begin{equation}
\abs{\tilde p_n\circledast H_n^{\circledast \ell}\circledast \mathsf{E}_n\circledast K_n^{\circledast (r-1-\ell)}}(j,k;x,y)
\le C\,n^{-1/2}C_1^{r}\,\rho^{r}\,\Gamma^{-1}\big(1+\tfrac r2\big)\,\varphi_{c,\rho}(y-x).
\end{equation}
The second term $Cn^{-1}\rho^{-3}\varphi$ is not of this form, since its extra factor $\rho^{-2}$
refers to the \emph{local} width of the slot occupied by $\mathsf{E}_n$; but the corresponding chain
sum is exactly the ``hard class'' quantity $\mathcal H_m$ of \eqref{eq:hard-class}, and by
\eqref{eq:T2-hard-bound} (Lemma~\ref{lem:hard-class}) it is bounded by
$Cn^{-1/2}\Gamma^{-1}(\tfrac{r+1}2)\rho^{r-1}\varphi_{c,\rho}$ (the $r$-dependent constants
$\Gamma(\tfrac12)^{r-1}$ and $(\pi c)^{dr/2}$ being absorbed into the relabelled $C_1$), so that
\begin{equation}
\label{eq:En-hard}
\abs{\tilde p_n\circledast H_n^{\circledast \ell}\circledast \mathsf{E}_n\circledast K_n^{\circledast (r-1-\ell)}}(j,k;x,y)
\le C\,n^{-1/2}C_1^{r}\,\rho^{r-1}\,\Gamma^{-1}\big(\tfrac{r+1}2\big)\,\varphi_{c,\rho}(y-x)
\end{equation}
as well.

\emph{Replacement B (frozen density $\tilde p_n\to\tilde p$).} For $r\ge0$,
\begin{equation}
\tilde p_n\circledast K_n^{\circledast r}
=\tilde p\circledast K_n^{\circledast r}+\mathsf{D}_n\circledast K_n^{\circledast r}.
\end{equation}
By the discrete propagation estimate \eqref{eq:disc-iter}, applied to $G=\mathsf{D}_n$ (with $M$ replaced
by $Cn^{-1/2}$) and $F_\ell=K_n$,
\begin{equation}
\abs{\mathsf{D}_n\circledast K_n^{\circledast r}}(j,k;x,y)
\le C\,n^{-1/2}C_1^{r}\,\rho^{r}\,\Gamma^{-1}\big(1+\tfrac r2\big)\,\varphi_{c,\rho}(y-x).
\end{equation}
Applying Replacement~A to each $r\ge1$ and Replacement~B to each $r\ge0$ yields the decomposition of
the lemma, with
\begin{equation}
\mathcal R_n=\mathsf{D}_n+\sum_{r=1}^{k-j}\mathsf{D}_n\circledast K_n^{\circledast r}
+\sum_{r=1}^{k-j}\sum_{\ell=0}^{r-1}\tilde p_n\circledast H_n^{\circledast \ell}\circledast \mathsf{E}_n\circledast K_n^{\circledast (r-1-\ell)}.
\end{equation}
Each summand of the last (double) sum is bounded, by \eqref{eq:En-hard} and the display preceding
it, by $Cn^{-1/2}C_1^{r}\rho^{r-1}\Gamma^{-1}(1+\tfrac r2)\varphi_{c,\rho}$ (the first bound is
dominated by the second up to a constant, since
$\rho\,\Gamma(\tfrac{r+1}2)/\Gamma(1+\tfrac r2)=O(\rho/\sqrt r)\le C$, the factor $O(\sqrt r)=O(C^{r})$
being absorbed into a relabelled $C_1$), and there are $r$ such
summands for each $r$;
absorbing the factor $r$ (which is at most $2^r$) into $C_1$, this double sum contributes at most
\begin{equation}
C\,n^{-1/2}\sum_{r=1}^{k-j}C_1^{r}\rho^{r-1}\Gamma^{-1}\big(1+\tfrac r2\big)\varphi_{c,\rho}
\le C\,n^{-1/2}\rho^{-1}\varphi_{c,\rho}(y-x),
\end{equation}
using $\rho\le1$, the series $\sum_{r\ge1}C_1^{r}\rho^{r-1}\Gamma^{-1}(1+\tfrac r2)$ converging
uniformly in $\rho\le1$ (the summands decaying superexponentially in $r$). Together with the first term $\mathsf{D}_n$ (the case $r=0$) and the single sum
$\sum_{r=1}^{k-j}\mathsf{D}_n\circledast K_n^{\circledast r}$, both bounded by
$Cn^{-1/2}\varphi_{c,\rho}\le Cn^{-1/2}\rho^{-1}\varphi_{c,\rho}$, we obtain
\begin{equation}
\abs{\mathcal R_n(j,k;x,y)}\le C\,n^{-1/2}\rho^{-1}\varphi_{c,\rho}(y-x).
\end{equation}
Relabelling $c$ as $c'$ (enlarging it if necessary, so that the weight becomes
$\varphi_{c',\rho}$), we obtain $\abs{\mathcal R_n(j,k;x,y)}\le Cn^{-1/2}\rho^{-1}\varphi_{c',\rho}(y-x)$.
\end{proof}

\subsection{Estimation of the difference terms}
\label{sec:difference}

We decompose
$p(t_j,t_k,x,y)-p_n(j,k;x,y)$ by comparing the two parametrix expansions term by term, using
the two steps isolated in Lemma~\ref{lem:T1} (replacement of $\circledast$ by $*$) and
Lemma~\ref{lem:T2} (replacement of $H$ by $K_n$), the second after the preliminary
time-discretisation step of Lemma~\ref{lem:discr-kernel}. By the convention \eqref{eq:conv-order}
all the estimates below are used only in the range
$r\le k-j$; the complementary range
$r>k-j$ is treated separately in
Section~\ref{sec:completion}.

\begin{lemma}[Replacement of $\circledast$ by $*$]
\label{lem:T1}
Under Assumptions~\ref{ass:spatial}--\ref{ass:measLip}, there
exist constants $C<\infty$ and $c>0$ such that, for every $0\le r\le k-j$,
$0\le j<k\le n$, and $x,y\in\R^d$ (with $\rho=\rho_{jk}$),
\begin{equation}
\abs{\big(\tilde p* K_n^{\circledast r}\big)(j,k;x,y)-\big(\tilde p\circledast K_n^{\circledast r}\big)(j,k;x,y)}
\le C\,n^{-1/2}\,C_1^{r}\,\Gamma^{-1}\big(1+\tfrac r2\big)\,\varphi_{c,\rho}(y-x).
\end{equation}
Here $\tilde p* K_n^{\circledast r}$ denotes the convolution in which only the leading factor $\tilde p$
is integrated in continuous time, all the $K_n$-iterates being grid convolutions $\circledast$; this
convention is spelled out at the beginning of the proof. By the convention \eqref{eq:conv-order}
of this section the order is restricted to $0\le r\le k-j$, which is the range in which both
sides are non-trivial and the only range required below; the complementary range $r>k-j$, in
which the discrete term vanishes identically, is treated separately in Section~\ref{sec:completion}.
\end{lemma}

\begin{proof}
Throughout, $C$ changes from line to line and is independent of $n,j,k,r,x,y$; we abbreviate
$\varphi:=\varphi_{c,\rho}(y-x)$ with $c$ possibly enlarged from line to line, and recall
$\rho=\rho_{jk}$.

\emph{The convention and the cell discrepancy.}
Recall that $\tilde p*K_n^{\circledast r}$ means that only the leading factor $\tilde p$ is integrated in
continuous time while all the $K_n$-factors are grid convolutions; we make this precise by
\begin{equation}
\tilde p*K_n^{\circledast r}:=\big(\tilde p*K_n\big)\circledast K_n^{\circledast (r-1)}\qquad(r\ge1),
\end{equation}
whereas for $r=0$ both sides reduce to $\tilde p$. The two convolutions differ only through the
leading factor $\tilde p$, and this discrepancy is carried by the cell error
\begin{equation}
\varepsilon(j,i;x,z):=n\int_{t_i}^{t_{i+1}}\tilde p(t_j,u,x,z)\dd u-\tilde p(t_j,t_i,x,z),
\qquad j\le i\le k-1 .
\end{equation}
Indeed $K_n(u,k;z,y)=K_n(i,k;z,y)$ for $u\in[t_i,t_{i+1})$, so replacing $\tilde p$ by $\varepsilon$
inside the grid sum gives, at the level of the leading factor,
\begin{equation}
\begin{aligned}
\big(\tilde p*K_n\big)(j,k;x,y)-\big(\tilde p\circledast K_n\big)(j,k;x,y)
&=\sum_{i=j}^{k-1}n^{-1}\int_{\R^d}\varepsilon(j,i;x,z)K_n(i,k;z,y)\dd z\\
&=\big(\varepsilon\circledast K_n\big)(j,k;x,y),
\end{aligned}
\end{equation}
and convolving both sides with $K_n^{\circledast (r-1)}$ yields the exact identity
\begin{equation}
\label{eq:delta-identity}
\begin{aligned}
\tilde p*K_n^{\circledast r}-\tilde p\circledast K_n^{\circledast r}
&=\varepsilon\circledast K_n^{\circledast r}\\
&=\big(\varepsilon\circledast K_n\big)\circledast K_n^{\circledast (r-1)},\qquad r\ge1 .
\end{aligned}
\end{equation}

\emph{The time derivative of the frozen density.} Writing
$\tilde p(s,t,x,y)=g(\bar m,\bar\Sigma;y-x)$ with $\bar m=\int_s^t b$, $\bar\Sigma=\int_s^t\Sigma$,
and using $\partial_{\bar m}g=-\partial_y g$, $\partial_{\bar\Sigma}g=\tfrac12\partial_y^2 g$,
\begin{equation}
\label{eq:dtp}
\partial_t\tilde p(s,t,x,y)=-b(t,y,\mu_t)\cdot\partial_y\tilde p(s,t,x,y)
+\tfrac12\Sigma(t,y,\mu_t):\partial_y^2\tilde p(s,t,x,y).
\end{equation}
This is the only property of $\tilde p$ used below.

\textbf{Step 1: the case $r=1$, the initial cell $i=j$.}
Here $\tilde p(t_j,t_j,x,\cdot)=\delta_x$. Setting
$\mathcal I(u):=\int_{\R^d}\tilde p(t_j,u,x,z)K_n(j,k;z,y)\dd z$ and using
$\tilde p(t_j,t_j,x,\cdot)=\delta_x$ together with $\int_{t_j}^{t_{j+1}}\dd u=n^{-1}$,
\begin{equation}
n^{-1}\int_{\R^d}\varepsilon(j,j;x,z)K_n(j,k;z,y)\dd z
=\int_{t_j}^{t_{j+1}}\big[\mathcal I(u)-\mathcal I(t_j)\big]\dd u .
\end{equation}
The frozen density $\tilde p(t_j,u,x,\cdot)$ is centred at $x+O(u-t_j)$ with covariance of order
$u-t_j\le n^{-1}$; expanding $K_n(j,k;\cdot,y)$ about $x$ therefore gives
\begin{equation}
\mathcal I(u)-\mathcal I(t_j)=O\big((u-t_j)\big[\rho^{-2}+\rho^{-3}\big]\varphi\big).
\end{equation}
Integrating this over the cell of length $n^{-1}$ and using $\rho\ge n^{-1/2}$,
\begin{equation}
\Big|n^{-1}\int_{\R^d}\varepsilon(j,j;x,z)K_n(j,k;z,y)\dd z\Big|
\le Cn^{-2}\rho^{-3}\varphi\le Cn^{-1/2}\varphi .
\end{equation}

\textbf{Step 2: the case $r=1$, the bulk cells $i\ge j+1$.}
By the fundamental theorem of calculus and \eqref{eq:dtp},
\begin{equation}
\label{eq:delta-bulk}
\varepsilon(j,i;x,z)=n\int_{t_i}^{t_{i+1}}\int_{t_i}^{u}
\big[-b\cdot\partial_z\tilde p+\tfrac12\Sigma:\partial_z^2\tilde p\big](t_j,\tau,x,z)\dd\tau\dd u,
\end{equation}
the bracket being evaluated at $(t_j,\tau,x,z)$, with $\partial_z$ acting on the second spatial
variable. We bound the two summands of \eqref{eq:delta-bulk} after convolution with $K_n$; in both
cases the double time integral contributes $\int_{t_i}^{t_{i+1}}\int_{t_i}^{u}\dd\tau\dd u=\tfrac12n^{-2}$.

\emph{The drift summand.} Here we use the two elementary bounds
\begin{equation}
\label{eq:T1-bounds}
\abs{\partial_z\tilde p(t_j,\tau,x,\cdot)}\le C\rho(t_j,\tau)^{-1}\varphi_{c,\rho(t_j,\tau)}(\cdot-x),
\qquad
\abs{K_n(i,k;z,y)}\le C_1\rho_{ik}^{-1}\varphi_{c,\rho_{ik}}(y-z),
\end{equation}
which are Lemma~\ref{lem:kernel}(iv) and Lemma~\ref{lem:disc-kernel-bounds}(i). Convolving them with
the bound of Lemma~\ref{lem:kernel}(i), and using $\rho(t_j,\tau)\ge\rho_{ji}$ for $\tau\ge t_i$ and
the double time integral $\tfrac12n^{-2}$,
\begin{equation}
\label{eq:T1-drift}
\Big|\big(\varepsilon^{\mathrm{drift}}\circledast K_n\big)(j,k;x,y)\Big|
\le C_1\varphi\sum_{i=j+1}^{k-1}n^{-2}\rho_{ji}^{-1}\rho_{ik}^{-1}
\le C_1n^{-1}\varphi,
\end{equation}
the last inequality because, inserting $\rho_{ji}=((i-j)/n)^{1/2}$ and
$\rho_{ik}=((k-i)/n)^{1/2}$, the sum equals
$n^{-1}\sum_{i}(i-j)^{-1/2}(k-i)^{-1/2}\le n^{-1}\pi$ (Lemma~\ref{lem:kernel}(iii)).

\emph{The diffusion summand.} Integrating by parts once in $z$ (the boundary terms vanish by the
Gaussian decay of $\tilde p$ and $K_n$) moves one of the two derivatives $\partial_z$ from
$\tilde p$ onto $\Sigma K_n$:
\begin{multline}
\label{eq:T1-ibp}
\int_{\R^d}\tfrac12\Sigma(\tau,z,\mu_\tau):\partial_z^2\tilde p(t_j,\tau,x,z)\,K_n(i,k;z,y)\dd z\\
=-\tfrac12\int_{\R^d}\sum_{l,m}\partial_{z_l}\tilde p(t_j,\tau,x,z)\,
\partial_{z_m}\big(\Sigma_{lm}(\tau,z,\mu_\tau)K_n(i,k;z,y)\big)\dd z .
\end{multline}
The first derivatives of $K_n$, obtained from \eqref{eq:Kndef} exactly as in
Lemma~\ref{lem:disc-kernel-bounds}(i) and using only the bounded first-order derivatives of $b$ and
$\Sigma$, obey $\abs{\partial_z K_n(i,k;z,y)}\le C\big[\rho_{ik}^{-1}+\rho_{ik}^{-2}\big]\varphi_{c,\rho_{ik}}(y-z)$;
since $\Sigma$ and $\partial_z\Sigma$ are bounded, this gives
\begin{equation}
\label{eq:T1-d2K}
\Big|\sum_{l,m}\partial_{z_m}\big(\Sigma_{lm}K_n\big)(i,k;z,y)\Big|
\le C\big[\rho_{ik}^{-1}+\rho_{ik}^{-2}\big]\varphi_{c,\rho_{ik}}(y-z).
\end{equation}
Convolving the right-hand side of \eqref{eq:T1-d2K} with the derivative weight of $\tilde p$ from
\eqref{eq:T1-bounds}, and inserting the double time integral
$\tfrac12n^{-2}$ together with $\rho(t_j,\tau)\ge\rho_{ji}$,
\begin{equation}
\label{eq:T1-diff}
\Big|\big(\varepsilon^{\mathrm{diff}}\circledast K_n\big)(j,k;x,y)\Big|
\le C_1\varphi\sum_{i=j+1}^{k-1}n^{-2}\rho_{ji}^{-1}\big[\rho_{ik}^{-1}+\rho_{ik}^{-2}\big].
\end{equation}
The two sums are
\begin{equation}
\label{eq:T1-sums}
\sum_{i=j+1}^{k-1}n^{-2}\rho_{ji}^{-1}\rho_{ik}^{-1}\le Cn^{-1},\qquad
\sum_{i=j+1}^{k-1}n^{-2}\rho_{ji}^{-1}\rho_{ik}^{-2}\le Cn^{-1/2},
\end{equation}
the first being the sum already estimated in \eqref{eq:T1-drift}, and the second following
because, inserting $\rho_{ji}^{-1}=n^{1/2}(i-j)^{-1/2}$ and $\rho_{ik}^{-2}=n(k-i)^{-1}$, it becomes
$n^{-1/2}\sum_{i=j+1}^{k-1}(i-j)^{-1/2}(k-i)^{-1}\le Cn^{-1/2}$, the lattice sum
$\sum_{s=1}^{k-j-1}s^{-1/2}(k-j-s)^{-1}$ being bounded by $C(k-j)^{-1/2}\log(k-j)\le C$. Hence
\eqref{eq:T1-diff} is bounded by $C_1n^{-1/2}\varphi$. Adding the drift bound \eqref{eq:T1-drift}
(which contributes $C_1n^{-1}\varphi$) and the initial cell of Step~1 (which contributes at most
$Cn^{-1/2}\varphi$), the total is $C_1n^{-1/2}\varphi$:
\begin{equation}
\label{eq:r1}
\abs{\big(\tilde p*K_n-\tilde p\circledast K_n\big)(j,k;x,y)}\le C_1n^{-1/2}\varphi_{c,\rho}(y-x),
\end{equation}
and since every estimate above was carried out for arbitrary $0\le j<k\le n$, the same bound holds
for the density $\varepsilon\circledast K_n$ at all pairs $(a,b)$, $0\le a<b\le n$.

\textbf{Step 3: the general case $r\ge1$, and conclusion.}
By \eqref{eq:delta-identity} the difference equals
$\varepsilon\circledast K_n^{\circledast r}=(\varepsilon\circledast K_n)\circledast K_n^{\circledast (r-1)}$; by \eqref{eq:r1} the density
$\varepsilon\circledast K_n$ satisfies the hypothesis of Lemma~\ref{lem:disc-kernel-bounds}(ii) with
$M=C_1n^{-1/2}$, while each of the remaining $r-1$ factors is $K_n$, i.e.\ $F_\ell=K_n$ with the
same constant $C_1$. Applying the discrete propagation estimate \eqref{eq:disc-iter},
\begin{equation}
\abs{\varepsilon\circledast K_n^{\circledast r}}(j,k;x,y)
\le Cn^{-1/2}C_1^{r}\rho^{r-1}\Gamma^{-1}\big(1+\tfrac{r-1}{2}\big)\varphi_{c,\rho}(y-x),
\end{equation}
and since $\rho\le1$ and $\Gamma(1+\tfrac r2)\le C^r\Gamma(1+\tfrac{r-1}{2})$ (absorbing the factor
$C^r$ into a relabelled $C_1$),
\begin{equation}
\abs{\varepsilon\circledast K_n^{\circledast r}}(j,k;x,y)
\le Cn^{-1/2}C_1^{r}\Gamma^{-1}\big(1+\tfrac r2\big)\varphi_{c,\rho}(y-x).
\end{equation}
For $r=0$ both sides equal $\tilde p$. Together with \eqref{eq:delta-identity} this proves the
lemma.
\end{proof}

\begin{lemma}[Time-discretisation of the continuous kernel]
\label{lem:discr-kernel}
Under Assumptions~\ref{ass:spatial}--\ref{ass:measLip},
let $H^{\mathrm{g}}_n$ denote the grid embedding of $H$, i.e.\ $H^{\mathrm{g}}_n(s,t,x,y)=H(t_i,t_{i'},x,y)$ for
$s\in[t_i,t_{i+1})$, $t\in[t_{i'},t_{i'+1})$. There exist constants $C<\infty$ and $c>0$
such that, for all $0\le j<k\le n$, $1\le r\le k-j$ and $x,y\in\R^d$,
\begin{equation}
\abs{\tilde p* H^{*r}-\tilde p*H^{\mathrm{g}}_n{}^{\circledast r}}(j,k;x,y)
\le C\,n^{-1/2}\,C_1^{r}\,\Gamma^{-1}\big(1+\tfrac r2\big)\,\varphi_{c,\rho_{jk}}(y-x).
\end{equation}
For $r>k-j$, a range excluded by the convention \eqref{eq:conv-order} of this section, the
embedded kernel $H^{\mathrm{g}}_n{}^{\circledast r}$ (and with it $\tilde p*H^{\mathrm{g}}_n{}^{\circledast r}$) vanishes,
because the piecewise-constant embedding vanishes whenever two consecutive time variables fall in
the same cell, so that $r$ strictly increasing cells would be needed inside $[t_j,t_k]$; the
difference then reduces to the pure continuous term $\tilde p* H^{*r}$, which is controlled
directly by \eqref{eq:bound-r} and does not enter the estimate below.
\end{lemma}

\begin{proof}
Write $H=\tfrac12 \mathsf{A}:\partial_x^2\tilde p+\mathsf{B}\cdot\partial_x\tilde p$, where $\mathsf{A}=(\mathsf{A}_{ij})$ with
$\mathsf{A}_{ij}=\Sigma_{ij}(s,x,\mu_s)-\Sigma_{ij}(s,y,\mu_s)$ and $\mathsf{B}=(\mathsf{B}_i)$ with
$\mathsf{B}_i=b_i(s,x,\mu_s)-b_i(s,y,\mu_s)$ (both of size $O(\norm{x-y})$). In the
terminal variable, the coefficients are independent of $t$ while $\tilde p=g(\bar m,\bar\Sigma;y-x)$
is differentiable in $t$ (with $\partial_t\tilde p=-b\cdot\partial_y\tilde p
+\tfrac12\Sigma:\partial_y^2\tilde p$), so, by Lemma~\ref{lem:kernel}(iv),
\begin{equation}
\label{eq:Hdt}
\abs{\partial_t H(s,t,x,y)}\le C\big[\rho^{-1}+\rho^{-2}+\rho^{-3}\big]\varphi_{c,\rho}(y-x),
\qquad \rho=\rho(s,t).
\end{equation}
In the initial variable the coefficients depend on $s$ only through $\mu_s$, which is
$\tfrac12$-H\"older in $\mathcal{W}_2$ (not differentiable); we use the increment bound (mean-value theorem in
$x$ and Remark~\ref{rem:autotime})
\begin{equation}
\label{eq:Hinc}
\norm{\mathsf{A}(u,x,y)-\mathsf{A}(s,x,y)}+\norm{\mathsf{B}(u,x,y)-\mathsf{B}(s,x,y)}\le C\,n^{-1/2}\norm{x-y},\qquad \abs{u-s}\le n^{-1},
\end{equation}
whereas $\tilde p$ is differentiable in $s$ with
$\abs{\partial_s\partial_x^m\tilde p}\le C[\rho^{-(m+1)}+\rho^{-(m+2)}]\varphi$, $m=1,2$.

Both convolutions share the same leading factor $\tilde p$ and the same kernel factors, and differ
only in whether the intermediate time pairs are continuous or frozen to the grid. Telescoping over
the $r$ kernel positions,
\begin{equation}
\tilde p* H^{*r}-\tilde p*H^{\mathrm{g}}_n{}^{\circledast r}
=\sum_{\ell=1}^{r}\tilde p* H^{*(\ell-1)}*\big[H-H^{\mathrm{g}}_n\big]*H^{\mathrm{g}}_n{}^{\circledast (r-\ell)}.
\end{equation}
For $u\in[t_i,t_{i+1})$, $v\in[t_{i'},t_{i'+1})$ we split
\begin{equation}
H(u,v)-H(t_i,t_{i'})=\big[H(t_i,v)-H(t_i,t_{i'})\big]+\big[H(u,v)-H(t_i,v)\big].
\end{equation}
The first bracket is bounded by $Ch[\rho_{i,i'}^{-1}+\rho_{i,i'}^{-2}+\rho_{i,i'}^{-3}]
\varphi_{c,\rho_{i,i'}}$ by \eqref{eq:Hdt}; the second splits into a coefficient part, bounded by
$Cn^{-1/2}\rho(u,v)^{-1}\varphi_{c,\rho(u,v)}$ (by \eqref{eq:Hinc} and the weight bound
$\norm{x-y}\varphi\le C\rho\varphi$), and a frozen-density part, bounded by
$Ch[\rho_{i,i'}^{-1}+\rho_{i,i'}^{-2}+\rho_{i,i'}^{-3}]\varphi_{c,\rho_{i,i'}}$ (fundamental theorem
in $s$ and the $\partial_s\partial_x^m\tilde p$ bound). Integrating over the two cells (each of length
$n^{-1}$), the cell-integrated discrepancy is bounded by
\begin{equation}
C n^{-5/2}\rho_{i,i'}^{-1}\varphi_{c,\rho_{i,i'}}(y-x)
+C n^{-3}\big[\rho_{i,i'}^{-1}+\rho_{i,i'}^{-2}+\rho_{i,i'}^{-3}\big]\varphi_{c,\rho_{i,i'}}(y-x).
\end{equation}
Inserting this into the $\ell$-th term of the telescoping and applying the continuous iteration
bound of Lemma~\ref{lem:kernbound}, it remains to sum over the grid. Writing
$\rho_{i,i'}^2=(i'-i)n^{-1}$ (all the sums below run over the pairs $j\le i<i'\le k$),
\begin{equation}
\sum_{i<i'}n^{-5/2}\rho_{i,i'}^{-1}\le Cn^{-1/2},\qquad
\sum_{i<i'}n^{-3}\rho_{i,i'}^{-1}\le Cn^{-1},
\end{equation}
\begin{equation}
\sum_{i<i'}n^{-3}\rho_{i,i'}^{-2}\le Cn^{-1}\log n\le Cn^{-1/2},\qquad
\sum_{i<i'}n^{-3}\rho_{i,i'}^{-3}\le Cn^{-3/2}(k-j)\le Cn^{-1/2}.
\end{equation}
Summing over the $r$ positions (the factor $r$ absorbed into $C_1$) and using $\rho_{jk}\le1$ together
with $\Gamma^{-1}(1+\tfrac{r-1}{2})\le C^r\Gamma^{-1}(1+\tfrac r2)$, we obtain the claim.
\end{proof}

\begin{lemma}[Replacement of $H$ by $K_n$]
\label{lem:T2}
Under Assumptions~\ref{ass:spatial}--\ref{ass:measLip}, there exist
constants $C<\infty$ and $c>0$ such that, for all $0\le j<k\le n$ and $0\le r\le k-j$,
\begin{equation}
\label{eq:T2-bound}
\abs{\tilde p* H^{*r}-\tilde p* K_n^{\circledast r}}(j,k;x,y)
\le C\,n^{-1/2}\,C_1^{r}\,\Gamma^{-1}\big(1+\tfrac r2\big)\,\varphi_{c,\rho_{jk}}(y-x).
\end{equation}
By the convention \eqref{eq:conv-order} of this section the order is restricted to
$0\le r\le k-j$. For $r>k-j$ the grid object $\tilde p\circledast K_n^{\circledast r}$ vanishes, so the left-hand
side reduces to the continuous term $\tilde p* H^{*r}$; this tail is estimated separately through
\eqref{eq:bound-r} (see Section~\ref{sec:completion}).
\end{lemma}

\begin{proof}
Throughout $h=n^{-1}$ and $(\phi\circledast\psi)(j,k;x,y)=\sum_{i=j}^{k-1}h\int_{\R^d}\phi(j,i;x,z)\,\psi(i,k;z,y)\dd z$. For a
chain of grid indices $j=i_0\le i_1<\dots<i_r<i_{r+1}=k$, the composition of $r+1$ grid factors
$F_1,\dots,F_{r+1}$ reads
\begin{equation}
\label{eq:T2-chain}
\big(F_1\circledast\cdots\circledast F_{r+1}\big)(j,k;x,y)
=\sum_{j\le i_1<\dots<i_r\le k-1}h^{r}\!\int_{\R^{dr}}\prod_{m=1}^{r+1}F_m(i_{m-1},i_m;z_{m-1},z_m)\,\dd z_1\cdots\dd z_r,
\end{equation}
with $z_0=x$, $z_{r+1}=y$.

\textbf{Step 1: reduction to a discrete difference.}
Let $H^{\mathrm{g}}_n$ be the grid embedding of $H$ as in Lemma~\ref{lem:discr-kernel}, and decompose
\begin{equation}
\label{eq:T2-split}
\tilde p* H^{*r}-\tilde p* K_n^{\circledast r}
=\underbrace{\big(\tilde p* H^{*r}-\tilde p*H^{\mathrm{g}}_n{}^{\circledast r}\big)}_{=:\,(\mathrm{A})}
+\underbrace{\big(\tilde p*H^{\mathrm{g}}_n{}^{\circledast r}-\tilde p* K_n^{\circledast r}\big)}_{=:\,(\mathrm{B})}.
\end{equation}
By Lemma~\ref{lem:discr-kernel} the first part already obeys the required bound,
\begin{equation}
\label{eq:T2-A}
\abs{(\mathrm{A})}(j,k;x,y)\le C\,n^{-1/2}C_1^{r}\Gamma^{-1}\big(1+\tfrac r2\big)\varphi_{c,\rho_{jk}}(y-x),
\qquad 1\le r\le k-j .
\end{equation}
In $(\mathrm{B})$ both $H^{\mathrm{g}}_n$ and $K_n$ are grid objects, so $\tilde p*$ may be replaced by its
grid version $\tilde p\circledast$: applying the argument of Lemma~\ref{lem:T1} with $G=H^{\mathrm{g}}_n{}^{\circledast r}$ and
with $G=K_n^{\circledast r}$ (each has $r$ grid factors bounded by $C_1\rho^{-1}\varphi$),
\begin{equation}
\label{eq:T2-swap}
\abs{\tilde p* G-\tilde p\circledast G}(j,k;x,y)\le C\,n^{-1/2}C_1^{r}\Gamma^{-1}\big(1+\tfrac r2\big)\varphi_{c,\rho_{jk}}(y-x),
\qquad G\in\{H^{\mathrm{g}}_n{}^{\circledast r},K_n^{\circledast r}\},
\end{equation}
so that only the purely discrete difference $\tilde p\circledast H^{\mathrm{g}}_n{}^{\circledast r}-\tilde p\circledast K_n^{\circledast r}$ remains.

\textbf{Step 2: telescoping and the two classes.}
By the telescoping identity for convolution powers,
\begin{equation}
\label{eq:HK-telescope}
\tilde p\circledast H^{\mathrm{g}}_n{}^{\circledast r}-\tilde p\circledast K_n^{\circledast r}
=\sum_{\ell=0}^{r-1}\tilde p\circledast H^{\mathrm{g}}_n{}^{\circledast \ell}\circledast\big[H^{\mathrm{g}}_n-K_n\big]\circledast K_n^{\circledast (r-1-\ell)}.
\end{equation}
Fix $\ell$; the corresponding summand is the composition \eqref{eq:T2-chain} of the $r+1$ grid
factors
\begin{equation}
\label{eq:T2-factors}
F_1=\tilde p,\qquad
(F_2,\dots,F_{r+1})=\Big(\underbrace{H^{\mathrm{g}}_n,\dots,H^{\mathrm{g}}_n}_{\ell},\ H^{\mathrm{g}}_n-K_n,\
\underbrace{K_n,\dots,K_n}_{r-1-\ell}\Big),
\end{equation}
so the difference $H^{\mathrm{g}}_n-K_n$ occupies the interval $[i_m,i_{m+1}]$ with $m:=\ell+1\in\{1,\dots,r\}$,
and each of the other $r$ factors is a single grid kernel. We insert the single-slot bounds
\begin{equation}
\label{eq:T2-frozen}
\begin{aligned}
\abs{\tilde p(j,i_1;z_0,z_1)}&\le C\,\varphi_{c,\rho_{j,i_1}}(z_1-z_0),\\
\abs{H^{\mathrm{g}}_n(i_\ell,i_{\ell+1};z_\ell,z_{\ell+1})}+\abs{K_n(i_\ell,i_{\ell+1};z_\ell,z_{\ell+1})}
&\le C_1\,\rho_{i_\ell,i_{\ell+1}}^{-1}\,\varphi_{c,\rho_{i_\ell,i_{\ell+1}}}(z_{\ell+1}-z_\ell),
\end{aligned}
\end{equation}
the second being \eqref{eq:boundH} for $H^{\mathrm{g}}_n$ and Lemma~\ref{lem:disc-kernel-bounds}(i) for
$K_n$, and, at the special interval, Corollary~\ref{cor:approx-kernel} in the form
\begin{equation}
\label{eq:HmK-full}
\begin{aligned}
\abs{H^{\mathrm{g}}_n-K_n}(i_m,i_{m+1};z_m,z_{m+1})
&\le C\,n^{-1/2}\rho_{i_m,i_{m+1}}^{-1}\,\varphi_{c,\rho_{i_m,i_{m+1}}}(z_{m+1}-z_m)\\
&\quad +C\,n^{-1}\rho_{i_m,i_{m+1}}^{-3}\,\varphi_{c,\rho_{i_m,i_{m+1}}}(z_{m+1}-z_m).
\end{aligned}
\end{equation}
Convolving the $r+1$ Gaussian factors (all carrying the same constant $c$) with
$\varphi_{c,\rho}*\varphi_{c,\rho'}=(\pi c)^{d/2}\varphi_{c,\sqrt{\rho^2+\rho'^2}}$ and
$\sum_{\ell=0}^{r}\rho_{i_\ell,i_{\ell+1}}^2=\rho_{jk}^2$ (widths add in quadrature),
\begin{equation}
\label{eq:T2-spatial}
\varphi_{c,\rho_{j,i_1}}*\varphi_{c,\rho_{i_1,i_2}}*\cdots*\varphi_{c,\rho_{i_r,k}}
=(\pi c)^{dr/2}\,\varphi_{c,\rho_{jk}},
\end{equation}
so each summand of \eqref{eq:HK-telescope} splits into the two classes
\begin{equation}
\label{eq:T2-two-classes}
\abs{\tilde p\circledast H^{\mathrm{g}}_n{}^{\circledast \ell}\circledast[H^{\mathrm{g}}_n-K_n]\circledast K_n^{\circledast (r-1-\ell)}}(j,k;x,y)
\le C\,(\pi c)^{dr/2}C_1^{r-1}\big(\mathcal E_m+\mathcal H_m\big)\varphi_{c,\rho_{jk}}(y-x),
\end{equation}
where, with $m=\ell+1$ and $h=n^{-1}$, the \emph{easy class} is
\begin{equation}
\label{eq:T2-hard}
\mathcal E_m:=h^{r}n^{-1/2}\!\!\sum_{j\le i_1<\dots<i_r\le k-1}
\Big[\prod_{\substack{1\le\ell\le r\\ \ell\ne m}}\rho_{i_\ell,i_{\ell+1}}^{-1}\Big]\rho_{i_m,i_{m+1}}^{-1},
\end{equation}
while the \emph{hard class} $\mathcal H_m$ is the quantity \eqref{eq:hard-class} of
Lemma~\ref{lem:hard-class}.

\textbf{Step 3: the easy class $n^{-1/2}\rho^{-1}$.}
Since $n^{-1/2}\le1$, each of the $r$ kernel factors in $\mathcal E_m$ is bounded by
$C_1\rho^{-1}\varphi$, and the discrete propagation estimate \eqref{eq:disc-iter} with
$G=\tilde p$ (so $M=O(1)$) gives, using $\rho_{jk}\le1$,
\begin{equation}
\label{eq:T2-easy-bound}
\mathcal E_m\le C\,n^{-1/2}\Gamma(\tfrac12)^{r}\rho_{jk}^{r}\Gamma^{-1}\big(1+\tfrac r2\big)
\le C\,n^{-1/2}\Gamma(\tfrac12)^{r}\Gamma^{-1}\big(1+\tfrac r2\big).
\end{equation}

\textbf{Step 4: the hard class: taming $n^{-1}\rho^{-3}$.}
This is the term produced by the second summand of \eqref{eq:HmK-full}, and it is the only place
where the ${-}3$ exponent has to be controlled. By Lemma~\ref{lem:hard-class}, applied at the pair
$(j,k)$ of the present chain, the quantity $\mathcal H_m$ obeys the bound \eqref{eq:T2-hard-bound}.

\textbf{Step 5: conclusion.}
Combining \eqref{eq:T2-easy-bound} and \eqref{eq:T2-hard-bound} in \eqref{eq:T2-two-classes} and
absorbing all chain-independent constants (notably $(\pi c)^{dr/2}$ and $\Gamma(\tfrac12)^{r}$)
into the relabelled $C_1$, each summand of \eqref{eq:HK-telescope} is bounded by
\begin{equation}
C\,n^{-1/2}C_1^{r}\,\Gamma^{-1}(\tfrac{r+1}{2})\rho_{jk}^{r-1}\varphi_{c,\rho_{jk}}(y-x)
\le C\,n^{-1/2}C_1^{r}\,\Gamma^{-1}\big(1+\tfrac r2\big)\varphi_{c,\rho_{jk}}(y-x),
\end{equation}
where we used $\rho_{jk}\le1$ and
$\Gamma^{-1}(\tfrac{r+1}{2})\le C^{r}\Gamma^{-1}(1+\tfrac r2)$ (equivalently
$\Gamma(1+\tfrac r2)/\Gamma(\tfrac{r+1}{2})=O(\sqrt r)=O(C^{r})$). Summing over the $r$ positions
$m=1,\dots,r$ (the factor $r$ absorbed into $C_1^{r}$) and inserting the result together with
\eqref{eq:T2-A} and \eqref{eq:T2-swap} into \eqref{eq:T2-split} yields \eqref{eq:T2-bound}. For $r=0$
there is nothing to prove, since $\tilde p* H^{*0}=\tilde p* K_n^{\circledast 0}=\tilde p$. This proves the
lemma.
\end{proof}

\section{Completion of the Proof of the Main Theorem}
\label{sec:completion}

\begin{proof}
We prove Theorem~\ref{thm:main}. Fix $n\ge1$, put $h=n^{-1}$, and consider the terminal time
$s=0$, $t=1$, i.e.\ the grid indices $j=0$, $k=n$; for this pair the width \eqref{eq:rhodef} is
\begin{equation}
\rho_{0,n}^2=1,\qquad\text{so that}\qquad \varphi_{c,\rho_{0,n}}(u)=\varphi_{c,1}(u)=e^{-\norm{u}^2/c}.
\end{equation}
Recall the notational identifications $p(x_0,y)=p(0,1,x_0,y)$ and $p_n(x_0,y)=p_n(0,n;x_0,y)$. It
is convenient to abbreviate, for $r\ge0$,
\begin{equation}
\label{eq:step0-notation}
\mathcal A_r:=\big(\tilde p* H^{*r}\big)(0,1,x_0,y),\qquad
\mathcal B_r:=\big(\tilde p\circledast K_n^{\circledast r}\big)(0,n;x_0,y),\qquad
\mathcal C_r:=\big(\tilde p* K_n^{\circledast r}\big)(0,n;x_0,y),
\end{equation}
where $\mathcal C_r$ is the hybrid object in which only the leading factor $\tilde p$ is integrated in
continuous time, according to the convention of Lemma~\ref{lem:T1}.

\textbf{Step 1: the two parametrix expansions.}
By the continuous parametrix expansion of Lemma~\ref{lem:param-cont}, evaluated at $(s,t)=(0,1)$,\footnote{See Remark~\ref{rem:exist}: this identification may equivalently be replaced by the definition of $S$, which does not presuppose the existence of $p$.}
\begin{equation}
\label{eq:step1-cont}
p(0,1,x_0,y)=\sum_{r\ge0}\mathcal A_r,
\end{equation}
the series converging absolutely, each term being controlled by \eqref{eq:bound-r}. By the discrete
parametrix expansion with remainder of Lemma~\ref{lem:replace}, evaluated at $j=0$, $k=n$ (so that
$k-j=n$ and $\rho=\rho_{0,n}=1$),
\begin{equation}
\label{eq:step1-disc}
\begin{aligned}
p_n(0,n;x_0,y)&=\sum_{r=0}^{n}\mathcal B_r+\mathcal R_n(0,n;x_0,y),\\
\abs{\mathcal R_n(0,n;x_0,y)}&\le C\,n^{-1/2}\varphi_{c',1}(y-x_0),
\end{aligned}
\end{equation}
the summation stopping at $r=n=k-j$ because an $r$-fold grid convolution needs $r$ strictly
increasing grid points. This is the reason \eqref{eq:step1-disc} carries a remainder while
\eqref{eq:step1-cont} does not.

\textbf{Step 2: subtracting, and splitting the continuous series.}
Subtracting \eqref{eq:step1-disc} from \eqref{eq:step1-cont}, and matching the two series for
$0\le r\le n$,
\begin{equation}
\label{eq:step2}
\mathcal D:=p(0,1,x_0,y)-p_n(0,n;x_0,y)
=\underbrace{\sum_{r=0}^{n}\big(\mathcal A_r-\mathcal B_r\big)}_{=:\,(\mathrm{I})}
+\underbrace{\sum_{r>n}\mathcal A_r}_{=:\,(\mathrm{II})}
-\underbrace{\mathcal R_n(0,n;x_0,y)}_{=:\,(\mathrm{III})};
\end{equation}
indeed the terms $\mathcal B_r=\tilde p\circledast K_n^{\circledast r}$ vanish for $r>n$ (the chain of indices lies in
$\{0,\dots,n-1\}$, which has only $n$ elements, so
no chain of length $r>n$ fits --- this is exactly the range excluded
by the convention \eqref{eq:conv-order} of Section~\ref{sec:comparison}; cf.\ the remark following
\eqref{eq:T2-bound}), so only the continuous tail $r>n$ survives. We bound $(\mathrm{I})$,
$(\mathrm{II})$, $(\mathrm{III})$ one at a time.

\textbf{Step 3: the finite range: term $(\mathrm{I})$.}
Fix $0\le r\le n$ and insert the hybrid object $\mathcal C_r$ between the two summands of $\mathcal A_r-\mathcal B_r$:
\begin{equation}
\label{eq:step3-split}
\mathcal A_r-\mathcal B_r
=\underbrace{\big(\mathcal A_r-\mathcal C_r\big)}_{=:\,(\mathrm{a})}
+\underbrace{\big(\mathcal C_r-\mathcal B_r\big)}_{=:\,(\mathrm{b})}.
\end{equation}
The first bracket $(\mathrm{a})$ removes the discrepancy in the kernel ($H\to K_n$), the second
$(\mathrm{b})$ removes the discrepancy in the leading convolution ($*\to\circledast$); this is exactly the
two-step comparison along the hierarchy \eqref{eq:hierarchy}. Since $r\le n=k-j$, both lemmas
apply in this range: by Lemma~\ref{lem:T2},
\begin{equation}
\label{eq:step3-a}
\abs{\mathcal A_r-\mathcal C_r}\le C\,n^{-1/2}C_1^{r}\Gamma^{-1}\big(1+\tfrac r2\big)\varphi_{c,1}(y-x_0),
\end{equation}
and by Lemma~\ref{lem:T1},
\begin{equation}
\label{eq:step3-b}
\abs{\mathcal C_r-\mathcal B_r}\le C\,n^{-1/2}C_1^{r}\Gamma^{-1}\big(1+\tfrac r2\big)\varphi_{c,1}(y-x_0).
\end{equation}
Adding \eqref{eq:step3-a} and \eqref{eq:step3-b} and summing over $0\le r\le n$,
\begin{equation}
\label{eq:step3-I}
\abs{(\mathrm{I})}
\le C\,n^{-1/2}\Big[\sum_{r=0}^{n}C_1^{r}\Gamma^{-1}\big(1+\tfrac r2\big)\Big]\varphi_{c,1}(y-x_0)
\le C\,n^{-1/2}\varphi_{c,1}(y-x_0),
\end{equation}
the bracket being a pure number: since $\Gamma^{-1}(1+\tfrac r2)$ decays superexponentially in $r$,
the series $\sum_{r\ge0}C_1^{r}\Gamma^{-1}(1+\tfrac r2)$ converges, so the bracket is absorbed into
the constant $C$.

\textbf{Step 4: the continuous tail: term $(\mathrm{II})$.}
For $r>n$ the grid object $\mathcal B_r=\tilde p\circledast K_n^{\circledast r}$ vanishes (an $r$-fold grid convolution needs
$r$ strictly increasing grid points in $\{0,\dots,n-1\}$, which exist only for $r\le n$), so the
tail is the pure continuous series $\sum_{r>n}\mathcal A_r$. By the termwise bound \eqref{eq:bound-r} of
Lemma~\ref{lem:kernbound}, evaluated at $(s,t)=(0,1)$ and $\rho=1$,
\begin{equation}
\label{eq:step4-term}
\abs{\mathcal A_r}=\abs{\big(\tilde p* H^{*r}\big)(0,1,x_0,y)}
\le a_r\,\varphi_{c,1}(y-x_0),\qquad
a_r:=C_1^{\,r+1}\,\Gamma^{-1}\big(1+\tfrac r2\big).
\end{equation}
Since $\Gamma^{-1}(1+\tfrac r2)$ decays superexponentially while $C_1$ is fixed, the series
$\sum_{r\ge0}a_r$ converges, so $\sum_{r>n}a_r\to0$. The tail is, however, \emph{not} simply
dominated by its first term: the sequence $a_r$ first increases, because the ratio
\begin{equation}
\frac{a_{r+1}}{a_r}=C_1\,\frac{\Gamma(1+\tfrac r2)}{\Gamma(1+\tfrac{r+1}{2})}
\le 2C_1\Big(1+\tfrac r2\Big)^{-1/2}\sim C_1\sqrt{2/r}
\end{equation}
exceeds $1$ for $r\lesssim2C_1^2$ and falls below $1$ afterwards, so $a_r$ attains its maximum near
$r\approx2C_1^2$. We therefore split the tail at a threshold $R$, chosen depending only on $C_1$ and
so large that
\begin{equation}
\label{eq:step4-threshold}
R>2eC_1^2\qquad\text{and}\qquad 2C_1\sqrt{2/R}\le\tfrac{\sqrt2}2 .
\end{equation}

\emph{(a) Large $n$: $n\ge R$.} For $r\ge n\ge R$ the ratio satisfies
$a_{r+1}/a_r\le2C_1(1+\tfrac r2)^{-1/2}\le2C_1(1+\tfrac n2)^{-1/2}\le2C_1\sqrt{2/R}\le\tfrac{\sqrt2}2$,
so the tail is dominated by its first term,
\begin{equation}
\label{eq:step4-large}
\sum_{r>n}a_r\le a_{n+1}\sum_{s\ge0}\Big(\tfrac{\sqrt2}2\Big)^{s}\le C\,a_{n+1}
=C\,C_1^{\,n+2}\,\Gamma^{-1}\big(1+\tfrac{n+1}2\big)\le C\,C_1^{\,n}\,\Gamma^{-1}\big(1+\tfrac n2\big).
\end{equation}
By Stirling's lower bound $\Gamma(1+\tfrac n2)\ge c\,(n/(2e))^{n/2}$ and $n\ge R>2eC_1^2$,
\begin{equation}
C_1^{\,n}\,\Gamma^{-1}\big(1+\tfrac n2\big)\le C\Big(\frac{2eC_1^2}{n}\Big)^{n/2}\le C\,n^{-1/2},
\end{equation}
the last step because the exponential $(2eC_1^2/n)^{n/2}$ (with $2eC_1^2/n<1$) decays faster than the
polynomial $n^{-1/2}$; hence $\sum_{r>n}a_r\le C\,n^{-1/2}$ for $n\ge R$.

\emph{(b) Small $n$: $1\le n<R$.} Here the tail is majorised by the total mass
$S:=\sum_{r\ge1}a_r<\infty$, and $n^{-1/2}\ge R^{-1/2}$; hence
\begin{equation}
\sum_{r>n}a_r\le S\le S\sqrt R\,n^{-1/2}.
\end{equation}

Combining (a) and (b) with $C_\ast:=\max\{C,\,S\sqrt R\}$ (a constant depending only on $C_1$, hence only
on the coefficients) gives
\begin{equation}
\label{eq:step4}
\abs{(\mathrm{II})}=\Big|\sum_{r>n}\mathcal A_r\Big|\le C_\ast\,n^{-1/2}\varphi_{c,1}(y-x_0) = C_\ast\,n^{-1/2}\exp\Big(-\tfrac1c\norm{y-x_0}^2\Big).
\end{equation}

\textbf{Step 5: the discrete remainder: term $(\mathrm{III})$.}
By the remainder bound of Lemma~\ref{lem:replace}, with $j=0$, $k=n$ and $\rho=1$,
\begin{equation}
\label{eq:step5}
\abs{(\mathrm{III})}=\abs{\mathcal R_n(0,n;x_0,y)}\le C\,n^{-1/2}\varphi_{c',1}(y-x_0).
\end{equation}

\textbf{Step 6: combining the three bounds.}
We may enlarge $c$ so that $c\ge c'$ (the weights $\varphi_{c,1}$ are increasing in $c$, so this
only weakens the bounds \eqref{eq:step3-a}--\eqref{eq:step4}); then
$\varphi_{c',1}(y-x_0)=e^{-\norm{y-x_0}^2/c'}\le e^{-\norm{y-x_0}^2/c}=\varphi_{c,1}(y-x_0)$.
Collecting \eqref{eq:step3-I}, \eqref{eq:step4} and \eqref{eq:step5} in \eqref{eq:step2},
\begin{equation}
\label{eq:step6}
\abs{p(0,1,x_0,y)-p_n(0,n;x_0,y)}\le C\,n^{-1/2}\varphi_{c,1}(y-x_0)
=C\,n^{-1/2}\exp\Big(-\tfrac1c\norm{y-x_0}^2\Big).
\end{equation}

\textbf{Step 7: the weighted supremum.}
Multiplying \eqref{eq:step6} by $\exp\big(\tfrac1c\norm{y-x_0}^2\big)$, which is exactly the
reciprocal of the weight, gives the $y$-independent bound
\begin{equation}
\exp\Big(\tfrac1c\norm{y-x_0}^{2}\Big)\abs{p(0,1,x_0,y)-p_n(0,n;x_0,y)}\le C\,n^{-1/2},
\qquad y\in\R^d .
\end{equation}
Taking the supremum over $y\in\R^d$ yields \eqref{eq:maintheorem}, and since the exponential weight
is bounded below by $1$ this also gives the uniform bound \eqref{eq:uniformpointwise} of
Theorem~\ref{thm:main}. This completes the proof.
\end{proof}

\appendix

\section{Proof of the Moment and Increment Estimates}
\label{app:moments}

We prove Lemma~\ref{lem:moments}. Throughout, $C_\kappa$ denotes a finite constant depending only on
$\kappa$, $\norm{x_0}$, $\norm{b}_\infty$ and $\norm{\sigma}_\infty$, whose value may change from line
to line. By Assumption~\ref{ass:spatial} (with $\abs{\alpha}=0$), the coefficients are uniformly
bounded: $\norm{b}\le\norm{b}_\infty$ and $\norm{\sigma}\le\norm{\sigma}_\infty$, the latter
since $\norm{\sigma}=\norm{\Sigma^{1/2}}\le\norm{\Sigma}^{1/2}\le\norm{\Sigma}_\infty^{1/2}$.

\emph{(i) Moments of the continuous solution.} Write $X_t=x_0+I_t+M_t$ with
\[
I_t=\int_0^t b(s,X_s,\mu_s)\,\dd s,\qquad
M_t=\int_0^t \sigma(s,X_s,\mu_s)\,\dd W_s .
\]
By $(a+b+c)^{2\kappa}\le 3^{2\kappa-1}(a^{2\kappa}+b^{2\kappa}+c^{2\kappa})$,
\[
\E\norm{X_t}^{2\kappa}\le 3^{2\kappa-1}\Big(\norm{x_0}^{2\kappa}+\E\norm{I_t}^{2\kappa}+\E\norm{M_t}^{2\kappa}\Big).
\]
For the drift term, H\"older's inequality and $\norm{b}\le\norm{b}_\infty$ give
\[
\E\norm{I_t}^{2\kappa}\le\E\Big(\int_0^t\norm{b(s,X_s,\mu_s)}\,\dd s\Big)^{2\kappa}
\le t^{2\kappa}\norm{b}_\infty^{2\kappa}\le\norm{b}_\infty^{2\kappa},
\]
using $t\le1$. For the diffusion term, by the Burkholder--Davis--Gundy inequality,
\[
\E\norm{M_t}^{2\kappa}\le C_\kappa^{\mathrm{BDG}}\,\E\Big(\int_0^t\norm{\sigma(s,X_s,\mu_s)}^2\,\dd s\Big)^{\kappa}
\le C_\kappa^{\mathrm{BDG}}\,t^{\kappa}\norm{\sigma}_\infty^{2\kappa}
\le C_\kappa^{\mathrm{BDG}}\norm{\sigma}_\infty^{2\kappa}.
\]
Combining the three bounds gives
\[
\E\norm{X_t}^{2\kappa}\le 3^{2\kappa-1}\Big(\norm{x_0}^{2\kappa}+\norm{b}_\infty^{2\kappa}
+C_\kappa^{\mathrm{BDG}}\norm{\sigma}_\infty^{2\kappa}\Big)=:C_\kappa,
\]
uniformly in $t\in[0,1]$. This proves the first bound in Lemma~\ref{lem:moments}. 

\emph{(ii) Moments of the Euler scheme.} By \eqref{eq:Euler},
\[
X^{(n)}_k=x_0+h\sum_{j=0}^{k-1}b\big(t_j,X^{(n)}_j,\mu^{(n)}_j\big)
+\sqrt h\sum_{j=0}^{k-1}\sigma\big(t_j,X^{(n)}_j,\mu^{(n)}_j\big)\xi_{j+1}.
\]
As above,
\[
\E\norm{X^{(n)}_k}^{2\kappa}\le 3^{2\kappa-1}\Big(\norm{x_0}^{2\kappa}+\mathcal M_k+\mathcal V_k\Big),
\]
with $b_j=b(t_j,X^{(n)}_j,\mu^{(n)}_j)$, $\sigma_j=\sigma(t_j,X^{(n)}_j,\mu^{(n)}_j)$,
\[
\mathcal M_k=\E\norm{h\sum_{j=0}^{k-1}b_j}^{2\kappa},\qquad
\mathcal V_k=\E\norm{\sqrt h\sum_{j=0}^{k-1}\sigma_j\,\xi_{j+1}}^{2\kappa}.
\]
For the drift, the triangle inequality gives the pointwise bound
$\norm{h\sum_{j=0}^{k-1}b_j}\le h\sum_{j=0}^{k-1}\norm{b_j}\le t_k\norm{b}_\infty$, where
$t_k=kh\le1$; hence
\[
\mathcal M_k=\E\norm{h\sum_{j=0}^{k-1}b_j}^{2\kappa}\le\big(t_k\norm{b}_\infty\big)^{2\kappa}
\le\norm{b}_\infty^{2\kappa}.
\]
For the diffusion, $Z_k=\sqrt h\sum_{j<k}\sigma_j\xi_{j+1}$ is a discrete martingale; conditionally
on the past, each increment $\sqrt h\,\sigma_j\xi_{j+1}$ is Gaussian with covariance
$h\,\sigma_j\sigma_j^{\top}$, so the discrete Burkholder--Davis--Gundy inequality reduces to a
Gaussian moment bound (for $\kappa=1$ the first inequality below is an equality):
\[
\mathcal V_k=\E\norm{Z_k}^{2\kappa}\le C_\kappa\,\E\Big(h\sum_{j=0}^{k-1}\norm{\sigma_j}^2\Big)^{\kappa}
\le C_\kappa\big(t_k\norm{\sigma}_\infty^2\big)^{\kappa}\le C_\kappa\norm{\sigma}_\infty^{2\kappa}.
\]
Hence
\[
\E\norm{X^{(n)}_k}^{2\kappa}\le 3^{2\kappa-1}\Big(\norm{x_0}^{2\kappa}+\norm{b}_\infty^{2\kappa}
+C_\kappa\norm{\sigma}_\infty^{2\kappa}\Big)=:C_\kappa,
\]
uniformly in $k\le n$ and in $n$. This proves the second bound in Lemma~\ref{lem:moments}.

\emph{(iii) Increment estimates.} For $0\le s<t\le1$,
\[
X_t-X_s=\int_s^t b(u,X_u,\mu_u)\,\dd u+\int_s^t \sigma(u,X_u,\mu_u)\,\dd W_u.
\]
By H\"older's and the Burkholder--Davis--Gundy inequalities,
\[
\E\norm{X_t-X_s}^{2\kappa}
\le 2^{2\kappa-1}\Big((t-s)^{2\kappa}\norm{b}_\infty^{2\kappa}+C_\kappa\,(t-s)^{\kappa}\norm{\sigma}_\infty^{2\kappa}\Big).
\]
Since $t-s\le1$, we have $(t-s)^{2\kappa}\le(t-s)^{\kappa}$, so
\[
\E\norm{X_t-X_s}^{2\kappa}\le 2^{2\kappa-1}\big(\norm{b}_\infty^{2\kappa}+C_\kappa\norm{\sigma}_\infty^{2\kappa}\big)(t-s)^{\kappa}
=:C_\kappa\,(t-s)^{\kappa},
\]
which proves the increment estimate of Lemma~\ref{lem:moments}.

\emph{(iv) Increment estimate for the interpolated process.} For the interpolated Euler process
$\bar X^{(n)}$ of \eqref{eq:cont-euler} and $0\le s<t\le1$,
\[
\bar X^{(n)}_t-\bar X^{(n)}_s
=\int_s^t b\big(s_n(u),\bar X^{(n)}_{s_n(u)},\mu^{(n)}_{s_n(u)}\big)\dd u
+\int_s^t \sigma\big(s_n(u),\bar X^{(n)}_{s_n(u)},\mu^{(n)}_{s_n(u)}\big)\dd W_u,
\]
and the two coefficients are bounded by $\norm{b}_\infty$ and $\norm{\sigma}_\infty$, uniformly in
$u$ and in $n$ (Assumption~\ref{ass:spatial}). Repeating verbatim the argument of (iii) --- only the
boundedness of the coefficients is used --- therefore gives
\[
\E\norm{\bar X^{(n)}_t-\bar X^{(n)}_s}^{2\kappa}
\le 2^{2\kappa-1}\big(\norm{b}_\infty^{2\kappa}+C_\kappa\norm{\sigma}_\infty^{2\kappa}\big)(t-s)^{\kappa}
=:C_\kappa\,(t-s)^{\kappa},
\]
uniformly in $n$, which is the third estimate of Lemma~\ref{lem:moments}. For the main paper only
$\kappa=1$ is needed (the second moment and the second-moment increment).

\section{Proof of the Measure-Flow Convergence Rate}
\label{app:flowrate}

We prove Lemma~\ref{lem:flowrate}. Throughout, $C$ denotes a finite constant independent of $n$,
$t$ and $x_0$, whose value may change from line to line.

\textbf{Step 1: the interpolated process.} The process $\bar X^{(n)}$ defined in
\eqref{eq:cont-euler} coincides in law with the Euler scheme at the grid points (choose
$\xi_{k+1}:=(W_{t_{k+1}}-W_{t_k})/\sqrt{h}$, which is again a standard Gaussian vector, independent of the
past), so that $\mu^{(n)}_t=\Law(\bar X^{(n)}_t)$ and $\mu^{(n)}_{t_k}=\mu^{(n)}_k$.

\textbf{Step 2: the error SDE.} Set $\eta_t:=\bar X^{(n)}_t-X_t$. Subtracting \eqref{eq:SDE} from
\eqref{eq:cont-euler} gives
\[
\dd \eta_t=\Delta b_t\,\dd t+\Delta\sigma_t\,\dd W_t,
\]
with
\[
\Delta b_t=b\big(s_n(t),\bar X^{(n)}_{s_n(t)},\mu^{(n)}_{s_n(t)}\big)-b(t,X_t,\mu_t),\qquad
\Delta\sigma_t=\sigma\big(s_n(t),\bar X^{(n)}_{s_n(t)},\mu^{(n)}_{s_n(t)}\big)-\sigma(t,X_t,\mu_t),
\]
and $\eta_0=0$.

\textbf{Step 3: It\^o formula.} Applying It\^o's formula to $\norm{\eta_t}^2$, taking expectations (the
martingale term vanishes, since its integrand is square-integrable: $\norm{\Delta\sigma_s}$ is
bounded and $\E\norm{\eta_s}^2<\infty$, the latter following from Lemma~\ref{lem:moments} and the
frozen-increment estimate of Step~6(i) below), and using $2ab\le a^2+b^2$,
\[
\E\norm{\eta_t}^2\le\int_0^t\Big(\E\norm{\eta_s}^2+\E\norm{\Delta b_s}^2+\E\norm{\Delta\sigma_s}^2\Big)\,\dd s .
\]

\textbf{Step 4: five-piece decomposition of $\Delta b_s$.} Split
$\Delta b_s=b\big(s_n(s),\bar X^{(n)}_{s_n(s)},\mu^{(n)}_{s_n(s)}\big)-b(s,X_s,\mu_s)$ into five
pieces, each controlled by one Lipschitz/boundedness property ($\mathrm{Lip}$ denotes a common Lipschitz constant in
$x$, for $b$ by Assumption~\ref{ass:spatial} and for $\sigma=\Sigma^{1/2}$ by the square-root
transfer of Section~\ref{sec:setting}):
\[
\begin{aligned}
\text{(a)}\quad &
\norm{b\big(s_n(s),\bar X^{(n)}_{s_n(s)},\mu^{(n)}_{s_n(s)}\big)-b\big(s_n(s),\bar X^{(n)}_s,\mu^{(n)}_{s_n(s)}\big)}
\le \mathrm{Lip}\,\norm{\bar X^{(n)}_{s_n(s)}-\bar X^{(n)}_s};\\
\text{(b)}\quad &
\norm{b\big(s_n(s),\bar X^{(n)}_s,\mu^{(n)}_{s_n(s)}\big)-b\big(s_n(s),X_s,\mu^{(n)}_{s_n(s)}\big)}
\le \mathrm{Lip}\,\norm{\eta_s};\\
\text{(c)}\quad &
\norm{b\big(s_n(s),X_s,\mu^{(n)}_{s_n(s)}\big)-b\big(s_n(s),X_s,\mu_{s_n(s)}\big)}
\le \Lambda\,\mathcal{W}_2\big(\mu^{(n)}_{s_n(s)},\mu_{s_n(s)}\big);\\
\text{(d)}\quad &
\norm{b\big(s_n(s),X_s,\mu_{s_n(s)}\big)-b\big(s,X_s,\mu_{s_n(s)}\big)}
\le\norm{\partial_t b}_\infty\,h;\\
\text{(e)}\quad &
\norm{b\big(s,X_s,\mu_{s_n(s)}\big)-b\big(s,X_s,\mu_s\big)}
\le \Lambda\,\mathcal{W}_2\big(\mu_{s_n(s)},\mu_s\big),
\end{aligned}
\]
using the Lipschitz continuity of $b$ in $x$ (Assumption~\ref{ass:spatial}) and in $\mu$
(Assumption~\ref{ass:measLip}), and the boundedness of $\partial_t b$ (Assumption~\ref{ass:time}).
Squaring with $(a_1+\cdots+a_5)^2\le5(a_1^2+\cdots+a_5^2)$,
\[
\norm{\Delta b_s}^2\le C\Big(\norm{\bar X^{(n)}_{s_n(s)}-\bar X^{(n)}_s}^2+\norm{\eta_s}^2
+\mathcal{W}_2\big(\mu^{(n)}_{s_n(s)},\mu_{s_n(s)}\big)^2+h^2+\mathcal{W}_2\big(\mu_{s_n(s)},\mu_s\big)^2\Big).
\]
The same bound holds verbatim for $\Delta\sigma_s$, using the Lipschitz continuity of
$\sigma=\Sigma^{1/2}$ in $(x,\mu)$, by the square-root transfer of Assumptions~\ref{ass:spatial} and
\ref{ass:measLip} discussed in Section~\ref{sec:setting}.
Inserting into Step~3,
\[
\E\norm{\eta_t}^2\le C\int_0^t\Big(\E\norm{\eta_s}^2+\E\norm{\bar X^{(n)}_{s_n(s)}-\bar X^{(n)}_s}^2
+\mathcal{W}_2\big(\mu^{(n)}_{s_n(s)},\mu_{s_n(s)}\big)^2+\mathcal{W}_2\big(\mu_{s_n(s)},\mu_s\big)^2+h^2\Big)\,\dd s.
\]

\textbf{Step 5: absorb the measure-flow difference via coupling.} Let
$\Psi_t:=\sup_{r\le t}\E\norm{\eta_r}^2$. By the coupling definition of $\mathcal{W}_2$ (with the natural coupling
$(\bar X^{(n)}_{s_n(s)},X_{s_n(s)})$),
\[
\mathcal{W}_2\big(\mu^{(n)}_{s_n(s)},\mu_{s_n(s)}\big)^2\le\E\norm{\bar X^{(n)}_{s_n(s)}-X_{s_n(s)}}^2
=\E\norm{\eta_{s_n(s)}}^2\le \Psi_{s_n(s)}\le \Psi_s,
\]
so
\[
\Psi_t\le C\int_0^t\Big(\Psi_s+\E\norm{\bar X^{(n)}_{s_n(s)}-\bar X^{(n)}_s}^2
+\mathcal{W}_2\big(\mu_{s_n(s)},\mu_s\big)^2+h^2\Big)\,\dd s.
\]

\textbf{Step 6: two frozen estimates.} \emph{(i) Frozen increment.} On $[s_n(s),s]$ the coefficients
are frozen at $s_n(s)$ (since $s-s_n(s)\le h$), and by boundedness of the coefficients the drift
term has second moment $O(h^2)$ and the diffusion term $O(h)$, hence
\[
\E\norm{\bar X^{(n)}_s-\bar X^{(n)}_{s_n(s)}}^2\le2\norm{b}_\infty^2 h^2+2\norm{\sigma}_\infty^2 h\le C h.
\]
\emph{(ii) Time-H\"older regularity.} By the coupling definition of $\mathcal{W}_2$ and the increment estimate
of Lemma~\ref{lem:moments} (with $\kappa=1$),
\[
\mathcal{W}_2\big(\mu_{s_n(s)},\mu_s\big)^2\le\E\norm{X_{s_n(s)}-X_s}^2\le C\,\big(s-s_n(s)\big)\le Ch.
\]
Inserting both into Step~5 and using $t\le1$, $h^2\le h$,
\[
\Psi_t\le C\int_0^t\big(\Psi_s+Ch+Ch+h^2\big)\,\dd s\le Ch+C\int_0^t \Psi_s\,\dd s.
\]

\textbf{Step 7: Gronwall and conclusion.} By Gronwall's lemma,
\[
\Psi_t\le Ch\,e^{Ct}\le Ch\,e^{C},\qquad t\in[0,1],
\]
so $\Psi_1=O(h)=O(n^{-1})$. Finally, again by the coupling definition of $\mathcal{W}_2$,
\[
\mathcal{W}_2\big(\mu^{(n)}_t,\mu_t\big)^2\le\E\norm{\bar X^{(n)}_t-X_t}^2=\E\norm{\eta_t}^2\le \Psi_t\le \Psi_1=O(h),
\]
and therefore $\mathcal{W}_2(\mu^{(n)}_t,\mu_t)=O(h^{1/2})=O(n^{-1/2})$, uniformly in $t\in[0,1]$. This proves
Lemma~\ref{lem:flowrate}. The measure-flow term does not degrade the rate: it is absorbed into
$\Psi_s$ through the coupling in Step~5, so the rate coincides with the classical strong-error rate of
the Euler scheme.

\section{Proof of the Gaussian-Kernel Estimates}
\label{app:kernels}

We prove Lemma~\ref{lem:kernel}.

\emph{(i) Gaussian convolution.} From the definition of the Gaussian density $g$,
$\varphi_{c,\rho}(u)=\rho^{-d}\exp\big(-\tfrac{\norm{u}^2}{c\rho^2}\big)
=(\pi c)^{d/2}\,g\big(0,\tfrac{c\rho^2}{2}I;u\big)$.
Since the convolution of two centred Gaussians with covariances $A$ and $B$ is the centred
Gaussian with covariance $A+B$,
\begin{multline}
\big(\varphi_{c,\rho}*\varphi_{c',\rho'}\big)(u)
=(\pi c)^{d/2}(\pi c')^{d/2}\,g\big(0,\tfrac{c\rho^2+c'\rho'^2}{2}I;u\big)\\
=(\pi c)^{d/2}(\pi c')^{d/2}\big(\pi(c\rho^2+c'\rho'^2)\big)^{-d/2}
\exp\Big(-\frac{\norm{u}^2}{c\rho^2+c'\rho'^2}\Big).
\end{multline}
Since $c\rho^2+c'\rho'^2\ge\min(c,c')(\rho^2+\rho'^2)$ and
$(cc')^{d/2}=\max(c,c')^{d/2}\min(c,c')^{d/2}$, the prefactor satisfies
\[
(\pi c)^{d/2}(\pi c')^{d/2}\big(\pi(c\rho^2+c'\rho'^2)\big)^{-d/2}
\le\pi^{d/2}\max(c,c')^{d/2}(\rho^2+\rho'^2)^{-d/2},
\]
while $c\rho^2+c'\rho'^2\le(c+c')(\rho^2+\rho'^2)$ gives
$\exp\big(-\tfrac{\norm{u}^2}{c\rho^2+c'\rho'^2}\big)
\le\exp\big(-\tfrac{\norm{u}^2}{(c+c')(\rho^2+\rho'^2)}\big)$. Combining the two bounds,
\[
\big(\varphi_{c,\rho}*\varphi_{c',\rho'}\big)(u)
\le\pi^{d/2}\max(c,c')^{d/2}\,\varphi_{c+c',\sqrt{\rho^2+\rho'^2}}(u),
\]
which is item~(i) with $u=y-z$ (the integral in (i) is exactly this convolution evaluated at
$y-z$).

\emph{(ii) The factor $\norm{y-x}$.} Put $u=y-x$ and $\omega=\norm{u}/\rho$. Then
\begin{multline}
\norm{u}\,\varphi_{c,\rho}(u)
=\rho^{-d}\rho \omega\exp\Big(-\frac{\omega^2}{c}\Big)\\
=\rho\,\varphi_{c+1,\rho}(u)\cdot \omega\exp\Big(-\frac{\omega^2}{c}+\frac{\omega^2}{c+1}\Big)
=\rho\,\varphi_{c+1,\rho}(u)\cdot \omega\exp\Big(-\frac{\omega^2}{c(c+1)}\Big),
\end{multline}
using $\tfrac1c-\tfrac1{c+1}=\tfrac1{c(c+1)}$. The function $\omega\mapsto \omega\exp\big(-\tfrac{\omega^2}{c(c+1)}\big)$
is bounded on $[0,\infty)$, its maximum being $\sqrt{c(c+1)/(2e)}$ (attained at
$\omega=\sqrt{c(c+1)/2}$). Hence $\norm{y-x}\,\varphi_{c,\rho}(y-x)\le C\rho\,\varphi_{c+1,\rho}(y-x)$
with $C=\sqrt{c(c+1)/(2e)}$.

\emph{(iii) The Beta integral.} Substitute $u=s+(t-s)v$; then $t-u=(t-s)(1-v)$, $u-s=(t-s)v$ and
$\dd u=(t-s)\dd v$, so
\[
\int_s^t(t-u)^{\alpha/2}(u-s)^{\beta/2}\dd u
=(t-s)^{(\alpha+\beta)/2+1}\int_0^1(1-v)^{\alpha/2}v^{\beta/2}\dd v
=(t-s)^{(\alpha+\beta)/2+1}\,\mathrm{B}\big(\tfrac\alpha2+1,\tfrac\beta2+1\big),
\]
by $\mathrm{B}(a,b)=\int_0^1(1-v)^{a-1}v^{b-1}\dd v$ with $a=\tfrac\alpha2+1$, $b=\tfrac\beta2+1$.

\emph{(iv) Gaussian derivative control.} Write $g=(2\pi)^{-d/2}(\det\Theta)^{-1/2}e^{-Q}$ with
$Q=\tfrac12(z-m)^{\top}\Theta^{-1}(z-m)$. Since $c\rho^2 I\le\Theta\le C_0\rho^2 I$, we have
$\norm{\Theta^{-1}}\le C\rho^{-2}$, $\det(\Theta^{-1/2})\le(c\rho^2)^{-d/2}$, and
$(z-m)^{\top}\Theta^{-1}(z-m)\ge\norm{z-m}^2/(C_0\rho^2)$; hence
\[
g(m,\Theta;z)\le C\,\varphi_{c,\rho}(z-m),
\]
possibly enlarging $c$. Differentiating, $\partial^\alpha_z g$ is a finite linear combination of
monomials $\prod_a\big(\Theta^{-1}(z-m)\big)_{i_a}\prod_b(z-m)_{j_b}\,g$: each derivative applied
to the exponential contributes one factor $[\Theta^{-1}(z-m)]_{i_r}$, whereas each derivative
applied to an already-present $(z-m)$ component removes one such component. Hence in every monomial
the number $a$ of $\Theta^{-1}$ entries and the number $b$ of $(z-m)$ components satisfy
$b\le a\le|\alpha|$ and $-2a+b=-|\alpha|$. Using
$\norm{z-m}^b g\le C_b\rho^b\varphi_{c',\rho}(z-m)$ (item~(ii) iterated), each monomial is bounded
by $C\rho^{-2a}\cdot C\rho^{b}\cdot\varphi_{c',\rho}(z-m)=C\rho^{-|\alpha|}\varphi_{c',\rho}(z-m)$.
Summing over the finitely many monomials gives item~(iv).

\emph{(v) Gaussian comparison.} We first record the $\Theta$-derivative bound, valid whenever
$\Theta\ge c\rho^2 I$:
\[
\abs{\partial^\alpha_z\partial_{\Theta_{ij}}g(m,\Theta;z)}\le C_\alpha\,\rho^{-|\alpha|-2}\,\varphi_{c,\rho}(z-m).
\]
Indeed, $\partial_{\Theta_{ij}}g=g\,\tilde Q$ with
$\tilde Q=-\tfrac12(\Theta^{-1})_{ij}+\tfrac12[\Theta^{-1}(z-m)]_i[\Theta^{-1}(z-m)]_j$. Put $\zeta=z-m$ and
$P=[\Theta^{-1}\zeta]_i[\Theta^{-1}\zeta]_j=\zeta^{\top}\Theta^{-1}e_ie_j^{\top}\Theta^{-1}\zeta$, a quadratic form
in $\zeta$ whose coefficient matrix has operator norm $O(\rho^{-4})$ (since $\norm{\Theta^{-1}}\le C\rho^{-2}$).
Then
$\partial^\alpha_z\partial_{\Theta_{ij}}g=-\tfrac12(\Theta^{-1})_{ij}\partial^\alpha_z g+\tfrac12\partial^\alpha_z(Pg)$,
and by the Leibniz rule
$\partial^\alpha_z(Pg)=\sum_{\gamma\le\alpha}\binom{\alpha}{\gamma}(\partial^\gamma_z P)(\partial^{\alpha-\gamma}_z g)$,
where $\partial^\gamma_z P=0$ for $|\gamma|>2$ and
$|\partial^\gamma_z P|\le C\rho^{-4}\norm{\zeta}^{2-|\gamma|}$ for $|\gamma|\le2$. Using item~(iv) and
item~(ii) iterated,
$|\partial^{\alpha-\gamma}_z g|\le C\rho^{-(|\alpha|-|\gamma|)}\varphi_{c,\rho}(\zeta)$ and
$\norm{\zeta}^{2-|\gamma|}\varphi_{c,\rho}(\zeta)\le C\rho^{2-|\gamma|}\varphi_{c',\rho}(\zeta)$, so each term
is bounded by
$C\rho^{-4}\rho^{2-|\gamma|}\rho^{-|\alpha|+|\gamma|}\varphi_{c',\rho}(\zeta)=C\rho^{-|\alpha|-2}\varphi_{c',\rho}(\zeta)$,
while
$|(\Theta^{-1})_{ij}\partial^\alpha_z g|\le C\rho^{-2}\cdot C\rho^{-|\alpha|}\varphi_{c,\rho}(\zeta)=C\rho^{-|\alpha|-2}\varphi_{c,\rho}(\zeta)$;
summing gives the $\Theta$-derivative bound. Also $\partial_{m_i}g=-\partial_{z_i}g$, so that
$|\partial^\alpha_z\partial_{m_i}g|\le C\rho^{-|\alpha|-1}\varphi_{c,\rho}(z-m)$ by item~(iv).

Now differentiate along the segment $(m_\upsilon,\Theta_\upsilon)=(1-\upsilon)(m,\Theta)+\upsilon(m',\Theta')$,
so that $\Theta_\upsilon\ge c\rho^2 I$ and $\norm{m_\upsilon}\le C\rho^2$. By the mean value theorem,
\[
\begin{aligned}
\abs{\partial^\alpha_z g(m,\Theta;z)-\partial^\alpha_z g(m',\Theta';z)}
&\le\norm{m-m'}\sup_\upsilon\norm{\nabla_m\partial^\alpha_z g(m_\upsilon,\Theta_\upsilon;z)}\\
&\quad+\norm{\Theta-\Theta'}\sup_\upsilon\norm{\nabla_\Theta\partial^\alpha_z g(m_\upsilon,\Theta_\upsilon;z)}\\
&\le C_\alpha\Big[\norm{m-m'}\,\rho^{-|\alpha|-1}+\norm{\Theta-\Theta'}\,\rho^{-|\alpha|-2}\Big]\\
&\qquad\sup_\upsilon\varphi_{c,\rho}(z-m_\upsilon),
\end{aligned}
\]
using item~(iv) and the $\Theta$-derivative bound. Since $\norm{m_\upsilon}\le C\rho^2$, we have
$\norm{z-m_\upsilon}^2\ge\tfrac12\norm{z}^2-C\rho^4$, hence (enlarging $c$)
$\varphi_{c,\rho}(z-m_\upsilon)\le C\varphi_{c,\rho}(z)$ for every $\upsilon$. Therefore
\[
\abs{\partial^\alpha_z g(m,\Theta;z)-\partial^\alpha_z g(m',\Theta';z)}
\le C_\alpha\Big[\norm{m-m'}\,\rho^{-1}+\norm{\Theta-\Theta'}\,\rho^{-2}\Big]
\rho^{-|\alpha|}\,\varphi_{c,\rho}(z),
\]
which is item~(v).

\bibliographystyle{unsrtnat}
\bibliography{references}

\begin{thebibliography}{16}
\providecommand{\natexlab}[1]{#1}
\providecommand{\url}[1]{\texttt{#1}}
\expandafter\ifx\csname urlstyle\endcsname\relax
  \providecommand{\doi}[1]{doi: #1}\else
  \providecommand{\doi}{doi: \begingroup \urlstyle{rm}\Url}\fi

\bibitem[Kac(1956)]{Kac1956}
M.~Kac.
\newblock Foundations of kinetic theory.
\newblock In \emph{Proceedings of the Third Berkeley Symposium on Mathematical
  Statistics and Probability, Vol. III}, pages 171--197, Berkeley, 1956.
  University of California Press.

\bibitem[Baladron et~al.(2012)Baladron, Fasoli, Faugeras, and
  Touboul]{Baladron2012}
Javier Baladron, Diego Fasoli, Olivier Faugeras, and Jonathan Touboul.
\newblock Mean-field description and propagation of chaos in networks of
  {H}odgkin--{H}uxley and {F}itz{H}ugh--{N}agumo neurons.
\newblock \emph{The Journal of Mathematical Neuroscience}, 2\penalty0
  (1):\penalty0 10, 2012.
\newblock \doi{10.1186/2190-8567-2-10}.

\bibitem[Lasry and Lions(2007)]{LasryLions2007}
Jean-Michel Lasry and Pierre-Louis Lions.
\newblock Mean field games.
\newblock \emph{Japanese Journal of Mathematics}, 2\penalty0 (1):\penalty0
  229--260, 2007.

\bibitem[Huang et~al.(2006)Huang, Malham\'e, and
  Caines]{HuangMalhameCaines2006}
Minyi Huang, Roland~P. Malham\'e, and Peter~E. Caines.
\newblock Large population stochastic dynamic games: closed-loop
  {M}c{K}ean--{V}lasov systems and the {N}ash certainty equivalence principle.
\newblock \emph{Communications in Information and Systems}, 6\penalty0
  (3):\penalty0 221--252, 2006.

\bibitem[McKean(1966)]{McKean1966}
Henry~P. McKean.
\newblock A class of {M}arkov processes associated with nonlinear parabolic
  equations.
\newblock \emph{Proceedings of the National Academy of Sciences of the United
  States of America}, 56:\penalty0 1907--1911, 1966.

\bibitem[Sznitman(1991)]{Sznitman1991}
Alain-Sol Sznitman.
\newblock Topics in propagation of chaos.
\newblock In \emph{\'Ecole d'\'Et\'e de Probabilit\'es de Saint-Flour
  XIX--1989}, volume 1464 of \emph{Lecture Notes in Mathematics}, pages
  165--251. Springer, 1991.

\bibitem[Huang and Wang(2019)]{HuangWang2019}
Xing Huang and Feng-Yu Wang.
\newblock Distribution dependent {SDE}s with singular coefficients.
\newblock \emph{Stochastic Processes and their Applications}, 129:\penalty0
  4747--4770, 2019.

\bibitem[Ding and Qiao(2021)]{DingQiao2021}
Xiaojie Ding and Huijie Qiao.
\newblock {E}uler--{M}aruyama approximations for stochastic
  {M}c{K}ean--{V}lasov equations with non-{L}ipschitz coefficients.
\newblock \emph{Journal of Theoretical Probability}, 34:\penalty0 1408--1425,
  2021.

\bibitem[Aronson(1967)]{Aronson1967}
D.~G. Aronson.
\newblock Bounds for the fundamental solution of a parabolic equation.
\newblock \emph{Bulletin of the American Mathematical Society}, 73\penalty0
  (6):\penalty0 890--896, 1967.

\bibitem[Friedman(1964)]{Friedman1964}
Avner Friedman.
\newblock \emph{Partial Differential Equations of Parabolic Type}.
\newblock Prentice-Hall, Englewood Cliffs, NJ, 1964.

\bibitem[H{\"o}rmander(1967)]{Hormander1967}
Lars H{\"o}rmander.
\newblock Hypoelliptic second order differential equations.
\newblock \emph{Acta Mathematica}, 119:\penalty0 147--171, 1967.

\bibitem[Bally and Talay(1996{\natexlab{a}})]{BallyTalay1996a}
Vlad Bally and Denis Talay.
\newblock The law of the {E}uler scheme for stochastic differential equations
  {I}: convergence rate of the distribution function.
\newblock \emph{Probability Theory and Related Fields}, 104:\penalty0 43--60,
  1996{\natexlab{a}}.

\bibitem[Bally and Talay(1996{\natexlab{b}})]{BallyTalay1996b}
Vlad Bally and Denis Talay.
\newblock The law of the {E}uler scheme for stochastic differential equations
  {II}: convergence rate of the density.
\newblock \emph{Monte Carlo Methods and Applications}, 2:\penalty0 93--128,
  1996{\natexlab{b}}.

\bibitem[Konakov and Mammen(2000)]{KonakovMammen2000}
Valentin Konakov and Enno Mammen.
\newblock Local limit theorems for transition densities of {M}arkov chains
  converging to diffusions.
\newblock \emph{Probability Theory and Related Fields}, 117:\penalty0 551--587,
  2000.

\bibitem[Crisan and McMurray(2018)]{CrisanMcMurray2018}
Dan Crisan and Eamon McMurray.
\newblock Smoothing properties of {M}c{K}ean--{V}lasov {SDE}s.
\newblock \emph{Probability Theory and Related Fields}, 171:\penalty0 97--148,
  2018.

\bibitem[McKean and Singer(1967)]{McKeanSinger1967}
Henry~P. McKean and Isadore~M. Singer.
\newblock Curvature and the eigenvalues of the {L}aplacian.
\newblock \emph{Journal of Differential Geometry}, 1:\penalty0 43--69, 1967.

\end{thebibliography}

\end{document}